\documentclass{article}
\usepackage[latin1]{inputenc}
\usepackage[english]{babel}
\usepackage[dvips]{graphics}
\usepackage{epsfig}
\usepackage{fullpage}
\usepackage{amsfonts}
\usepackage{multicol}
\usepackage{latexsym,amssymb}     
\usepackage{amsmath}
\usepackage{indentfirst}
\usepackage{makeidx}
\usepackage{theorem}
\usepackage{bbm}
\usepackage{mathrsfs}
\usepackage{psfrag}
\usepackage{enumitem}
\usepackage{epsfig} 
\usepackage{tikz}
\usepackage{multirow}

\newcommand{\R}{\mathbb{R}}
\newcommand{\Q}{\mathbb{Q}}

\newcommand{\D}{\mathbb{D}}

\newcommand{\E}{\mathbb{E}}

\newcommand{\N}{\mathbb{N}}

\newcommand{\Px}{\mathbb{P}}

\newcommand{\xcn}{\hat{X}^N}

\newcommand{\ivec}[3]{#1_{\lbrace #2 ,#3 \rbrace}}
\newcommand{\ipartial}[2]{\partial_{\lbrace #1 ,#2 \rbrace}}
\newcommand{\interv}[2]{\lbrace #1,\ldots,#2\rbrace}

\newcommand{\iunderset}[2]{\underset{#1}{\underbrace{#2}}}

\newcommand{\set}[1]{\{#1\}}
\newcommand{\indi}[1]{\mathbbm{1}_{#1}}

\newcommand{\blanc}[1]{\vspace{#1\baselineskip}}

\newtheorem{nt_theorem}{Theorem}
\newenvironment{theorem}{\blanc{0.5}\begin{nt_theorem}---}{\end{nt_theorem}\blanc{0.5}}

\newtheorem{nt_proposition}[nt_theorem]{Proposition}
\newenvironment{proposition}{\blanc{0.5}\begin{nt_proposition}---}{\end{nt_proposition}\blanc{0.5}}

\newtheorem{nt_corollaire}[nt_theorem]{Corollary}
\newenvironment{Corollary}{\blanc{0.5}\begin{nt_corollaire}---}{\end{nt_corollaire}\blanc{0.5}}

\newtheorem{nt_definition}[nt_theorem]{Definition}

\newtheorem{nt_lemma}[nt_theorem]{Lemma}
\newenvironment{lemma}{\blanc{0.5}\begin{nt_lemma}---}{\end{nt_lemma}\blanc{0.5}}

\newtheorem{nt_remark}[nt_theorem]{Remark}
\newenvironment{remark}{\blanc{0.2}\begin{nt_remark}---}{\end{nt_remark}\blanc{0.2}}

\newtheorem{nt_exemple}[nt_theorem]{Example}

\newenvironment{proof}{{\textit{Proof : }}}{\hfill$\Box$ \\}

\makeatletter
\newcounter{hypo}
\newcommand*{\symm}{{\mathcal{S}_d(\mathbb R)}}

\newcommand*{\genm}{{\mathcal{M}_d(\mathbb R)}}
\newcommand*{\dpos}{{\mathcal{S}_d^{+,*}(\mathbb R)}}
\newcommand*{\posm}{{\mathcal{S}_d^+(\mathbb R)}}

\newcommand{\Subm}[2]{{#1}^{[#2]}}
\newcommand{\opr}{\mathcal{L}}

\newcommand{\dposs}[1]{\mathcal{S}_{#1}^{+,*}(\mathbb R)}

\newcommand{\corr}{{\mathfrak{C}_{d}(\mathbb R)}}
\newcommand{\corri}{{\mathfrak{C}_{d}^*(\mathbb R)}}

\newcommand{\equi}{\Leftrightarrow}
\newcommand{\corrps}[5]{{{MRC}_{#1}(#2,#3,#4,#5)}}
\newcommand{\corrp}{MRC_{d}(x,\kappa,c,a)}
\newcommand{\corrpst}[5]{{MRC}_{#1}(#2,#3,#4,#5;t)}
\newcommand{\corrpstb}[6]{{MRC}_{#1}(#2,#3,#4,#5;#6)}
\newcommand{\corrpt}{{MRC}_{d}(x,\kappa,c,a;t)}

\newcommand{\kp}{\kappa}

\newcommand*{\dohypo}{\textbf{(${\mathcal H}$\thehypo)}}
\newcommand*{\Tr}{\textup{Tr}} 
\newcommand*{\Rg}{\textup{Rk}} 
\newcommand*{\adj}{\textup{adj}} 

\def\hypref#1{\hyperref[hyp:#1]{(${\mathcal H}$\ref*{hyp:#1})}}
\def\hypreff#1#2{\hyperref[hyp:#2]{(${\mathcal H}$\ref*{hyp:#1}-{\it
\ref*{hyp:#2})}}}

\makeatother

\begin{document}
%
%
\vspace{5mm}
\title{A Mean-Reverting SDE on Correlation Matrices }
\author{Abdelkoddousse Ahdida and Aur\'elien Alfonsi\footnote{Universit\'e Paris-Est, CERMICS,
  Project team MathFi ENPC-INRIA-UMLV,
   Ecole des Ponts, 6-8 avenue Blaise Pascal,
  77455 Marne La Vall\'ee, France. {\tt \{ahdidaa,alfonsi\}@cermics.enpc.fr
}}} 
\maketitle
\begin{abstract} 
We introduce a mean-reverting SDE whose solution is naturally
defined on the space of correlation matrices. This SDE can be seen as an
extension of the well-known Wright-Fisher diffusion. We provide conditions
that ensure weak and strong uniqueness of the SDE, and describe its ergodic
limit. We also shed light on a useful connection with Wishart processes that makes
understand how we get the full SDE. Then, we focus on the simulation of this
diffusion and present discretization schemes that achieve a second-order weak
convergence. Last, we explain how these correlation processes could be used to
model the dependence between financial assets.
\end{abstract}


{ {\it AMS Class 2010:}  65C30,  60H35, 91B70.
}\\

{\bf Acknowledgements:} We would like to acknowledge the support of the
Credinext project from Finance Innovation, the ``Chaire Risques Financiers''
of Fondation du Risque and the French National Research Agency (ANR) program BigMC.

\section*{Introduction}\label{Chapter1}

The scope of this paper is to introduce an SDE that is well defined on the set
of correlation matrices. Our main motivation comes from an application to finance, where the
correlation is commonly used to describe the dependence between assets. More
precisely, a diffusion on correlation matrices can be used to model the
instantaneous correlation between the log-prices of different stocks. 
Thus, it is also very important for practical purpose to be able to sample
paths of this SDE in order to compute expectations (for prices or greeks). This is why
an entire part of this paper is devoted to get  an efficient simulation
scheme. More generally, processes on correlation
matrices can naturally be used to model the dynamics of the dependence between
some quantities and can be applied to a much wider range of applications. 
In this paper, we mainly focus on the
definition, the mathematical properties and the sampling of this SDE. However,
we discuss in Section~\ref{sec_fin} a possible implementation of these processes in
finance to model a basket of risky assets.

There are works on particular Stochastic Differential Equations that are
defined on positive semidefinite matrices such as Wishart processes
(Bru~\cite{Bru}) or their Affine extensions (Cuchiero et
al.~\cite{Teichmann}). On the contrary,  there is to the best of our knowledge
very few literature dedicated to some 
stochastic differential equations that are valued on correlation matrices. Of
course, general results are known for stochastic differential equations on
manifolds. However, no particular SDE defined on correlation matrices has
been studied in detail. In dimension~$d=2$, correlation matrices
are naturally described by a single real~$\rho \in [-1,1]$. The probably most
famous SDE on~$[-1,1]$ is the following Wright-Fisher diffusion:
\begin{equation}\label{WF_1D}dX_t = \kp (\bar{\rho}-X_t)dt + \sigma \sqrt{1-X_t^2}dB_t,\end{equation}
where $\kp \ge 0,\ \bar{\rho} \in [-1,1],\ \sigma \ge 0$, and $(B_t)_{t \geq
  0}$ is a real Brownian motion. Here, we make a slight abuse of language.
Strictly speaking, Wright-Fisher diffusions are defined on~$[0,1]$ and this is
in fact the process $(\frac{1+X_t}{2},t\ge 0)$ that is a Wright-Fisher
one. They have originally been used to model gene frequencies~(see Karlin and Taylor~\cite{KT}). The
marginal law of $X_t$ is known explicitly with its moments, and its density
can be written as an expansion with respect to the Jacobi orthogonal
polynomial basis (see Mazet~\cite{Mazet}). This is why the process $(X_t,t\ge 0)$ is
sometimes also called Jacobi process in the literature. In higher dimension
($d\ge 3$), no similar SDE has been yet considered. To get processes on
correlation matrices, it is instead used parametrization 
of subsets of correlation matrices. For example, one can consider~$X_t$ defined by $(X_t)_{i,j}=\rho_t$ for $1\le i\not = j
\le d$,
where $\rho_t$ is a Wright-Fisher diffusion on $[-1/(d-1),1]$. More
sophisticated examples can be found in~\cite{Kaya}. The main purpose of this
paper is to propose a natural extension of the Wright-Fisher
process~\eqref{WF_1D} that is defined on the whole set of correlation
matrices.

Let us now introduce the process. We first advise the reader to have a look at
our notations for matrices located at the end if this introduction, even though they
are rather standard. We consider $(W_t,t \ge 0)$, a $d$-by-$d$ square matrix
process whose elements are independent real standard Brownian motions, and
focus on the following SDE on the correlation matrices~$\corr$:
\begin{equation}\label{SDE_CORR}
X_t = x + \int_0^t \left( \kp (c-X_s) + (c-X_s) \kp \right) ds+ \sum_{n=1}^da_n \int_0^t \left( \sqrt{X_s-X_s e_d^{n} X_s} dW_se_d^{n}
  +e_d^{n} dW_s^T \sqrt{X_s-X_s e_d^{n} X_s} \right),
\end{equation}
where $x, c \in \corr$ and  $\kp=diag(\kp_1,\dots,\kp_d)$
and $a=diag(a_1,\dots,a_d)$ are nonnegative diagonal matrices such that
\begin{equation} \label{cond_weak_existence}
  \kp c+c \kp -(d-2) a^2 \in \posm \text{ or } d=2.
\end{equation}
Under these assumptions, we will show in
Section~\ref{sec_exist} that this SDE has a unique weak solution which is well-defined
on correlation matrices, i.e. $\forall t \ge 0, X_t \in \corr$. We
will also show that strong uniqueness holds if we assume moreover that $x\in
\corri$ and 
\begin{equation} \label{cond_strong_existence}
 \kp c+c \kp -d a^2 \in \posm.
\end{equation}
Looking at the diagonal coefficients, conditions~\eqref{cond_weak_existence}
and~\eqref{cond_strong_existence} imply respectively $\kp_i\ge (d-2)a_i^2/2$
and $\kp_i\ge da_i^2/2$. This heuristically means that the speed of the
mean-reversion has to be high enough with respect to the noise in order to
stay in~$\corr$. Throughout the paper, we will denote $\corrp$ the law of the process $(X_t)_{t
  \ge 0}$ and $ \corrpt$ the law of $X_t$. Here, $MRC$ stands for Mean-Reverting Correlation process.
 When using these notations, we implicitly assume  that~\eqref{cond_weak_existence} holds.

In dimension $d=2$, the only non trivial component is~$(X_t)_{1,2}$. We can
show easily that there is a real Brownian motion~$(B_t,t\ge 0)$ such that
$$d(X_t)_{1,2}=(\kp_1+\kp_2)(c_{1,2}-(X_t)_{1,2})dt+\sqrt{a_1^2+
  a_2^2}\sqrt{1-(X_t)_{1,2}^2}dB_t. $$
Thus, the process~\eqref{SDE_CORR} is simply a Wright-Fisher diffusion.
Our parametrization is however redundant in dimension~$2$, and we can assume without
loss of generality that $\kp_1=\kp_2$ and
$a_1=a_2$. Then, the condition $\kp c+c \kp \in \posm$  is always
satisfied, while assumption~\eqref{cond_strong_existence} is the condition that
ensures~$\forall t \ge 0, (X_t)_{1,2} \in (-1,1)$. In larger dimensions $d\ge
3$, we can also show that each non-diagonal element of~\eqref{SDE_CORR} follows a Wright-Fisher
diffusion~\eqref{WF_1D}.

The paper is structured as follows. In the first section, we present first
properties of Mean-Reverting Correlation processes. We
calculate the infinitesimal generator and give explicitly their
moments. In particular, this enables us to describe the ergodic limit. We also
present a connection with Wishart processes that clarifies how we get the
SDE~\eqref{SDE_CORR}. It is also useful later in the paper to construct
discretization schemes. Last, we show a link between some MRC processes and
the multi-allele Wright-Fisher model. Then, Section~\ref{sec_exist} is devoted to
the study of the weak existence and strong uniqueness of the
SDE~\eqref{SDE_CORR}. We discuss the extension of these results to time and space
dependent coefficients~$\kp, \ c, \ a$. Also, we exhibit a change of
probability that preserves the family of MRC processes. The third
section is devoted to obtain discretization schemes for~\eqref{SDE_CORR}. This
is a crucial issue if one wants to use MRC processes effectively. To do so, we
use a remarkable splitting of the infinitesimal generator as well as standard
composition technique. Thus, we construct discretization schemes with a weak error of
order~$2$. This can be done either by reusing the second order schemes for
Wishart processes obtained in~\cite{AA} or by an ad-hoc splitting (see
Appendix~\ref{schema_original}). All these schemes are tested numerically and
compared with a (corrected) Euler-Maruyama scheme. In the last section, we
explain how Mean-Reverting Correlation processes and possible extensions
could be used for the modeling of $d$~risky assets. First, we recall the main
advantages of modeling the instantaneous correlation instead of the
instantaneous covariance of the assets. Then, we discuss our model and its
relevance on index option market data.

\subsection*{Notations for real matrices :}\label{Notations_matrices}
\begin{itemize}
\item  For $d\in \N^*$, $\mathcal{M}_{d}(\mathbb R)$ denotes the real $d$
  square matrices;  $\mathcal{S}_{d}(\mathbb R)$,
  $\mathcal{S}_d^+(\mathbb R),\mathcal{S}_d^{+,*}(\mathbb R) $, and
  $\mathcal{G}_d(\mathbb R)$ denote respectively the set of 
  symmetric, symmetric positive semidefinite, symmetric positive definite and
  non singular matrices.

\item The set of correlation matrices is denoted by~$\corr$: 
\begin{equation*}
\corr = \left \lbrace x \in \posm, \forall 1\leq i\leq d,\,\, x_{i,i} = 1 \right \rbrace
\end{equation*}
We also define $\corri=\corr \cap \mathcal{G}_d(\mathbb R)$, the set of the invertible correlation matrices.
\item For $x \in\mathcal{M}_{d}(\mathbb R)$,
  $x^T$, $\adj(x)$, $\det(x)$,  $\Tr(x)$ and $\Rg(x)$ are respectively the
  transpose, the adjugate, the determinant, the trace and the rank of $x$.
\item For $x\in\mathcal{S}_d^+(\mathbb R)$, $\sqrt{x}$ denotes the unique
  symmetric positive semidefinite matrix such that $(\sqrt{x})^2=x$
\item The identity matrix is denoted by $I_d$. We set for   $1\le i,j\le d$,
  $e^{i,j}_d=(\indi{k=i,l=j})_{1\le k,l \le d}$ and $e^i_d=e^{i,i}_d$. Last,
  we define $e^{\{i,j\}}_d=e^{i,j}_d+\indi{i \not = j}e^{j,i}_d$.
\item For $x \in \symm$, we denote by $x_{\set{i,j}}$ the value of $x_{i,j}$,
  so that $x=\sum_{1 \le i\le j\le d} x_{\set{i,j}}e^{\{i,j\}}_d $. We use both notations in the paper:
notation  $(x_{i,j})_{1 \le i,j \le d}$ is of course more convenient for matrix
calculations while $(x_{\set{i,j}})_{1 \le i \le j \le d}$ is preferred to
emphasize that we work on symmetric matrices and that we have
$x_{i,j}=x_{j,i}$.
\item For $\lambda_1,\dots,\lambda_d \in \R$,
  $diag(\lambda_1,\dots,\lambda_d)\in\symm $ denotes the diagonal matrix such
  that $diag(\lambda_1,\dots,\lambda_d)_{i,i}=\lambda_i$.
\item For $x \in \posm$ such that $x_{i,i}>0$ for all $1\le i \le d$, we
  define ${\bf p}(x)\in \corr$ by
  \begin{equation}\label{proj_corr}
    ({\bf p}(x))_{i,j}=\frac{x_{i,j}}{\sqrt{x_{i,i}x_{j,j}}}, \ 1\le i,j \le d.
  \end{equation}
\item For $x\in \symm$ and $ 1 \le i \le d$, we denote by $\Subm{x}{i}\in
  \mathcal{S}_{d-1}(\R)$ the matrix defined by
  $\Subm{x}{i}_{k,l}=x_{k+\indi{k\ge i} ,l+\indi{l\ge i} }$ and $x^i \in
  \R^{d-1}$ the vector defined by $x^i_k=x_{i,k}$ for $1 \le k <i$ and
  $x^i_k=x_{i,k+1}$ for $ i \le k \le d-1$. For $x\in \corr$, we have $\Subm{(x-xe^i_dx)}{i}=\Subm{x}{i}-x^i(x^i)^T$.
\end{itemize}


\section{Some properties of MRC processes} \label{sec_properties}

\subsection{The infinitesimal generator }

We first calculate the quadratic covariation of~$\corrp$. By
Lemma~\ref{calcul_crochet}, we get:
\begin{eqnarray}\label{crochet_MRC}
\langle d(X_t)_{i,j}, d(X_t)_{k,l} \rangle &=&  \Big[a_i^2 (\indi{i=k} 
  (X_t-X_te_d^{i}X_t)_{j,l} + \indi{i=l} (X_t-X_te_d^{i}X_t)_{j,k} )
\nonumber \\
& &    + a_j^2 ( \indi{j=k} 
  (X_t-X_te_d^{j}X_t)_{i,l} +  \indi{j=l} 
  (X_t-X_te_d^{j}X_t)_{i,k} ) \Big]dt \nonumber\\
&=&\Big[ a_i^2 (\indi{i=k}  ((X_t)_{j,l}-(X_t)_{i,j}(X_t)_{i,l}) +
\indi{i=l}  ((X_t)_{j,k}-(X_t)_{i,j}(X_t)_{i,k})) \\
& & +a_j^2 (\indi{j=k}
((X_t)_{i,l}-(X_t)_{j,i}(X_t)_{j,l}) +\indi{j=l}
((X_t)_{i,k}-(X_t)_{j,i}(X_t)_{j,k}) )  \Big]dt. \nonumber
\end{eqnarray}
We remark in particular that $d \langle (X_t)_{i,j}, d(X_t)_{k,l} \rangle=0$ when
$i,j,k,l$ are distinct.

We are now in position to calculate the
infinitesimal generator of~$\corrp$. The infinitesimal generator
on~$\genm$ is
defined by:
$$x\in \corr,\  L^{\mathcal{M}}f(x)=\lim_{t\rightarrow 0^+}
\frac{\E[f(X^x_t)]-f(x)}{t} \text{ for } f \in \mathcal{C}^2(\genm,\R) \text{ with
  bounded derivatives}. $$
By straightforward calculations, we get from~\eqref{crochet_MRC} that:
$$L^{\mathcal{M}} = \sum_{\substack{1 \le i,j \le d \\ j \not = i}}
(\kp_i+ \kp_j)(c_{i,j}-x_{i,j}) \partial_{i,j} + \frac{1}{2} \sum_{\substack{1 \le
  i,j,k \le d \\ j \not = i, k \not = i}} a_i^2
(x_{j,k}-x_{i,j}x_{i,k})[\partial_{i,j}\partial_{i,k}+\partial_{i,j}\partial_{k,i}
+\partial_{j,i}\partial_{i,k}+\partial_{j,i}\partial_{k,i} ].$$
Here, $\partial_{i,j}$ denotes the derivative with respect to the element at
the $i^{th}$ line and $j^{th}$ column. We know however that the process that we consider is valued in~$\corr
\subset \symm$. Though it is equivalent, it is often more convenient to work with
the infinitesimal generator on $\symm$, which is defined by:
$$x\in \corr, \
L f(x)=\lim_{t\rightarrow 0^+} \frac{\E[f(X^x_t)]-f(x)}{t} \text{ for } f \in \mathcal{C}^2(\symm,\R) \text{ with
  bounded derivatives}, $$
since it eliminates redundant coordinates. For $x \in \symm$, we denote by $x_{\set{i,j}}=x_{i,j}=x_{j,i}$ the value of
the  coordinates $(i,j)$ and $(j,i)$, so that $x=\sum_{1\le i \le j \le d}
x_{\set{i,j}} (e^{i,j}_d+ \indi{i \not = j}e^{j,i}_d )$. For $f  \in
\mathcal{C}^2(\symm,\R)$, $\partial_{\set{i,j}}f$ denotes its
derivative with respect to $x_{\set{i,j}}$. For $x \in
\genm$, we set $\pi(x)=(x+x^T)/2$. It is such that $\pi(x)=x$ for $x \in
\symm$, and we have $ L f(x)= L^{\mathcal{M}}f\circ \pi (x). $
By the chain rule, we have for $x\in\symm$, 
  $\partial_{i,j}f \circ \pi (x)=  (\indi{i=j}+\frac{1}{2}\indi{i\neq j})\partial_{\set{i,j}}f (x)$ and we get:

 \begin{equation}   \label{EQUATION_OPERATOR_C}
 L=  \sum_{i=1}^d \left( \sum_{\substack{1 \le j \le d \\ j \not = i}}
\kp_i (c_{\set{i,j}}-x_{\set{i,j}}) \partial_{\{i,j\}} + \frac{1}{2} \sum_{\substack{1 \le
  j,k \le d \\ j \not = i, k \not = i}} a_i^2
(x_{\{j,k\}}-x_{\{i,j\}}x_{\{i,k\}}) \partial_{\{i,j\}}\partial_{\{i,k\}} \right).
 \end{equation}
Then, we will say that a process $(X_t,t\ge 0)$ valued in~$\corr$ solves the
martingale problem of~$\corrp$ if for any $n\in \N^*$, $0\le t_1\le\dots \le t_n\le
s \le t$, $g_1,\dots, g_n \in
\mathcal{C}(\symm,\R)$, $f\in\mathcal{C}^2(\symm,\R)$ we have:
\begin{equation}\label{PbMg}
\E\left[ \prod_{i=1}^n g_i(X_{t_i})\left(f(X_t)-f(X_s)-\int_s^t Lf(X_u)du
  \right) \right]=0,\text{ and } X_0=x
\end{equation}
 Now, we state simple but interesting properties of mean-reverting correlation
 processes. Each non-diagonal coefficient follows a  Wright-Fisher
 type diffusion and any principal submatrix is also a  mean-reverting
 correlation process. This result is a direct consequence of the calculus above and the weak
 uniqueness of the SDE~\eqref{SDE_CORR} obtained in Corollary~\ref{weak_uniq}.

\begin{proposition}\label{prop_el_jac}
Let $(X_t)_{t \geq 0} \sim \corrp$. For $1\le i \not = j \le d$, there is
Brownian motion $(\beta^{i,j}_t,t\ge 0)$ such that
\begin{equation}\label{SDE_elements}
  d(X_t)_{i,j}=(\kp_i+\kp_j)(c_{i,j}-(X_t)_{i,j})dt
+\sqrt{a_i^2+a_j^2}\sqrt{1-(X_t)_{i,j}^2}d\beta^{i,j}_t.
\end{equation}
Let $I=\{k_1<\dots<k_{d'}\} \subset \{1,\dots, d \}$ such that
$1<d'<d$. For $x \in \genm$, we define $x^I \in \mathcal{M}_{d'}(\R)$
by $(x^I)_{i,j}=x_{k_i,k_j}$ for $1\le i,j \le d'$. We have:
$$(X^I_t)_{t \geq 0} \overset{law}{=}MRC_{d'}(x^I,\kp^I,c^I,a^I).$$
\end{proposition}

\subsection{Calculation of moments and the ergodic law}
We first introduce some notations that are useful to characterise the general
form for moments. For every $x \in \symm, m  \in \mathcal{S}_d(\N), $ we set:
$$x^m = \prod_{1 \leq i \le j \leq d}x_{\{i,j\}}^{m_{\{i,j\}}} \text{ and }
|m|= \sum_{1 \leq i \le j \leq d}m_{\{i,j\}}.$$
A function $f: \symm \rightarrow \R$ is a polynomial function of degree
smaller than $n\in
\N$ if there are real numbers $a_m$ such that $f(x)=\sum_{|m|\le n} a_m x^m$,
and we define the norm of~$f$ by $\|f\|_{\mathbb{P}}=\sum_{|m|\le n} |a_m|$.

We want to calculate the moments $\E[X_t^m]$ of $(X_t,t\ge 0)\sim
\corrp$. Since the diagonal elements are equal to~$1$, we will take
$m_{\{i,i\}}=0$. Let us also remark that  for $i\not = j$ such that $\kp_i=\kp_j=0$, we have
from~\eqref{cond_weak_existence} that $a_i=a_j=0$. Therefore we get
$(X_t)_{i,j}=x_{i,j}$ by~\eqref{SDE_elements}.

\begin{proposition}\label{prop_equ_ODE}
  Let $m\in \mathcal{S}_d(\N)$ such that $m_{i,i}=0$ for $1 \le i\le d$. Let $(X_t)_{t \geq 0}  \sim \corrp$. For
  $m\in\mathcal{S}_d(\N)$, $L x^m =-K_m x^m +f_m(x)$, with
  $$ K_m =\sum_{i=1}^d \sum_{j=1}^d \kp_i m_{\{i,j\}} +\frac{1}{2}\sum_{i=1}^d
  a_i^2 \sum_{j,k=1}^d m_{\{i,j\}}m_{\{i,k\}} $$
  and $$f_m(x)=\sum_{i=1}^d \sum_{j=1}^d \kp_i c_{\{i,j\}}
  m_{\{i,j\}}x^{m-e^{\{i,j\}}_d}+\frac{1}{2}\sum_{i=1}^d
  a_i^2 \sum_{j,k=1}^d m_{\{i,j\}}m_{\{i,k\}}x^{m-e^{\{i,j\}}_d-e^{\{i,k\}}_d+e^{\{j,k\}}_d} $$ is a polynomial function of degree smaller than $|m|-1$. We have
\begin{equation}\label{induc_moments}
\E\left[ X_t^m\right] = x^m \exp(-tK_m)+ \exp(-tK_m)\int_0^t \exp(sK_m) \E[f_{m}(X_s)]ds.
\end{equation}
\end{proposition}
\begin{proof}
The calculation of~$L x^m $ is straightforward
from~\eqref{EQUATION_OPERATOR_C}. By using It\^o's formula, we get easily that
$\frac{d\E[X_t^m]}{dt}=-K_mE[X_t^m]+\E[f_{m}(X_t)]$, which gives~\eqref{induc_moments}.
\end{proof}

Equation~\eqref{induc_moments} allows us to calculate explicitly any moment by
induction on~$|m|$. Here are the formula for moments of order~$1$ and~$2$:
\begin{eqnarray*}
  \forall 1 \leq i \neq j \leq d,\,\, \E\left[(X_t)_{i,j} \right]&=&x_{i,j}e^{-t \left( \kappa_{i}+ \kappa_{j}\right)} + c_{i,j}(1-e^{-t(\kappa_i+\kappa_j)})  ,
\end{eqnarray*}
and for given $1 \leq i \neq j \leq d$ and $1 \leq k \neq l \leq d $
such that $\kp_i+\kp_j>0$ and $\kp_k+\kp_l>0$, 
\begin{eqnarray*}
 \E\left[(X_t)_{i,j}(X_t)_{k,l} \right] &=& x_{i,j}x_{k,l}e^{-t{K}_{i,j,k,l}} + (\kappa_i+\kappa_j)c_{i,j}\gamma_{k,l}(t)+ (\kappa_k + \kappa_l)c_{k,l}\gamma_{i,j}(t)\\
&&+ a_i^2\left( \indi{i=k} \gamma_{j,l}(t)+ \indi{i=l} \gamma_{j,k}(t)\right)+a_j^2\left( \indi{j=k} \gamma_{i,l}(t)+ \indi{j=l} \gamma_{i,k}(t)\right),
\end{eqnarray*}
where $K_{i,j,k,l} =  \kappa_i+ \kappa_j + \kappa_k+ \kappa_l + a_i^2\left(
\indi{i=k} + \indi{i=l} \right)+a_j^2\left( \indi{j=k}  + \indi{j=l}  \right)$
and
$$
\forall m,n \in \left \lbrace i,j,k,l \right \rbrace,\, \gamma_{m,n}(t) = c_{m,n}\frac{1-e^{-t{K}_{i,j,k,l}}}{{K}_{i,j,k,l}} + (x_{m,n}-c_{m,n})\frac{e^{-t(\kappa_m+\kappa_n)}-e^{t{K}_{i,j,k,l}}}{{K}_{i,j,k,l}-\kappa_m-\kappa_n}.
$$
Let $f$ be a polynomial function of degree smaller than $n\in \N$. From
Proposition~\ref{prop_equ_ODE}, $L$ is a linear mapping on the polynomial
functions of degree smaller than $n$, and there is a constant $C_n>0$ such
that $\|Lf\|_{\mathbb{P}}\le C_n \|f\|_{\mathbb{P}}$. On the other hand, we
have by It\^o's formula $\E[f(X_t)]=f(x)+\int_0^t\E[Lf(X_s)]ds$, and by
iterating $\E[f(X_t)]=\sum_{i=0}^k\frac{t^i}{i!} L^i f(x) +
\int_0^t\frac{(t-s)^k}{k!}\E[L^{k+1}f(X_s)]ds$. Since $\|L^i
f\|_{\mathbb{P}}\le C_n^i \|f\|_{\mathbb{P}}$, the series converges and we
have
\begin{equation}\label{ito_polynom}
  \E[f(X_t)]=\sum_{i=0}^\infty \frac{t^i}{i!} L^i f(x)
\end{equation}
for any polynomial function~$f$. We also remark that the same iterated It\^o's
formula gives
\begin{equation}\label{controle_erreur}
 \forall f\in
\mathcal{C}^\infty(\symm,\R), \forall k\in \N^*, \exists C>0, \forall x \in \corr, \  |\E[f(X_t)]-\sum_{i=0}^k\frac{t^i}{i!} L^i f(x)|\le C t^{k+1},
\end{equation}
since $L^{k+1}f$ is a bounded function on~$\corr$.

Let us discuss some interesting consequences of Proposition~\ref{prop_equ_ODE}.
Obviously, we can calculate explicitly in the same manner
$\E[X_{t_1}^{m_1}\dots X_{t_n}^{m_n}]$ for $0\le t_1 \le \dots \le t_n$ and
$m_1,\dots,m_n \in \mathcal{S}_d(\N)$. Therefore, the law of
$(X_{t_1},\dots,X_{t_n})$ is entirely determined and we get the weak
uniqueness for the SDE~\eqref{SDE_CORR}.
\begin{Corollary}\label{weak_uniq}
Every solution $(X_t,t\ge 0)$ to the martingale problem~\eqref{PbMg} have the
same law.
\end{Corollary}

Proposition~\ref{prop_equ_ODE} allows us to compute the limit
$\lim_{t\rightarrow +\infty} \E[X_t^m] $ that we note $\E[X_\infty^m]$ by a slight
abuse of notation. Let us observe that $K_m> 0$ if and only if there is $i,j$
such that $\kp_i+\kp_j>0$ and $m_{i,j}>0$.  We have
\begin{eqnarray} \label{ergo_mom}
 \E[X_\infty^m]&=&x^m \text{ if } m\in \mathcal{S}_d(\N) \text{ is such that } m_{\{i,j\}}>0 \iff \kp_i=\kp_j=0,\\
\E[X_\infty^m]&=&\E[f_m(X_\infty)]/K_m \text{ otherwise}. \nonumber
\end{eqnarray}
Thus, $X_t$ converges in law when $t \rightarrow + \infty$, and the moments $\E[X_\infty^m]$ are uniquely
determined by~\eqref{ergo_mom} with an induction on~$|m|$. In addition, if
$\kp_i+\kp_j>0$ for any $1\le i,j\le d$ (which means that at most only one
coefficient of $\kp$ is equal to~$0$), the law of $X_\infty$ does not depend
on the initial condition and is the unique invariant law. In this case the
ergodic moments of order~$1$ and~$2$ are given by:
\begin{eqnarray*}
 \E\left[(X_\infty)_{i,j} \right] &=& c_{i,j},\\
 \E\left[(X_\infty)_{i,j}(X_\infty)_{k,l} \right] &=&
 \frac{(\kappa_i+\kappa_j+\kappa_k+\kappa_l)c_{i,j}c_{k,l}+a_i^2(\indi{i=k}c_{j,l}+\indi{i=l}c_{j,k}
   )+a_j^2(  \indi{j=k}c_{i,l}+ \indi{j=l}c_{i,k} )}{K_{i,j,k,l}}.
\end{eqnarray*}

\subsection{The connection with Wishart processes}

Wishart processes are affine processes on positive semidefinite
matrices. They have been introduced by Bru~\cite{Bru} and solves the
following SDE:
\begin{equation}\label{EDS_Wishart}
Y_t^y =y + \int_{0}^t \left( (\alpha+1) a^T a  + bY_s^y+Y_s^yb^T \right )ds +
\int_{0}^t \left( \sqrt{Y_s^y}dW_sa + a^TdW_s^T\sqrt{Y_s^y}\right ) , 
\end{equation}
where $a,b\in \genm$ and $y \in \posm$.  Strong uniqueness holds when $\alpha
\ge d$ and $y\in\dpos$. Weak
existence and uniqueness holds when $\alpha \ge d-2$. This is in fact very similar to the
results that we obtain for mean-reverting correlation processes. The parameter
$\alpha+1$ is called the number of degrees of freedom, and we denote by
$WIS_d(y,\alpha+1,b,a)$ the law of $(Y^y_t,t\ge 0)$.

Once we have a positive semidefinite matrix~$y \in \posm$ such that
$y_{i,i}>0$ for $1 \le i\le d$, a trivial way to
construct a correlation matrix is to consider ${\bf
  p}(y)$, where ${\bf p}$ is defined by~\eqref{proj_corr}. Thus, it is
somehow natural then to look at the dynamics of ${\bf p}(Y_t^y)$, provided
that the diagonal elements of the Wishart process do not vanish. In general,
this does not lead to an autonomous SDE. However, the
particular case where the Wishart parameters are $a=e^1_d$ and $b=0$ is
interesting since it leads to the SDE satisfied by the
mean-reverting correlation processes, up to a change of time. Obviously, we
have a similar property for $a=e^i_d$ and $b=0$ by a permutation of the $i$th
and the first coordinates. 
\begin{proposition}\label{WIS_MRC}
  Let $\alpha \ge \max(1,d-2)$ and $y\in\posm$ such that $y_{i,i}>0$ for $1
  \le i\le  d$. Let $(Y^y_t)_{t\ge 0} \sim
  WIS_d(y,\alpha+1,0,e^1_d)$. Then, $(Y^y_t)_{i,i}=y_{i,i}$ for $2\le i \le d$ and
  $(Y^y_t)_{1,1}$ follows a squared Bessel process of dimension $\alpha+1$ and
  a.s. never vanishes. We set
  $$ X_t={\bf p}(Y^y_t), \ \phi(t)=\int_0^t \frac{1}{(Y^y_s)_{1,1}} ds.$$
  The function $\phi$ is  a.s. one-to-one on $\R_+$ and defines a time-change
  such that:
  $$ (X_{\phi^{-1}(t)},t\ge 0)\overset{law}{=}\corrps{d}{{\bf
      p}(y)}{\frac{\alpha}{2}e^1_d}{I_d}{e^1_d}.$$
  In particular, there is a weak solution to~$\corrps{d}{{\bf
      p}(y)}{\frac{\alpha}{2}e^1_d}{I_d}{e^1_d}$.
  Besides, the processes $(X_{\phi^{-1}(t)},t\ge 0)$ and $((Y^y_t)_{1,1},t\ge 0)$ are independent. 
\end{proposition}

\begin{proof}
  From~\eqref{EDS_Wishart}, $a=e^1_d$ and $b=0$, we get $d(Y_t^y)_{i,j}=0$ for
  $2 \le i,j \le d$ and
  \begin{equation}\label{dyn_Y11}d(Y_t^y)_{1,1}=(\alpha+1)dt + 2 \sum_{k=1}^d
  (\sqrt{Y_t^y})_{1,k}(dW_t)_{k,1}, \
  d(Y_t^y)_{1,i}=\sum_{k=1}^d(\sqrt{Y_t^y})_{i,k}(dW_t)_{k,1}.
\end{equation}
In particular, $d\langle (Y_t^y)_{1,1}\rangle=4(Y_t^y)_{1,1} dt$ and
$(Y_t^y)_{1,1}$ is a squared Bessel process of dimension $\alpha+1$. Since
$\alpha+1\ge 2$ it almost surely  never
vanishes. Thus, $(X_t,t\ge 0)$ is well defined, and we get:
\begin{eqnarray}\label{dyn_Xi}
  d(X_t)_{1,i}&=&-\frac{\alpha}{2}(X_t)_{1,i}
  \frac{dt}{(Y_t^y)_{1,1}}+\sum_{k=1}^d \left(\frac{(\sqrt{Y_t^y})_{i,k}}{\sqrt{(Y_t^y)_{1,1}
  y_{i,i}}} - (X_t)_{1,i} \frac{(\sqrt{Y_t^y})_{1,k}}{(Y_t^y)_{1,1}} \right)(dW_t)_{k,1}
\end{eqnarray}
By Lemma~\ref{lemma_change_time}, $\phi(t)$ is a.s. one-to-one on $\R_+$, and we consider the
Brownian motion $(\tilde{W}_t,t\ge 0)$ defined by
$(\tilde{W}_{\phi(t)})_{i,j}=\int_0^t
\frac{(dW_s)_{i,j}}{\sqrt{(Y_s^y)_{1,1}}}ds$. We have by straightforward
calculus
\begin{eqnarray}\label{dyn_X}
&&  d(X_{\phi^{-1}(t)})_{1,i}=-\frac{\alpha}{2}(X_{\phi^{-1}(t)})_{1,i}
  dt+\sum_{k=1}^d \left(\frac{(\sqrt{Y_{\phi^{-1}(t)}^y})_{i,k}}{\sqrt{
  y_{i,i}}} - (X_{\phi^{-1}(t)})_{1,i} \frac{(\sqrt{Y_{\phi^{-1}(t)}^y})_{1,k}}{\sqrt{(Y_{\phi^{-1}(t)}^y)_{1,1}}} \right)(d\tilde{W}_t)_{k,1}\\
&&  d\langle
  (X_{\phi^{-1}(t)})_{1,i},(X_{\phi^{-1}(t)})_{1,j}\rangle=[{(X_{\phi^{-1}(t)})_{i,j}-(X_{\phi^{-1}(t)})_{1,i}(X_{\phi^{-1}(t)})_{1,j}}]  dt, \nonumber
\end{eqnarray}
which shows by uniqueness of the solution of the martingale problem
(Corollary~\ref{weak_uniq}) that $(X_{\phi^{-1}(t)},t\ge
0)\overset{law}{=}\corrps{d}{{\bf p}(y)}{\frac{\alpha}{2}e^1_d}{I_d}{e^1_d}$.

Let us now show the independence. We can check easily that
\begin{equation}\label{crochet_nul}
  d\langle
  (X_t)_{1,i},(X_t)_{1,j}\rangle=\frac{1}{(Y_t^y)_{1,1}}[(X_t)_{i,j}-(X_t)_{1,i}(X_t)_{1,j}]
  \text{ and }d\langle
  (X_t)_{1,i},(Y^y_t)_{1,1}\rangle=0.
\end{equation}
 We define $\Psi(y) \in \symm$ for $y\in \posm$ such that
  $y_{i,i}>0$ by $\Psi(y)_{1,i}=\Psi(y)_{i,1}=y_{1,i}/\sqrt{y_{1,1}
    y_{i,i}}$ and $\Psi(y)_{i,j}=y_{i,j}$ otherwise. By~\eqref{dyn_Y11}
  and~\eqref{dyn_Xi}, $(\Psi({Y}_t),t\ge 0)$ solves an SDE on~$\symm$. This
  SDE has a unique weak solution. Indeed, we can check that for any solution
  $(\tilde{Y}_t,t\ge 0)$ starting from $\Psi(y)$, 
  $(\Psi^{-1}(\tilde{Y}_t),t\ge 0)\sim WIS_d(y,\alpha+1,0,e^1_d) $, which
  gives our claim since $\Psi$ is one-to-one and weak uniqueness holds for
  $WIS_d(y,\alpha+1,0,e^1_d) $ (see~\cite{Bru}).  Let $(B_t,t\ge 0)$ denote a real Brownian motion
  independent of~$(W_t,t\ge 0)$. We consider a weak solution to the SDE
  \begin{eqnarray*}
    d(\bar{Y}_t)_{1,1}&=&(\alpha+1)dt + 2 \sqrt{(\bar{Y}_t)_{1,1}}dB_t,\ 
    d(\bar{Y}_t)_{i,j}=0 \text{ for } 2\le i,j \le d,\\
    d(\bar{Y}_t)_{1,i}&=&-\frac{\alpha}{2}(\bar{Y}_t)_{1,i}
  \frac{dt}{(\bar{Y}_t)_{1,1}}+\sum_{k=1}^d \left(\frac{(\sqrt{\bar{Y}_t})_{i,k}}{\sqrt{(\bar{Y}_t)_{1,1}
  y_{i,i}}} - (\bar{Y}_t)_{1,i}
\frac{(\sqrt{\bar{Y}_t})_{1,k}}{(\bar{Y}_t)_{1,1}} \right)(dW_t)_{k,1}, \ i=2,\dots,d
  \end{eqnarray*}
  that starts from  $\bar{Y}_0=\Psi(y)$. It solves the same martingale problem
  as $\Psi(Y_t)$, and therefore $(\Psi(Y_t),t\ge 0)
  \overset{law}{=}(\bar{Y}_t,t\ge 0)$. We set $\bar{\phi}(t)= \int_0^t
  \frac{1}{(\bar{Y}_s)_{1,1}}ds$. As above,
  $((\bar{Y}_{\bar{\phi}^{-1}(t)})_{1,i},i=2\dots,d)$ solves an SDE driven by
  $(W_t,t\ge0)$ and is therefore independent of $((\bar{Y}_t)_{1,1},t\ge 0)$,
  which gives the desired independence.
\end{proof}
\begin{remark}
  There is   a connection between squared-Bessel processes and
  one-dimensional Wright-Fisher diffusions that is similar to Proposition~\ref{WIS_MRC}.
  Let us consider $Z^i_t=z_i+\beta_i t + \int_0^t \sigma \sqrt{Z^i_s}d B^i_s, i=1,2$
  two squared Bessel processes driven by independent Brownian motions. We
  assume that $\beta_1,\beta_2,\sigma \ge 0$ and $\sigma^2 \le
  2(\beta_1+\beta_2)$ so that $Y_t=Z^1_t+Z^2_t$ is a squared Bessel processes
  that never reaches~$0$. By using It\^o calculus, there is a real Brownian
  $(B_t,t\ge 0)$
  motion such that $X_t=Z^1_t/Y_t$ satisfies 
  $$dX_t=(\beta_1+\beta_2)(\frac{\beta_1}{\beta_1+\beta_2}-X_t)\frac{dt}{Y_t}+\sigma
  \sqrt{X_t(1-X_t)} \frac{d B_t}{\sqrt{Y_t}},$$
  and we have $\langle dX_t,dY_t \rangle =0$. Thus, we can use the same
  argument as in the proof above: we set $\phi(t)=\int_0^t 1/(Y_s)ds$ and get
  that $(X_{\phi^{-1}(t)},t\ge 0)$ is a one-dimensional Wright-Fisher
  diffusion that is independent of $(Y_t,t\ge 0)$. This property obviously
  extends the well known identity between Gamma and Beta laws. This kind of
  change of time have also been considered in the literature
  by~\cite{Fernholz_Karatzas} or~\cite{Gourieroux} for similar but different
  multi-dimensional settings.\end{remark}

\subsection{A remarkable splitting of the infinitesimal generator}

In this section, we present a remarkable splitting for the mean-reverting
correlation matrices. This result will play a key role in the simulation part. In
fact, we have already obtained in~\cite{AA} very similar properties for Wishart
processes. Of course, these properties are related through
Proposition~\ref{WIS_MRC}, which is illustrated in the proof below.

\begin{theorem}\label{theorem_spliopper}
Let $\alpha \ge d-2$.  Let $L$ be the generator associated to the 
  $\corrps{d}{x}{\frac{\alpha}{2} a^2}{I_d}{a}$ on $\corr$ and $L_i$ be the generator associated to
  $\corrps{d}{x}{\frac{\alpha}{2} e^i_d}{I_d}{e^i_d}$, for $i \in \lbrace 1,\ldots,d \rbrace$.
Then, we have  
 \begin{equation}\label{Split_Can}
   L= \sum_{i=1}^d a_i^2 L_i\text{   and    }\,\,\,\,\forall i,j \in \lbrace
  1,\ldots,d \rbrace, \,  L_iL_j= L_jL_i.
\end{equation}
\end{theorem}
\begin{proof}
The formula $ L= \sum_{i=1}^d a_i^2 L_i$ is obvious
from~\eqref{EQUATION_OPERATOR_C}. The commutativity property can be obtained
directly by a tedious but simple calculus, which is made in
Appendix~\ref{proof_theorem_spliopper}. Here, we give another proof that uses
the link between Wishart and Mean-Reverting Correlation processes given by Proposition~\ref{WIS_MRC}.

Let $L^W_i$ denotes the generator
of~$WIS_d(x,\alpha+1,0,e^i_d)$. From~\cite{AA}, we have $L^W_iL^W_j=L^W_j
L^W_i$ for $1\le i,j\le d$.  Let us consider $\alpha\ge \max(5,d-2)$ and $x \in \corr$. We set for $i=1,2$
 $(Y^{i,x}_t,t\ge 0)\sim WIS_d(x,\alpha+1,0,e^i_d)$, and we
  assume that the Brownian motions of their associated SDEs are independent. Since  $L^W_1L^W_2=L^W_2 L^W_1$,
we know from~\cite{AA} that $Y^{1,Y^{2,x}_t}_t\overset{law}{=}Y^{2,Y^{1,x}_t}_t$ and thus $$\E[f({\bf
  p}(Y^{1,Y^{2,x}_t}_t))]=\E[f({\bf p}(Y^{2,Y^{1,x}_t}_t))],$$
for any polynomial function $f$. By Proposition~\ref{WIS_MRC}, ${\bf
  p}(Y^{1,Y^{2,x}_t}_t)\overset{law}{=}X^{1,{\bf
    p}(Y^{2,x}_t)}_{(\phi^1)^{-1}(\phi^1(t))} $, where $(X^{1,{\bf
    p}(Y^{2,x}_t)}_{(\phi^1)^{-1}(u)},u\ge 0)$ is a mean-reverting correlation
process independent of $\phi^1(t)=\int_0^t\frac{1}{(Y^{1,Y^{2,x}_t}_s)_{1,1}}ds $. Since
$(Y^{2,x}_t)_{1,1}=1$, $(Y^{1,Y^{2,x}_t}_s)_{1,1}$ follows a squared Bessel
of dimension~$\alpha+1$ starting from~$1$. Using the independence, we get by~\eqref{controle_erreur}
$$\E[f({\bf p}(Y^{1,Y^{2,x}_t}_t))|Y^{2,x}_t,\phi^1(t)]=f({\bf
  p}(Y^{2,x}_t))+\phi^1(t) L_1 f({\bf
  p}(Y^{2,x}_t))+ \frac{\phi^1(t)^2}{2}L_1^2 f({\bf
  p}(Y^{2,x}_t)) + O(\phi^1(t)^3).$$ By Lemma~\ref{lemma_change_time2}, we have
$\E[\phi^1(t)]=t+\frac{3-\alpha}{2} t^2 + O(t^3)$, $\E[\phi^1(t)^2]=t+O(t^3)$,
$\E[\phi^1(t)^3]=O(t^3)$. Thus, we get:
$$\E[f({\bf p}(Y^{2,Y^{1,x}_t}_t))|Y^{2,x}_t]=f({\bf
  p}(Y^{2,x}_t))+ t L_1 f({\bf
  p}(Y^{2,x}_t)) + \frac{t^2}{2} [L_1^2 f({\bf
  p}(Y^{2,x}_t)) +(3-\alpha)L_1f({\bf
  p}(Y^{2,x}_t))]+ O(t^3).$$ Once again, we use  Proposition~\ref{WIS_MRC} and~\eqref{controle_erreur}
 to get similarly that $E[f({\bf
  p}(Y^{2,x}_t))]=f(x)+t L_2f(x)+ \frac{t^2}{2}[L_2^2f(x) + (3-\alpha)
L_2f(x)] + O(t^3)$ for any polynomial function~$f$. We finally get:
$$ \E[f({\bf p}(Y^{1,Y^{2,x}_t}_t))]=f(x)+t(L_1+ L_2)f(x)+
\frac{t^2}{2}[L_1^2f(x) + 2 L_2 L_1f(x)+
L_2^2f(x) + (3-\alpha)(L_1+
L_2)f(x)] + O(t^3).$$
Similarly, we also have
\begin{equation}\label{compo_wis_proj} \E[f({\bf p}(Y^{2,Y^{1,x}_t}_t))]=f(x)+t(L_1+ L_2)f(x)+
\frac{t^2}{2}[L_1^2f(x) + 2 L_1 L_2f(x)+
L_2^2f(x) + (3-\alpha)(L_1+
L_2)f(x)]t^2 + O(t^3),
\end{equation}
and since both expectations are equal, we get $L_1L_2f(x)=L_2L_1f(x)$ for any
$\alpha\ge \max(5,d-2)$. However, we can write $L_i=\frac{1}{2}(\alpha L_i^D+L_i^M)$, with
$$L^D_i= \sum_{\substack{1 \le j \le d \\ j \not = i}}
x_{\set{i,j}} \partial_{\{i,j\}} \text{ and }L_i^M=\sum_{\substack{1 \le
  j,k \le d \\ j \not = i, k \not = i}} 
(x_{\{j,k\}}-x_{\{i,j\}}x_{\{i,k\}}) \partial_{\{i,j\}}\partial_{\{i,k\}}.$$
Thus, we have  $\alpha^2 L_1^D L_2^D + \alpha  ( L_1^D L_2^M+ L_1^M L_2^D) +
L_1^M L_2^M =\alpha^2 L_2^D L_1^D + \alpha  ( L_2^D L_1^M+ L_2^M L_1^D) +
L_2^M L_1^M$ for any $\alpha \ge \max(5,d-2)$. This gives $L_1^D L_2^D=L_2^D
L_1^D$, $ L_1^D L_2^M+ L_1^M L_2^D=L_2^D L_1^M+ L_2^M L_1^D$, $L_1^M L_2^M
=L_2^M L_1^M$, and therefore $L_1L_2=L_2L_1$ holds without restriction on~$\alpha$.
\end{proof}

\begin{remark}\label{rem_oper}
Let $x\in\corr$, $(Y^{1,x}_t,t\ge 0)\sim WIS_d(x,\alpha+1,0,e^1_d)$ and $L^W_1$ its
infinitesimal generator. Equation~\eqref{compo_wis_proj} and the formula $E[f({\bf
  p}(Y^{1,x}_t))]=f(x)+t L_1f(x)+ \frac{t^2}{2}[L_1^2f(x) + (3-\alpha)
L_1f(x)] + O(t^3)$ used in the proof above lead formally to the following
identities for $x \in \corr$ and  $f \in \mathcal{C}^\infty(\symm,\R)$, 
$$ L_1^W (f\circ {\bf p})
(x)=L_1 f(x), \ (L_1^W)^2 (f\circ {\bf p}) (x)=L_1^2 f(x)+(3-\alpha)L_1
f(x), \ L_1^W L_2^W (f\circ {\bf p}) (x)=L_1L_2 f(x),$$
that can be checked by basic calculations.
\end{remark}

The property given by Theorem~\ref{theorem_spliopper} will help us to prove  the weak existence of
mean-reverting correlation processes. It plays also a key
role to construct discretization scheme for these diffusions. In fact, it
gives a simple way to sample the law $\corrpst{d}{x}{\frac{\alpha}{2}
  a^2}{I_d}{a}$. Let $x \in \corr$. We construct iteratively:
\begin{itemize}
  \item  $X^{1,x}_t \sim \corrpst{d}{x}{\frac{\alpha}{2} a_1^2 e^1_d}{I_d}{ a_1 e^1_d}$,
\item For $2\le i\le d$, conditionally to $X^{i-1,\dots^{X^{1,x}_t}}_t$, $X^{i,\dots^{X^{1,x}_t}}_t\sim
  \corrpst{d}{X^{i-1,\dots^{X^{1,x}_t}}_t}{\frac{\alpha}{2} a_i^2 e^i_d}{I_d}{ a_i
    e^i_d}$ is sampled independently according to the distribution of a mean-reverting
correlation process at time~$t$ with parameters $(\frac{\alpha}{2} a_i^2
e^i_d,I_d, a_i e^i_d)$ starting from
$X^{i-1,\dots^{X^{1,x}_t}}_t$.
\end{itemize}
\begin{proposition}\label{sim_exact}
Let $X^{d,\dots^{X^{1,x}_t}}_t$ be defined as above. Then,
$X^{d,\dots^{X^{1,x}_t}}_t\sim \corrpst{d}{x}{\frac{\alpha}{2}
  a^2}{I_d}{a}$.
\end{proposition}
Let us notice that $\corrpst{d}{x}{\frac{\alpha}{2} a_i^2 e^i_d}{I_d}{ a_i
    e^i_d}\overset{law}{=}\corrpstb{d}{x}{\frac{\alpha}{2} e^i_d}{I_d}{ 
    e^i_d}{a_i^2 t}$ and that $\corrpst{d}{x}{\frac{\alpha}{2} e^i_d}{I_d}{
    e^i_d}$ and $\corrpst{d}{x}{\frac{\alpha}{2} e^1_d}{I_d}{ 
    e^1_d}$ are the same law up to the permutation of the first and the $i$-th
  coordinate. Thus, it is sufficient to be able to sample this latter law in
  order to sample  $\corrpst{d}{x}{\frac{\alpha}{2}
  a^2}{I_d}{a}$ by Proposition~\ref{sim_exact}.

\begin{proof}
Let $f$ be a polynomial function and $X^x_t\sim\corrpst{d}{x}{\frac{\alpha}{2}
  a^2}{I_d}{a}$. By~\eqref{ito_polynom}, $\E[f(X^x_t)]=\sum_{j=0}^\infty
\frac{t^j}{j!} L^jf(x)$.
Using once again~\eqref{ito_polynom}, 
$\E[f(X^{d,\dots^{X^{1,x}_t}}_t)]= \E[\E[f(X^{d,\dots^{X^{1,x}_t}}_t)|X^{d-1,\dots^{X^{1,x}_t}}_t]]$\\$=\sum_{j=0}^\infty
\frac{t^j}{j!} \E[ L_d^j f(X^{d-1,\dots^{X^{1,x}_t}}_t)]$, and we finally
obtain by iterating
$$\E[f(X^{d,\dots^{X^{1,x}_t}}_t)]=
\sum_{j_1,\dots,j_d=0}^\infty\frac{t^{j_1+\dots+j_d}}{j_1!\dots
  j_d!}L_1^{j_1}\dots L_d^{j_d}f(x)=\sum_{j=0}^d
\frac{t^j}{j!}(L_1+\dots+L_d)^j f(x)=\E[f(X^x_t)],$$
since the operators commute.
\end{proof}

We can also extend Proposition~\ref{sim_exact} to the limit laws. More
precisely, let us denote by $\corrpstb{d}{x}{\kp}{c}{ a}{\infty}$ the law
characterized by~\eqref{ergo_mom}. We define similarly for $x \in \corr$,
$X^{1,x}_\infty \sim \corrpstb{d}{x}{\frac{\alpha}{2} a_1^2 e^1_d}{I_d}{ a_1
  e^1_d}{\infty}$ and, conditionally to $X^{i-1,\dots^{X^{1,x}_\infty}}_\infty$,
 $X^{i,\dots^{X^{1,x}_\infty}}_\infty \sim
  \corrpstb{d}{X^{i-1,\dots^{X^{1,x}_\infty}}_\infty}{\frac{\alpha}{2} a_i^2 e^i_d}{I_d}{ a_i
    e^i_d}{\infty}$ for $2\le i\le d$. We have:
  \begin{equation}\label{commut_infty}X^{d,\dots^{X^{1,x}_\infty}}_\infty \sim \corrpstb{d}{x}{\frac{\alpha}{2}
  a^2}{I_d}{a}{\infty}.
\end{equation}
To check this we consider $(X_t,t\ge 0)\sim \corrps{d}{x}{\frac{\alpha}{2}
  a^2}{I_d}{a}$ and $m\in \mathcal{S}_d(\N)$ such that $m_{i,i}=0$. By
Proposition~\ref{prop_equ_ODE}, $\E[X^m_t]$ is a polynomial function of~$x$
that we write $\E[X^m_t]=\sum_{m' \in \mathcal{S}_d(\N),|m'|\le |m|} \gamma_{m,m'}(t) x^{m'}$. From the
convergence in law~\eqref{ergo_mom}, we get that the coefficients
$\gamma_{m,m'}(t)$ go to a limit $\gamma_{m,m'}(\infty)$ when $t\rightarrow +
\infty$, and $\E[X^m_\infty]=\sum_{|m'|\le |m|} \gamma_{m,m'}(\infty)
x^{m'}$. Similarly, the moment~$m$ of $\corrpstb{d}{x}{\frac{\alpha}{2} a_i^2 e^i_d}{I_d}{ a_i
    e^i_d}{t}$ can be written as $\sum_{|m'|\le |m|} \gamma^i_{m,m'}(t)
  x^{m'}$. We get from Proposition~\ref{sim_exact}:
  $$\E[X^m_t]=\sum_{|m_1|\le \dots\le |m_d| \le|m|}
    \gamma^d_{m,m_d}(t)\gamma^{d-1}_{m_d,m_{d-1}}(t) \dots
    \gamma^{1}_{m_2,m_{1}}(t) x^{m_1},$$
    which gives~\eqref{commut_infty} by letting $t\rightarrow + \infty$.

\subsection{A link with the multi-allele Wright-Fisher model}

Theorem~\ref{theorem_spliopper} and Proposition~\ref{sim_exact} have shown
that any law $\corrpst{d}{x}{\frac{\alpha}{2} a^2}{I_d}{a}$ can be obtained by
composition with the elementary law $\corrpst{d}{x}{\frac{\alpha}{2}
}{I_d}{e^1_d}$. By the next proposition, we can go further  and focus on the case where
$(x_{i,j})_{2 \le i,j\le d}=I_{d-1}$.
\begin{proposition}\label{prop_reduction}
Let $x \in \corr$. Let  $u
\in \mathcal{M}_{d-1}(\R)$ and $\check{x} \in \corr$ such that
$x=\left( \begin{array}{cc} 1& 0\\ 0 & u \end{array} \right)
\check{x}\left( \begin{array}{cc} 1& 0\\ 0 & u^T \end{array} \right)$ and
$(\check{x})_{2 \le i,j\le d}=I_{d-1}$ (Lemma~\ref{Lemma_Decompo} gives a construction of such
matrices). Then, for $\alpha \ge 2$,
$$ \corrps{d}{x}{ \frac{\alpha}{2} e^1_d}{I_d}{ e^1_d}\overset{law}{=} \left( \begin{array}{cc} 1& 0\\ 0 & u \end{array}
\right)\corrps{d}{\check{x}}{ \frac{\alpha}{2}   e^1_d}{I_d}{ e^1_d}  \left( \begin{array}{cc} 1& 0\\ 0 & u^T \end{array}
\right). $$
\end{proposition}
\begin{proof}
Let $(\check{X}_t,t\ge 0)\sim \corrps{d}{\check{x}}{ \frac{\alpha}{2} 
  e^1_d}{I_d}{ e^1_d}$. We set $X_t=\left( \begin{array}{cc} 1& 0\\ 0 & u \end{array}
\right) \check{X}_t \left( \begin{array}{cc} 1& 0\\ 0 & u^T \end{array}
\right).$ Clearly,  $((\check{X}_t)_{i,j})_{2\le i,j\le d}= I_{d-1}$ and the matrix
$((X_t)_{i,j})_{2\le i,j\le d}$ is constant and 
equal to $uu^T=(x_{i,j})_{2\le i,j\le d}$. We have for $2\le i\le d$,
$(X_t)_{1,i}=\sum_{k=2}^d u_{i-1,k-1}(\check{X}_t)_{1,k}$. By~\eqref{crochet_MRC}, we get
$\langle d  (\check{X}_t)_{1,k},d  (\check{X}_t)_{1,l}
\rangle=[\indi{k=l}-(\check{X}_t)_{1,k}(\check{X}_t)_{1,l}]dt$. Therefore,
the quadratic variations
\begin{eqnarray*}
  \langle d  (X_t)_{1,i},d  (X_t)_{1,j} \rangle&=& \left(
\sum_{k=2}^du_{i-1,k-1}u_{j-1,k-1} -
\sum_{k,l=2}^du_{i-1,k-1}(\check{X}_t)_{1,k}u_{i-1,l-1}(\check{X}_t)_{1,l}
\right)dt \\
&=& \left((X_t)_{i,j}-(X_t)_{1,i}(X_t)_{1,j}\right)dt ,
\end{eqnarray*}
are by~\eqref{crochet_MRC} the one of $\corrps{d}{x}{ \frac{\alpha}{2} 
  e^1_d}{I_d}{ e^1_d}$. This gives the claim by using the weak uniqueness
(Corollary~\ref{weak_uniq}). \end{proof}

For $x \in \symm$ such that $(x_{i,j})_{2 \le i,j\le d}=I_{d-1}$ and
$x_{1,1}=1$, we have $\det(x)=1-\sum_{i=2}^d x_{1,i}^2$ and
therefore
\begin{equation}\label{cns_corr}x \in \corr \iff  \sum_{i=2}^d x_{1,i}^2 \le 1.
\end{equation}
The process $(X_t)_{t\ge 0} \sim\corrpst{d}{x}{\frac{\alpha}{2}
}{I_d}{e^1_d}$ is such that $((X_t)_{i,j})_{2 \le i,j\le d}=I_{d-1}$. In this
case, the only non constant elements are on the first row (or column). More
precisely, $((X_t)_{1,i})_{i=2,\dots,d}$  is a vector process on the unit ball
in dimension~$d-1$ such that
$$d \langle (X_t)_{1,i}, (X_t)_{1,j} \rangle = (\indi{i=j}-(X_t)_{1,i}
(X_t)_{1,j})dt.$$
For $i=1,\dots,d-1$, we set $\zeta^i_t=(X_t)_{1,i+1}^2$. We have
 $\langle d\zeta^i_t, d\zeta^j_t
\rangle=4\zeta^i_t(\indi{i=j}-\zeta^j_t)dt$ and the drift of $\zeta^i_t$ is
$(1-(1+2\alpha)\zeta^i_t)dt$ . Thus, $(\zeta^i_t)_{1\le i\le
  d-1}$ satisfies $\sum_{i=1}^{d-1}\zeta^i_t \le 1$ and has the following infinitesimal generator
$$\sum_{i=1}^{d-1}[1-(1+2 \alpha)  z_i] \partial_{z_i} +2 \sum_{1 \le
  i,j\le d-1} z_i(\indi{i=j}-z_j)\partial_{z_i}\partial_{z_j}$$
This is a particular case of the multi-allele Wright-Fisher diffusion (see for
example Etheridge~\cite{Etheridge}), where
$(\zeta^1_t,\dots,\zeta^{d-1}_t,1-\sum_{i=1}^{d-1}\zeta^i_t)$ describes
population ratios along the time. Similar diffusions have also been considered
by Gourieroux and Jasiak~\cite{Gourieroux2} in a different context. Roughly speaking,
$((X_t)_{1,i})_{2 \le i\le d}$ can be seen as a square-root of a
multi-allele Wright-Fisher diffusion that is such that its drift coefficient
remains linear.

Also, the identity in law given by Proposition~\ref{prop_reduction} allows us to
compute more explicitly the ergodic limit law. Let $x\in \corr$ such that
$(x_{i,j})_{2 \le i,j\le d}=I_{d-1}$,  $(X^x_t)_{t\ge 0} \sim
\corrps{d}{x}{\frac{\alpha}{2} e^1_d}{I_d}{e^1_d}$ and $(Y^x_t)_{t\ge 0} \sim
WIS_d(x,\alpha+1,0,e^1_d)$. We know by~\cite{AA} that $((Y^x_t)_{i,j})_{1\le i,j\le
  d}=I_{d-1}$ and
$$((Y^x_t)_{1,i})_{1\le i \le d}\overset{law}{=}
(Z^{x_{1,1}}_t+\sum_{i=2}^d(x_{1,i}+\sqrt{t} N_i)^2,x_{1,2}+\sqrt{t} N_2
,\dots,x_{1,d}+\sqrt{t} N_d), $$
where $N_i \sim \mathcal{N}(0,1)$ are independent standard Gaussian variables
and $Z^{x_{1,1}}_t=x_{1,1}+(\alpha+2-d)t + 2 \int_0^t \sqrt{Z^{x_{1,1}}_u}d
\beta_u$ is a Bessel process independent of the Gaussian variables starting from~$x_{1,1}$. By a time
scaling, we have $Z^{x_{1,1}}_t \overset{law}{=} t Z^{x_{1,1}/t}_1$, and thus:
  $$ ({\bf p}(Y^x_t)_{1,i})_{2 \le i \le
    d}\overset{law}{=}\frac{\left(\frac{x_{1,2}}{\sqrt{t}}
    +N_2,\dots, \frac{x_{1,d}}{\sqrt{t}}+ N_d \right)}{\sqrt{Z^{x_{1,1}/t}_1+\sum_{i=2}^d
      (\frac{x_{1,i}}{\sqrt{t}} +N_i})^2} \underset{t \rightarrow
    +\infty}{\rightarrow}\frac{\left(N_2,\dots, N_d
    \right)}{\sqrt{Z^{0}_1+\sum_{i=2}^d N_i^2}}. $$
On the other hand, we know that $X^x_t$ converges in law when $t\rightarrow +
\infty$, and Proposition~\ref{WIS_MRC} immediately gives, with the help of Lemma~\ref{lemma_change_time} that $((X^x_\infty)_{1,i})_{2 \le i \le
    d}\overset{law}{=}\frac{\left(N_2,\dots, N_d
    \right)}{\sqrt{Z^{0}_1+\sum_{i=2}^d N_i^2}}$. By simple calculations, we
  get that $((X^x_\infty)_{1,i})_{2 \le i \le
    d}$ has the following density:
  \begin{equation}\label{sqrt_dirichlet} \indi{ \sum_{i=2}^d z_i^2 \le 1}
  \frac{\Gamma\left(\frac{\alpha+1}{2}\right)
  }{(\sqrt{\pi})^{d-1}\Gamma\left(\frac{\alpha+2-d}{2}\right)}
  \left(1-\sum_{i=2}^d z_i^2 \right).
\end{equation}
In particular, we can check that $((X^x_\infty)^2_{1,i})_{2 \le i \le
    d}$ follows a Dirichlet law, which is known as the ergodic limit of
  multi-allele Wright-Fisher models. Last, let us mention that we can get an
  explicit but cumbersome expression of the density of the law $\corrpstb{d}{x}{\frac{\alpha}{2}
  a^2}{I_d}{a}{\infty}$ by
  combining~\eqref{commut_infty}, Proposition~\ref{prop_reduction} and~\eqref{sqrt_dirichlet}.

\section{Existence and uniqueness results for MRC processes}\label{sec_exist}

In this section we show weak and strong existence results for the
SDE~\eqref{SDE_CORR}, respectively under
assumptions~\eqref{cond_weak_existence}
and~\eqref{cond_strong_existence}. These assumptions are of the same nature as the one
known for Wishart processes. To prove the strong existence and uniqueness, we
make assumptions on the coefficients that ensures that $X_t$ remains in the
set of the invertible correlation matrices where the coefficients are locally
Lipschitz. This is similar to the proof given by Bru~\cite{Bru} for 
Wishart processes. Then, we  prove the weak existence by introducing a
sequence of processes defined on~$\corr$, which is tight such that any
subsequence limit solves the martingale problem~\eqref{PbMg}. Next, we extend
our existence results when the parameters are no longer constant. Last, we
exhibit some change of probability that preserves the global dynamics of our
Mean-Reverting Correlation processes.

\subsection{Strong existence and uniqueness}

\begin{theorem}\label{thm_strong} Let $x \in \corri$. We assume
  that~\eqref{cond_strong_existence} holds. Then, there is a unique strong
  solution of the SDE~\eqref{SDE_CORR} that is such that $\forall t\ge 0, X_t
  \in \corri$.
\end{theorem}
\begin{proof}
  By Lemma~\ref{lemma_correlmatrix}, we have 
$\Subm{(\sqrt{x-xe^n_dx})}{n}=\sqrt{\Subm{x}{n}-x^n(x^n)^T}$ and
$\Subm{x}{n}-x^n(x^n)^T \in \mathcal{S}_{d-1}^{+,*}(\R)$ when $x \in \corri$.
For $x \in \dpos$ such that $\Subm{x}{n}-x^n(x^n)^T\in
\mathcal{S}_{d-1}^{+,*}(\R)$, we define $f^n(x)\in \posm$ by $(f^n(x))_{n,j}=0$
for $1\le j\le d$ and
$\Subm{(f^n(x))}{n}=\sqrt{\Subm{x}{n}-x^n(x^n)^T}$. The function $f^n$ is well
defined on an open set of $\symm$ that includes $\corri$, and is such that
$f^n(x)=\sqrt{x-xe^n_dx}$ for $x \in \corri$. Since the square-root of a
positive semi-definite matrix is locally Lipschitz on the positive definite
matrix set, we get that the SDE
$$X_t = x + \int_0^t \left( \kp (c-X_s) + (c-X_s) \kp \right) ds+ \sum_{n=1}^da_n \int_0^t \left( f^n(X_s) dW_se_d^{n}
  +e_d^{n} dW_s^T f^n(X_s) \right), $$
has a unique strong solution for $0 \le t <\tau$, where
$$ \tau= \inf\{t \ge 0, X_t \not \in \dpos \text{ or } \exists i \in
\{1,\dots,d\}, \Subm{X_t}{i}-X_t^i(X_t^i)^T \not \in
\mathcal{S}_{d-1}^{+,*}(\R)\},\ \inf \emptyset=+\infty.$$
For $1\le i\le d$, we have $(f^n(X_s)dW_s
e_d^{n})_{i,i}=\indi{i=n}\sum_{j=1}^d f^n(X_s)_{n,j}(dW_s)_{j,n}=0$ and then:
$$d(X_t)_{i,i}=2\kp_{i,i}(1-(X_t)_{i,i})dt,$$
which immediately gives $(X_t)_{i,i}=1$ for $0\le t < \tau$. Thus, $X_t \in
\corri$ for  $0\le t < \tau$ and $\tau=\inf\{t \ge 0, X_t \not \in \corri \}$
by Lemma~\ref{lemma_correlmatrix}, and the process $X_t$ is solution of~\eqref{SDE_CORR} up to time~$\tau$.
We set $Y_t=\log(\det(X_t))+\Tr(2\kp-a^2)t$. By Lemma~\ref{lemma_det_sde}, we
have
\begin{eqnarray*}
  Y_t&=&Y_0 +\int_0^t \Tr[X_s^{-1}(\kp c+c \kp-d a^2)]ds +2
  \int_0^t\sqrt{\Tr[a^2(X_t^{-1}-I_d)]}d\beta_s\\
  &\ge &Y_0  +2
  \int_0^t\sqrt{\Tr[a^2(X_t^{-1}-I_d)]}d\beta_s,
\end{eqnarray*}
since $\kp c+c \kp-d a^2 \in \posm$ by Assumption~\eqref{cond_strong_existence}. Now, we use the McKean
argument exactly like Bru~\cite{Bru} did for Wishart processes: on
$\{\tau<\infty\}$,  $Y_t \underset{t\rightarrow \tau}{\rightarrow}- \infty$,
and the local martingale
$\int_0^t\sqrt{\Tr[a^2(X_t^{-1}-I_d)]}d\beta_s\underset{t\rightarrow
  \tau}{\rightarrow}- \infty$, which is almost surely not possible. We deduce
that $\tau=+\infty,$ a.s.
\end{proof}

\subsection{Weak existence and uniqueness}\label{sec_weak_ex}

The weak uniqueness has already been obtained in
Proposition~\ref{prop_equ_ODE}, and we provide in this section
a constructive proof of a weak solution to the
SDE~\eqref{SDE_CORR}.   In the case $d=2$, this result is already well-known. In
fact, by Proposition~\ref{prop_el_jac}, the associated martingale problem is
the one of a one-dimensional Wright-Fisher process. For this SDE,
strong (and therefore weak) existence and uniqueness holds since the diffusion
coefficient is $1/2$-Hölderian.

Thus, we can assume without loss of generality that $d\ge 3$. The first step
is to focus on the existence when  $a=diag(a_1,\dots,a_d) \in \posm$,
$\alpha \ge d-2$, $\kp=\frac{\alpha}{2} a^2$ and $c=I_d$. By Proposition~\ref{WIS_MRC}, we know that weak existence
holds for $\corrps{d}{x}{\frac{\alpha}{2}e^1_d}{I_d}{e^1_d}$, and thus for
$\corrps{d}{x}{\frac{\alpha}{2}a_i^2 e^i_d}{I_d}{a_i e^i_d}$ for $i=1,\dots,d$
and $a_i\ge 0$, by using a permutation of the coordinates and a linear
time-scaling. Therefore, by using Proposition~\ref{sim_exact}, the
distribution $\corrpst{d}{x}{\frac{\alpha}{2}a^2}{I_d}{a}$ is also
well-defined on~$\corr$ for any $t\ge 0$.
Let $T>0$ be a time-horizon, $N \in \N^*$, and $t_i^N=iT/N$. We define
$(\hat{X}^N_t,t\in [0,T])$ as follows.
\begin{itemize}
  \item We set $\hat{X}^N_0=x$.
  \item For $i=0,\dots,N-1$, $\hat{X}^N_{t^N_{i+1}}$ is sampled 
according to the law
$\corrpstb{d}{\hat{X}^N_{t^N_{i}}}{\frac{\alpha}{2}a^2}{I_d}{a}{T/N}$, conditionally to $\hat{X}^N_{t^N_{i}}$.
\item For $t\in[t^N_i,t^N_{i+1}]$,
  $\hat{X}^N_t=\frac{t-t^N_i}{T/N}\hat{X}^N_{t^N_{i}}+\frac{t^N_{i+1}-t}{T/N}\hat{X}^N_{t^N_{i+1}}=\hat{X}^N_{t^N_{i}}+
  \frac{t-t^N_i}{T/N}(\hat{X}^N_{t^N_{i+1}}-\hat{X}^N_{t^N_{i}})$.
\end{itemize}
The process $(\hat{X}^N_t,t\in [0,T])$ is continuous and such that almost
surely, $\forall t \in [0,T], \hat{X}^N_t \in \corr$. We endow the set of
matrices with the norm $\|x\|=\left(\sum_{i,j=1}^d
  x_{i,j}^4\right)^{1/4}$. The sequence of processes $(\hat{X}^N_t,t\in [0,T])_{ N\ge 1}$
satisfies the following Kolmogorov tightness criterion.
\begin{lemma}\label{lemme_kolmo}
Under the assumptions above, there is a constant $K>0$ such that:
\begin{equation}\label{kolmo}\forall 0\le s\le t \le T,\  \E[\|\hat{X}^N_t-\hat{X}^N_s\|^4] \le K(t-s)^2.
\end{equation}
\end{lemma}
\begin{proof}
  We first consider the case $s=t^N_k$ and $t=t^N_l$ for some $0\le k\le l\le
  N$.  Then, by Proposition~\ref{sim_exact}, we know that conditionally on
  $\hat{X}^N_{t^N_k}$, $\hat{X}^N_{t^N_l}$ follows the law of
  $\corrps{d}{\hat{X}^N_{t^N_k}}{\frac{\alpha}{2} a^2}{I_d}{a}$. In
  particular, each element~$(\hat{X}^N_{t^N_l})_{i,j}$ follows the marginal
  law of a
  one-dimensional Wright-Fisher process with parameters given by
  equation~\eqref{SDE_elements}. Thus, by Proposition~\ref{moments_jacobi}
  there is a constant still denoted by~$K>0$ such that for any $1\le i,j \le d$,
  $\E[((\hat{X}^N_{t^N_l})_{i,j}-(\hat{X}^N_{t^N_k})_{i,j})^4] \le K (t^N_l-t^N_k)^2$,
  and therefore
  $$\E[\| \hat{X}^N_{t^N_l}-\hat{X}^N_{t^N_k}\|^4] \le K d^2  (t^N_l-t^N_k)^2.$$
  Let us consider now $0\le s\le t \le T$. If there exists $0\le k\le N-1$,
  such that $s,t\in[t^N_k,t^N_{k+1}]$, then  $\E[\|\hat{X}^N_t-\hat{X}^N_s\|^4]=
  \left(\frac{s-t}{T/N}\right)^4 \E[\|
  \hat{X}^N_{t^N_{k+1}}-\hat{X}^N_{t^N_k}\|^4]\le K d^2 (s-t)^2$. Otherwise,
  there are $k\le l$ such that $t^N_k-T/N<s\le t^N_k \le t^N_l\le t<t^N_l+T/N$,
  and $\E[\|\hat{X}^N_t-\hat{X}^N_s\|^4]\le Kd^2[(t^N_k-s)^2+(t-t^N_l)^2+(t^N_l-t^N_k)^2]\le K'(t-s)^2$ for some constant $K'>0$. 
\end{proof}

The sequence $(\hat{X}^N_t,t\in [0,T])_{ N\ge 1}$ is tight, and we will show
that any limit of subsequence solves the martingale
problem~\eqref{PbMg}. More precisely, we will show that for any $n\in \N^*$, $0\le t_1\le\dots \le t_n\le
s\le t\le T$, $g_1,\dots, g_n \in
\mathcal{C}(\symm,\R)$, $f\in\mathcal{C}^\infty(\symm,\R)$ we have:
\begin{equation}\label{lim_pbmg}
\lim_{N\rightarrow +\infty} \E\left[ \prod_{i=1}^n g_i(\hat{X}^N_{t_i})\left(f(\hat{X}^N_t)-f(\hat{X}^N_s)-\int_s^t Lf(\hat{X}^N_u)du
  \right) \right]=0.
\end{equation}
We set $k^N(s)$ and $l^N(t)$ the indices such that $t^N_{k^N(s)}-T/N<s\le
t^N_{k^N(s)}$ and $t^N_{l^N(t)}\le t<t^N_{l^N(t)}+T/N$. Clearly, $f$ is
Lipschitz and $Lf$ is bounded on~$\corr$. It is therefore sufficient to show that
\begin{equation}\label{lim_pbmg2}
\lim_{N\rightarrow +\infty} \E\left[ \prod_{i=1}^n g_i(\hat{X}^N_{t_i})\left(f(\hat{X}^N_{t^N_{l^N(t)}})-f(\hat{X}^N_{t^N_{k^N(s)}})-\int_{t^N_{k^N(s)}}^{t^N_{l^N(t)}} Lf(\hat{X}^N_u)du
  \right) \right]=0.
\end{equation}
We decompose the expectation as the sum of 
\begin{equation}\label{PbMg_intermed1}
  \E\left[ \prod_{i=1}^n g_i(\hat{X}^N_{t_i}) \int_{t^N_{k^N(s)}}^{t^N_{l^N(t)}}
  (Lf(\hat{X}^N_{t^N_{l^N(u)}})- Lf(\hat{X}^N_u))du
\right]+\E\left[ \prod_{i=1}^n g_i(\hat{X}^N_{t_i}) \left(
    \sum_{j=k^N(s)}^{l^N(t)-1}
    f(\hat{X}^N_{t^N_{j+1}})-f(\hat{X}^N_{t^N_{j}})-\frac{T}{N}Lf(\hat{X}^N_{t^N_{j}}) \right)
\right]
\end{equation}
To get that the first expectation goes to~$0$, we claim that:
\begin{equation}\label{PbMg_intermed2}\E\left[ \int_{t^N_{k^N(s)}}^{t^N_{l^N(t)}} |\beta(u,\hat{X}^N_u)-
\beta(t^N_{l^N(u)},\hat{X}^N_{t^N_{l^N(u)}})|du \right]\rightarrow 0
\end{equation}
when $\beta:(t,x) \in [0,T] \times \corr \rightarrow \R$ is continuous. This
formulation will be reused later on. By Lemma~\ref{lemme_kolmo},
\eqref{PbMg_intermed2} holds when $\beta$ is Lipschitz with respect
to~$(t,x)$. If $\beta$ is not Lipschitz, we can still approximate it uniformly on the
compact set $[0,T]\times\corr$ by using for example the Stone-Weierstrass
theorem, which gives~\eqref{PbMg_intermed2}.

On the other hand,
we know by~\eqref{controle_erreur} that the second expectation  goes
to~$0$. To be precise,~\eqref{controle_erreur} has been obtained by using
It\^o's formula while we do not know yet at this stage that the process
$\corrps{d}{x}{\frac{\alpha}{2}a^2}{I_d}{a}$
exists. It is nevertheless true: \eqref{controle_erreur}~holds for
$\corrps{d}{x}{\frac{\alpha}{2}a_i^2e^i_d}{I_d}{e^i_d}$ since this process is
already known to be well
defined, and we get by using Proposition~\ref{sim_exact} and Proposition~\ref{prop_compo_schemas} that
$\exists K>0,
|f(\hat{X}^N_{t^N_{j+1}})-f(\hat{X}^N_{t^N_{j}})-(T/N)Lf(\hat{X}^N_{t^N_{j}})|\le
K/N^2$.
Thus, $(\hat{X}^N_t,t\in [0,T])_{ N\ge 1}$ converges in law to a solution
of the martingale problem~\eqref{PbMg}. This concludes the existence of
$\corrps{d}{x}{\frac{\alpha}{2}a^2}{I_d}{a}$.

Now, we are in position to show the existence of
$\corrps{d}{x}{\kp}{c}{a}$ under Assumption~\eqref{cond_weak_existence}. We denote by $\xi(t,x)$ the solution
to the linear ODE:
\begin{equation}\label{def_xi}
  \xi'(t,x)=\kp(c-x)+(c-x)\kp -\frac{d-2}{2}[a^2(I_d-x)+(I_d-x)a^2], \
  \xi(0,x)=x \in \corr .
\end{equation}
By Lemma~\ref{lemm_somme_ODE}, we know that $\forall t\ge 0,  \xi'(t,x)\in
\corr$. It is also easy to check that:
$$\exists K>0, \forall x\in \corr, \ \|\xi(t,x)-x\| \le Kt.$$
Now, we define $(\hat{X}^N_t,t\in [0,T])$ as follows.
\begin{itemize}
  \item We set $\hat{X}^N_0=x\in \corr$.
  \item For $i=0,\dots,N-1$, $\hat{X}^N_{t^N_{i+1}}$ is sampled 
according to $\corrpstb{d}{\xi(T/N,
  \hat{X}^N_{t^N_{i}})}{\frac{d-2}{2}a^2}{I_d}{a}{T/N}$, conditionally to
$\hat{X}^N_{t^N_{i}}$. More precisely, we denote by $(\bar{X}^N_t, t\in
    [t^N_i,t^N_{i+1}])$ a solution to
    \begin{eqnarray*}
      \bar{X}^N_t &=&\xi(T/N,\hat{X}^N_{t^N_i})+\frac{d-2}{2}\int_{t^N_i}^t \left[
      a^2(I_d-\bar{X}^N_u)+(I_d-\bar{X}^N_u)a^2\right]du \\
      && +\sum_{n=1}^da_n \int_{t^N_i}^t \left( \sqrt{ \bar{X}^N_u- \bar{X}^N_u e_d^{n} \bar{X}^N_u} dW_ue_d^{n}
  +e_d^{n} dW_u^T \sqrt{ \bar{X}^N_u- \bar{X}^N_u e_d^{n}  \bar{X}^N_u}\right),
\end{eqnarray*}
and we set $\hat{X}^N_{t^N_{i+1}}=\bar{X}^N_{t^N_{i+1}}$.

\item For $t\in[t^N_i,t^N_{i+1}]$,
  $\hat{X}^N_t=\hat{X}^N_{t^N_{i}}+
  \frac{t-t^N_i}{T/N}(\hat{X}^N_{t^N_{i+1}}-\hat{X}^N_{t^N_{i}})$.
\end{itemize}
We proceed similarly and show that the Kolmogorov
criterion~\eqref{kolmo} holds for $(\hat{X}^N_t,t\in[0,T])_{N\ge 1}$. As already shown in Lemma~\ref{lemme_kolmo}, it is
sufficient to check that this criterion holds for $s=t^N_k\le t=t^N_l$. We have
\begin{eqnarray*}
  \|\hat{X}^N_{t^N_l}-\hat{X}^N_{t^N_k} \|^4 &=&\|
  \sum_{j=k}^{l-1}\hat{X}^N_{t^N_{j+1}}-\xi(T/N,\hat{X}^N_{t^N_j})+\xi(T/N,\hat{X}^N_{t^N_j})-\hat{X}^N_{t^N_j}
  \|^4 \\
  &\le &2^3 \left( \|
  \sum_{j=k}^{l-1}\bar{X}^N_{t^N_{j+1}}-\bar{X}^N_{t^N_{j}}\|^4
  +(l-k)^4\left(\frac{KT}{N}\right)^4 \right).
\end{eqnarray*}
Since $(\bar{X}^N_t,t\in[0,T])$ is valued in the compact set~$\corr$, we get
easily by using Burkholder-Davis-Gundy inequality that $\E[ \|
  \sum_{j=k}^{l-1}\bar{X}^N_{t^N_{j+1}}-\bar{X}^N_{t^N_{j}}\|^4]\le K
  (t_l-t_k)^2$ and then $\E [\|\hat{X}^N_{t^N_l}-\hat{X}^N_{t^N_k} \|^4]\le
  K(t_l-t_k)^2$ for some constant $K>0$ that does not depend on~$N$. 

Thus, $(\hat{X}^N_t,t\in[0,T])_{N\ge 1}$ satisfies the Kolmogorov criterion
and is tight. It remains to show that any subsequence converges in law to the
solution of the martingale problem~\eqref{PbMg}. We proceed as before and
reuse the same notations. From~\eqref{PbMg_intermed1}, it is sufficient to
show that
$$ \exists K>0,
|f(\hat{X}^N_{t^N_{j+1}})-f(\hat{X}^N_{t^N_{j}})-(T/N)Lf(\hat{X}^N_{t^N_{j}})|\le
K/N^2.$$
Once again, we cannot directly use~\eqref{controle_erreur} since we do not
know at this stage that $\corrps{d}{x}{\kp}{c}{a}$
exists. We have $L=L^{\xi}+\tilde{L}$, where $L^{\xi}$ is the operator
associated to~$\xi(t,x)$ and $\tilde{L}$ is the infinitesimal generator of
$\corrps{d}{x}{\frac{d-2}{2}a^2}{I_d}{a}$. We have: $\exists K>0,\forall x \in
\corr, |f(\xi(t,x))-f(x)-tL^{\xi}f(x)|\le Kt^2$, and~\eqref{controle_erreur}
holds for~$\tilde{L}$. By Proposition~\ref{prop_compo_schemas}, we get: $\exists K>0,\forall x \in
\corr, |f(\xi(t,x))-f(x)-tf(x)|\le Kt^2$, which gives~\eqref{lim_pbmg} and
concludes the proof of the weak existence.

\begin{theorem}\label{thm_we}
Under assumption~\eqref{cond_weak_existence}, there is a unique weak solution
$(X_t,t\ge 0)$
to SDE~\eqref{SDE_CORR} such that $\Px (\forall t \ge 0, X_t \in \corr)=1$.
\end{theorem}

\begin{remark} Assumption~\eqref{cond_weak_existence} has only be used in the
  proof of Theorem~\ref{thm_we} to ensure that $\xi$ defined by~\eqref{def_xi} satisfies
  \begin{equation}\label{cond_weak_existence_bis}
    \forall t\ge 0 ,x \in \corr, \   \xi(t,x) \in \corr.
\end{equation}
As pointed by Remark~\ref{rem_CNS}, this is a sufficient but
  not necessary condition. In fact, a weak solution of~\eqref{SDE_CORR} exists
  under~\eqref{cond_weak_existence_bis}, which is more general but less
  tractable condition than~\eqref{cond_weak_existence}.

\end{remark}

\subsection{Extension to non-constant coefficients}

In this paragraph, we consider the SDE~\eqref{SDE_CORR} with time and space
dependent coefficients:
\begin{eqnarray}\label{SDE_CORR_coefgen}
X_t &=& x + \int_0^t \left[ \kp(s,X_s) (c(s,X_s)-X_s) + (c(s,X_s)-X_s)
  \kp(s,X_s) \right] ds
\\&&+ \sum_{n=1}^da_n(s,X_s) \int_0^t \left( \sqrt{X_s-X_s e_d^{n} X_s} dW_se_d^{n}
  +e_d^{n} dW_s^T \sqrt{X_s-X_s e_d^{n} X_s} \right), \nonumber
\end{eqnarray}
where $\kp(t,x)$, $c(t,x)$ and $a(t,x)$ are measurable functions such that for
any $t\ge 0$ and $x \in \corr$,
$\kp(t,x)$ and $a(t,x)$ are nonnegative diagonal matrices and $c(t,x)\in
\corr$. Then, under the following assumption
\begin{eqnarray} \label{cond_strong_existence_gen}
&&\hspace{-0.5cm}  \forall T>0, \sup_{t\in [0,T]}|\kp(t,I_d)|<\infty, \  \forall t\in[0,T], \exists K>0,
  \|f(t,x)-f(t,y)\|\le K \|x-y\| \text{ for } f\in \{\kp,c,a\} ,\nonumber \\
&&\hspace{-0.5cm}  \forall t \ge 0, x \in \corr, \ \kp(t,x) c(t,x)+c(t,x) \kp(t,x) -d a^2(t,x) 
\in \posm \text{ and } X_0 \in \corri,
\end{eqnarray}
strong existence and uniqueness holds for~\eqref{SDE_CORR_coefgen}. To get
this result, we observe that ${\bf p}(x)$ is Lipschitz on $\{ x \in \posm \
s.t. \
 \forall 1\le i \le d, 1/2 \le x_{i,i} \le 2 \}$. Therefore,
 the SDE $X_t = x + \int_0^t \big( \kp(s,{\bf p}(X_s)) [c(s,{\bf p}(X_s))-X_s]$ $+
   [c(s,{\bf p}(X_s))-X_s] \kp(s,{\bf p}(X_s)) \big) ds+ \sum_{n=1}^d
 \int_0^t a_n(s,{\bf p}(X_s)) \left( f^n(X_s) dW_se_d^{n}
  +e_d^{n} dW_s^T f^n(X_s) \right) $ has a unique solution up to time
$ \tau= \inf\{t \ge 0, X_t \not \in \dpos \text{ or } \exists i \in
\{1,\dots,d\}, \Subm{X_t}{i}-X_t^i(X_t^i)^T \not \in
\mathcal{S}_{d-1}^{+,*}(\R) \text{ or } (X_t)_{i,i} \not \in [1/2,2] \},$ and we proceed then
exactly as for the proof of Theorem~\ref{thm_strong}.

Also, weak existence holds for~\eqref{SDE_CORR_coefgen} if we assume that:
\begin{eqnarray} \label{cond_weak_existence_gen}
&&  \kp(t,x), c(t,x), a(t,x) \text{ are continuous on } \R_+ \times \corr
\nonumber \\
&& \forall t \ge 0, x \in \corr, \ \kp(t,x) c(t,x)+c(t,x) \kp(t,x) -(d-2) a^2(t,x) 
\in \posm.
\end{eqnarray}
To get this result, we proceed as in Section~\ref{sec_weak_ex} and define
$(\hat{X}^N_t,t\in[0,T])$ as follows.
\begin{itemize}
  \item We set $\hat{X}^N_0=x$.
  \item For $i=0,\dots,N-1$,  we denote by $(\bar{X}^N_t, t\in
    [t^N_i,t^N_{i+1}])$ a solution to
    \begin{eqnarray*}
      \bar{X}^N_t &=&\hat{X}^N_{t^N_i}+\int_{t^N_i}^t  \left[
      \kp(t^N_i,\hat{X}^N_{t^N_i}) (c(t^N_i,\hat{X}^N_{t^N_i})-\bar{X}^N_u)+(c(t^N_i,\hat{X}^N_{t^N_i})-\bar{X}^N_u)\kp(t^N_i,\hat{X}^N_{t^N_i})\right]du \\
      && +\sum_{n=1}^da_n(t^N_i,\hat{X}^N_{t^N_i}) \int_{t^N_i}^t \left( \sqrt{ \bar{X}^N_u- \bar{X}^N_u e_d^{n} \bar{X}^N_u} dW_ue_d^{n}
  +e_d^{n} dW_u^T \sqrt{ \bar{X}^N_u- \bar{X}^N_u e_d^{n}  \bar{X}^N_u}\right),
\end{eqnarray*}
and we set $\hat{X}^N_{t^N_{i+1}}=\bar{X}^N_{t^N_{i+1}}$.
\item For $t\in[t^N_i,t^N_{i+1}]$,
  $\hat{X}^N_t=\hat{X}^N_{t^N_{i}}+
  \frac{t-t^N_i}{T/N}(\hat{X}^N_{t^N_{i+1}}-\hat{X}^N_{t^N_{i}})$.
\end{itemize}
We can check that $(\hat{X}^N_t,t\in[0,T])$ satisfies the Kolmogorov criterion
and is tight. To obtain~\eqref{lim_pbmg}, we proceed as in
Section~\ref{sec_weak_ex}. More precisely, let us denote for $u\in [0,T]$
$L_u$ the infinitesimal generator of~\eqref{SDE_CORR_coefgen}, and $\hat{L}_u$
the infinitesimal generator with frozen coefficient at
$(t_i,\hat{X}^N_{t^N_{i}})$ when $u\in[t^N_{i},t^N_{i+1})$.
In~\eqref{PbMg_intermed1}, the first term $
 \E\left[ \prod_{i=1}^n g_i(\hat{X}^N_{t_i}) \int_{t^N_{k^N(s)}}^{t^N_{l^N(t)}}
  (\hat{L}_uf(\hat{X}^N_{t^N_{l^N(u)}})- L_uf(\hat{X}^N_u))du
\right]\rightarrow 0$ thanks to~\eqref{PbMg_intermed2}, and the second term
goes to~$0$ as before.

To sum up, it is rather easy to extend our results of strong existence and uniqueness, and
weak existence when the coefficients are not constant. However, we can no longer get explicit formulas
for the moments in this case. Thus, if the coefficients
satisfy~\eqref{cond_weak_existence_gen} but
not~\eqref{cond_strong_existence_gen}, the weak uniqueness remains an open  
question, which is beyond the scope of this paper. 

\subsection{A Girsanov Theorem}

In this section, we will use an alternative writing of the
SDE~\eqref{SDE_CORR}. In fact, by Lemma~\eqref{calcul_crochet}, the SDE
\begin{equation}\label{SDE_CORR_alt}
X_t = x + \int_0^t \left( \kp (c-X_s) + (c-X_s) \kp \right) ds+ \sum_{n=1}^da_n \int_0^t \left( h_n(X_s) dW_se_d^{n}
  +e_d^{n} dW_s^T h_n(X_s)^T \right),
\end{equation}
is associated to the same martingale problem as~$\corrp$ for any functions
$h_n:\symm\rightarrow \genm$ such that $h_n(x)h_n(x)^T=x-x e^n_d x$ for $x \in \corr$. In this
paper, we have arbitrarily decided to take the symmetric version
$h_n(x)=\sqrt{x-x e^n_d x}$. Obviously, other choices are possible. An
interesting choice is the following one:
\begin{equation}\label{choix_hn}x \in \posm,
h_n(x)=\sqrt{x}\sqrt{I_d-\sqrt{x}e^n_d\sqrt{x}}=\sqrt{x}(I_d-\sqrt{x}e^n_d\sqrt{x}),
\end{equation}
where the second equality comes from
Lemma~\ref{lemma_correlmatrix2}. Obviously, our weak existence and uniqueness
results (Theorem~\ref{thm_we}) applies to~\eqref{SDE_CORR_alt} since
\eqref{SDE_CORR} and~\eqref{SDE_CORR_alt} solve the same martingale problem. However, we have
to show again that strong uniqueness holds for~\eqref{SDE_CORR_alt} under
Assumption~\eqref{cond_strong_existence} and $x\in \corri$. The proof is in fact very similar to
Theorem~\ref{thm_strong}. We know that there is one strong solution to 
$ X_t = x + \int_0^t \left( \kp (c-X_s) + (c-X_s) \kp \right) ds$ $+ \sum_{n=1}^da_n \int_0^t \left(\sqrt{X_s}(I_d-\sqrt{X_s}e^n_d\sqrt{X_s})   dW_se_d^{n}
  +e_d^{n}(I_d-\sqrt{X_s}e^n_d\sqrt{X_s}) \sqrt{X_s} dW_s^T  \right) $
up to time $\tau=\inf \{t\ge 0, X_t \not \in \posm \}$. On $t\in [0,\tau)$,
there are real Brownian motions $\beta^i_t$ such
that $$d(X_t)_{i,i}=2\kp_{i}(1-(X_t)_{i,i})dt+2 a_i
(1-(X_t)_{i,i})\sqrt{(X_t)_{i,i}}d\beta^i_t,$$
which gives $(X_t)_{i,i}=1$ by strong uniqueness of this SDE. We then conclude
as in the proof of Theorem~\ref{thm_strong} and get in particular that $X_t \in \corri$ for
$t\ge 0$.

We consider now a solution to~\eqref{SDE_CORR_alt}, and a progressively measurable process $(H_s)_{s \geq 0},$ valued in $\genm,$ such that 
\vspace{-3mm}\begin{equation}\label{mg_exp}
 \mathcal{E}_t^H = \exp\left(\int_0^t \Tr(H_s^TdW_s)-\frac{1}{2}\int_0^t \Tr(H_s^TH_s)ds\right)
\end{equation}
 is a martingale. For a given time horizon~$T>0$, we
denote by $\Q$ the probability measure, if it exists, defined as 
\vspace{-0.5mm}\begin{equation}\label{def_Q} 
 \frac{d \mathbb{Q}}{d \Px} \Big|_{\mathcal F_T} =  \mathcal{E}_T^H,
\end{equation}
where $(\mathcal F_t)_{t \geq 0}$ is the natural filtration of the process $(X_t)_{t \geq 0}.$ 
Then, $W^{\Q}_t = W_t - \int_0^t H_s ds$ is a $d\times d$ Brownian matrix
under~$\Q$, and the process $(X_t)_{t \geq 0}$ satisfies
\begin{eqnarray}\label{SDE_under_new_measure}
X_t&=& x+ \int_0^t \left( \kappa(c-X_s) + (c-X_s)\kappa  \right)ds\\ \nonumber
&& +   \int_0^t \left(\sum_{i=1}^da_i\left\lbrace \sqrt{X_s}\left[{I_d-\sqrt{X_s}e_d^i\sqrt{X_s}}\right]H_se_d^i + e_d^iH_s^T\left[{I_d-\sqrt{X_s}e_d^i\sqrt{X_s}}\right]\sqrt{X_s} \right \rbrace \right)ds\\ \nonumber
&& + \sum_{i=1}^da_i\int_0^t \left( \sqrt{X_s}\left[{I_d-\sqrt{X_s}e_d^i\sqrt{X_s}}\right]dW^{\Q}_se_d^i + e_d^id(W^{\Q}_s)^T\left[{I_d-\sqrt{X_s}e_d^i\sqrt{X_s}}\right]\sqrt{X_s} \right). \nonumber
\end{eqnarray}
We present now changes of probability such that $(X_t,t\ge 0)$ is also a
mean-reverting correlation process under~$\Q$.

\begin{proposition}
We assume~\eqref{cond_weak_existence}. We consider $(X_t,t\ge 0)\sim \corrp$
and take $H_t=\sqrt{X_t} \lambda$, with
$\lambda=diag(\lambda_1,\dots,\lambda_d) \in \symm$. Then, \eqref{mg_exp} is a
martingale and $(X_t,t\ge 0)\sim \corrp$ under $\Q$.
\end{proposition}
\begin{proof}
Since the process $(X_t,t\ge 0)$ is bounded,  \eqref{mg_exp} is clearly a
martingale.  For $y \in \corr$, $e^i_d y e^i_d=e^i_d$ and  we have
$\left(\sqrt{y}{(I_d-\sqrt{y}e_d^i\sqrt{y})}\right)\sqrt{y} \lambda e_d^i=
\lambda_i {(y-ye_d^iy)} e_d^i=0$, which gives the result by~\eqref{SDE_under_new_measure}.
\end{proof}

\begin{proposition}\label{prop_Girsanov}
Let $x\in \corri$. We consider $(X_t,t\ge 0)\sim
\corrps{d}{x}{\kp^1}{c^1}{a}$ and assume that $\kp^1,c^1,a$
satisfy~\eqref{cond_strong_existence}. Let $c^2\in \corr$ and $\kp^2$ be a
real diagonal matrix such that $a_i=0\implies \kp^2_i=0$ and $\kp^1
c^1+c^1\kp^1+ \kp^2 c^2+c^2\kp^2 -da^2 \in \posm$. We set:
$$\lambda=diag(\lambda_1,\dots,\lambda_d) \text{ with }
\lambda_i=\begin{cases}\kp^2_i/a_i \text{ if } a_i>0 \\
  0 \text{ otherwise } \end{cases} \text{ and }H_t=(\sqrt{X_t})^{-1} c^2\lambda
$$
This defines with~\eqref{mg_exp} and~\eqref{def_Q} a change of probability such that
$$(X_t,t\ge 0)\sim \corrp \text{ under }\Q,$$
where $\kp=diag(\kp_1,\dots,\kp_d)\in \posm$ and $c\in \corr$ are defined as
in Lemma~\ref{lemm_somme_ODE}.
\end{proposition}
\begin{proof}
We have $a^1_i\sqrt{y}(I_d-\sqrt{y}e^i_d \sqrt{y})\sqrt{y}^{-1}c^2 \lambda e^i_d=
\kp^2_i(c^2 e^i_d -ye^i_dc^2e^i_d)= \kp^2_i(c^2-y)e^i_d$, which gives the
claim by~\eqref{SDE_under_new_measure}, provided that $\E[\mathcal{E}^H_T]=1
$ for any $T>0$. We prove now this martingale property with an argument already used in
Rydberg~\cite{Rydberg} and Cheridito, Filipovic, and Yor~(\cite{Yor5}, Theorem 2.4).

Let $(X_t,t\ge 0)$ (resp. $(\bar{X}_t,t\ge 0)$) be a strong solution
to~\eqref{SDE_CORR_alt} with parameters $\kp^1,\ c^1,\ a$ (resp. $\kp,\ c,\
a$) and Brownian motion~$(W_t,t\ge 0)$.
For $\varepsilon>0$, we define:
$$\tau^\varepsilon=\inf\{t\ge 0, \det(X_t)\le\varepsilon \}, \ H_t^\varepsilon=\indi{\tau^\varepsilon \ge
  t}(\sqrt{X_t})^{-1} c^2\lambda.$$ We have $\lim_{\varepsilon \rightarrow
  0^+}\tau^\varepsilon=+\infty,$ a.s. and therefore
$$\E[\mathcal{E}_T^H ]=\lim_{\varepsilon \rightarrow 0}\E[\mathcal{E}_T^H
\indi{\tau^\varepsilon \ge T} ].$$
On the other hand, we have $\E[\mathcal{E}_T^H \indi{\tau^\varepsilon \ge T} ]=\E[\mathcal{E}_T^{H^\varepsilon}
\indi{\tau^\varepsilon \ge T} ]$. We clearly have
$\E[\mathcal{E}_T^{H^\varepsilon}]=1$ and
$W^\varepsilon_t=W_t-\int_0^tH^\varepsilon_sds$ is a Brownian motion under
$\frac{d\Q^\varepsilon}{d\Px}=\mathcal{E}_T^{H^\varepsilon}$. Let
$(\bar{X}^\varepsilon_t,t\in[ 0,T])$ be the strong solution
to~\eqref{SDE_CORR_alt} with the Brownian motion~$W^\varepsilon_t$ and
parameters $\kp,\ c,\ a$. By construction, $\bar{X}^\varepsilon_t=X_t$ for $0\le t\le T \wedge
\tau^\varepsilon$ and thus $\indi{\tau^\varepsilon\ge
  T}=\indi{\bar{\tau}^\varepsilon\ge T}$, where $\bar{\tau}^\varepsilon=\inf\{t\ge 0, \det(\bar{X}^\varepsilon_t)\le
\varepsilon \}$. We deduce that $\E[\mathcal{E}_T^{H^\varepsilon}
\indi{\tau^\varepsilon \ge T} ]=\Q^\varepsilon(\bar{\tau}^\varepsilon\ge
T)=\Px(\inf\{t\ge 0, \det(\bar{X}_t) \le \varepsilon \}\ge T)\underset{\varepsilon
  \rightarrow 0^+}{\rightarrow} 1$, since $\kp c+c\kp-da^2=\kp^1
c^1+c^1\kp^1+ \kp^2 c^2+c^2\kp^2 -da^2 \in \posm$.
\end{proof}

Let us assume now that $a_i>0$ for any $1\le i \le d$.
A consequence of Proposition~\ref{prop_Girsanov} is that the probability
measures induced by $\corrps{d}{x}{\kp}{c}{a}$ and $\corrps{d}{x}{\kp'}{c'}{a}$
are equivalent as soon as~\eqref{cond_strong_existence} holds for $\kp, c, a$
and $\kp', c', a$.  By transitivity, it is in fact sufficient to check this
for $\kp'=\frac{d}{2}a^2$ and $c'=I_d$. By Lemma~\ref{lemm_somme_ODE}, there
is a diagonal nonnegative matrix $\tilde{\kp}$ and $\tilde{c} \in \corr$ such
that $\tilde{\kp}\tilde{c}+\tilde{c}\tilde{\kp}=\kp c+c\kp-da^2$. We get then
the probability equivalence by using twice Proposition~\ref{prop_Girsanov}
with $\kp^1=\frac{d}{2}a^2$, $c^1=I_d$, $\kp^2=\tilde{\kp}$, $c^2=\tilde{c}$
and $\kp^1=\kp$, $c^1=c$, $\kp^2=-\tilde{\kp}$, $c^2=\tilde{c}$.

\section{Second order discretization schemes for MRC processes}\label{Sec_simu}

In the previous sections, we focused on the existence of Mean-Reverting
Correlation processes~\eqref{SDE_CORR} and some of their mathematical
properties. From a practical perspective, it is also very important to be able
to sample such processes. By sampling, we mean here that we have an
algorithm to generate the process on a given time-grid. 
Through this section, we will consider for sake of simplicity a regular time
grid $t^N_i=iT/N, \ i=0,\dots,N$ for a given time horizon~$T>0$.
Despite our investigations, the sampling of the exact distribution does not
seem trivial, and we will focus on discretization schemes. Anyway,
discretization schemes are in practice equally or more efficient than exact
sampling, at least in the case of  square-root diffusions such as
Cox-Ingersoll-Ross process and Wishart process (see respectively~\cite{Alfonsi}
and~\cite{AA}).
First, let us say that usual schemes such as Euler-Maruyama fail to be defined for~\eqref{SDE_CORR}
as well as for other square-root diffusions. Indeed, this scheme is given by
\begin{eqnarray}\label{Euler}
  \hat{X}^N_{t^N_{i+1}} &=& \hat{X}^N_{t^N_{i}} + \left( \kp (c-\hat{X}^N_{t^N_{i}}) + (c-\hat{X}^N_{t^N_{i}}) \kp
  \right) \frac{T}{N} \\
  &&+ \sum_{n=1}^da_n  \left( \sqrt{\hat{X}^N_{t^N_{i}}-\hat{X}^N_{t^N_{i}} e_d^{n}
      \hat{X}^N_{t^N_{i}}} (W_{t^N_{i+1}}-W_{t^N_{i}}) e_d^{n}
  +e_d^{n} (W_{t^N_{i+1}}-W_{t^N_{i}})^T \sqrt{\hat{X}^N_{t^N_{i}}-\hat{X}^N_{t^N_{i}} e_d^{n} \hat{X}^N_{t^N_{i}}} \right).\nonumber
\end{eqnarray}
Thus, even if $\hat{X}^N_{t^N_{i}}\in \corr$, $\hat{X}^N_{t^N_{i+1}}$ can no longer be
in~$\corr$ and the matrix square-root can no longer be defined at the next
time-step. A possible correction is to consider the following modification of the Euler
scheme:
\begin{equation}\label{Euler_mod}\hat{X}^N_{t^N_{i+1}}={\bf p}((\tilde{X}_{t^N_{i+1}})^+) ,
\end{equation}
where $\tilde{X}^N_{t^N_{i+1}}$ denotes the right hand side
of~\eqref{Euler}. Here, $x^+\in \posm$ is defined for $x\in \symm$ as the
unique symmetric semidefinite matrix that shares the same eigenvectors as~$x$,
but the eigenvalues are the positive part of the one of~$x$. Namely, $x^+=o
diag(\lambda_1^+,\dots,\lambda_d^+)o$  for $x \in \symm$ such that $
x=o diag(\lambda_1,\dots,\lambda_d)o$ where $o$ is an orthogonal matrix.  Let us check that this scheme is well defined if we start
from~$\hat{X}^N_{t^N_{0}}\in \corr$. By Lemma~\ref{lemma_correlmatrix}, the
square-roots are well defined, we have $(\tilde{X}_{t^N_{1}})_{i,i}=1$ and thus
$(\tilde{X}_{t^N_{1}})^+_{i,i} \ge 1$ and ${\bf p}((\tilde{X}_{t^N_{1}})^+)$
is well defined. By induction, this modified Euler scheme is always defined and
takes values in the set of correlation matrices. However, as we will see in
the numerical experiments, it is time-consuming and converges rather slowly.

In this section, we present discretization schemes that are obtained by
composition, thanks to a splitting of the infinitesimal generator. This
technique has already been used for square-root type diffusions such as the
Cox-Ingersoll-Ross model~\cite{Alfonsi} and Wishart processes~\cite{AA}, leading to
accurate schemes. The strength of this approach is that we can, by an ad-hoc
splitting of the operator, decompose the sampling of the whole diffusion into
pieces that are more tractable and that we can simulate by preserving the
domain (here, the set of correlation matrices). Besides, it is really easy to
analyze the weak error of these schemes.

\subsection{Some results on the weak error of discretization schemes}
We present now the main results on the splitting technique that can be found
in~\cite{Alfonsi} and~\cite{AA} for the framework of Affine diffusions. Here,
we have in addition further simplifications that comes from the fact that the
domain that we consider $\D\subset \R^\zeta$  is compact (typically $\corr$ or
$\D=\{x\in \R^{d-1}, \sum_{i=1}^{d-1} x_i^2 \}$ in Appendix~\ref{schema_original}). For
$\gamma \in \N^\zeta$, we set $\partial_\gamma
f=\partial^1_{\gamma_1}\dots\partial^\zeta_{\gamma_\zeta}$ and
$|\gamma|=\sum_{i=1}^\zeta \gamma_i$. We denote
by $\mathcal{C}^\infty(\D)$ the set of infinitely differentiable functions
on~$\D$ and say that that $(C_\gamma)_{\gamma \in \N^\zeta}$ is a {\it good
  sequence} for~$f\in \mathcal{C}^\infty(\D)$ if we have $\max_{x \in \D}
|\partial_\gamma f(x)|\le C_\gamma$. A differential
operator $Lf(x)=\sum_{0<|\gamma|\le 2}a_\gamma(x) \partial_\gamma f(x)$
satisfies the {\it  required assumption} if we have $a_\gamma \in
\mathcal{C}^\infty(\D)$ for any~$\gamma$. This property if of course satisfied
by the infinitesimal generator~\eqref{EQUATION_OPERATOR_C} of~$\corrp$ since
the functions $a_\gamma$ are either affine or polynomial functions of second
degree. Since we are considering Markovian processes on~$\D$, we will by a slight
abuse of notation represent a discretization scheme by a probability measure
$\hat{p}_x(t)(dz)$ on~$\D$ that describes the law of the scheme starting
from~$x\in\D$ with a time step~$t>0$. Also, we denote by $\hat{X}^x_t$ a random
variable that follows this law. Then, the discretization scheme on the full
time grid~$(t^N_i,i=0,\dots,N)$ will be obtained by:
\begin{itemize}
\item $\hat{X}^N_{t^N_0}=x \in \D$,
\item conditionally to~$\hat{X}^N_{t^N_i}$, $\hat{X}^N_{t^N_{i+1}}$ is sampled
  according to the probability law $\hat{p}_{\hat{X}^N_{t^N_i}}(T/N)(dz)$, and
  we write with a slight abuse of notation $\hat{X}^N_{t^N_{i+1}}=\hat{X}^{\hat{X}^N_{t^N_i}}_{T/N}$.
\end{itemize}
A discretization scheme $\hat{X}^x_t$ is said to be a {\it potential $\nu$-th
  order scheme for the operator~$L$} if for a sequence $(C_\gamma)_{\gamma \in
  \N^\zeta} \in (\R_+)^{\N^\zeta}$, there are constants $C, \eta>0$ such that for
any function $f\in\mathcal{C}^\infty(\D)$ that admits  $(C_\gamma)_{\gamma \in
  \N^\zeta}$ as a good sequence, we have:
\begin{equation}\forall t \in (0, \eta),x \in \D \ \left|  \E[f(\hat{X}^x_t)]-\left[ f(x)+ \sum_{k=1}^\nu
  \frac{1}{k!}t^k L^k
  f(x)\right] \right| \le Ct^{\nu+1}.\label{def_potential}
\end{equation}
  This is the main assumption that a discretization scheme
should satisfy to get a weak error of order~$\nu$. This is precised by the
following theorem given in~\cite{Alfonsi} that relies on the idea developed by
Talay and Tubaro~\cite{Talay} for the Euler-Maruyama scheme. 

\begin{theorem}\label{Thm_weak} Let~$L$ be an operator satisfying the
  required assumptions on a compact domain~$\D$. We assume that:
  \begin{enumerate}
  \item $\hat{X}^x_t$ is a 
  potential weak $\nu$th-order scheme  for~$L$,
\item $f: \D \rightarrow \R$ is a function such that $u(t,x)=\E[f(X^x_{T-t})]$
    is defined and~$\mathcal{C}^\infty$ on $[0,T] \times \D$, and solves $\forall t \in
    [0,T], \forall x \in \D, \partial_t u(t,x)=-Lu(t,x)$.
\end{enumerate}
Then, there is $K>0$, $N_0 \in \N$, such that $|\E[f(\xcn_{t^N_N})]-\E[f(X^{x}_T)]|
\le K /N^\nu$ for $N \ge N_0$.
\end{theorem}

The mathematical analysis of the Cauchy problem for Mean-Reverting Correlation
processes is beyond the scope of this
paper. This issue has recently been addressed for the case of one-dimensional
Wright-Fisher processes by Epstein and Mazzeo~\cite{EpsteinMazzeo}, and Chen
and Stroock~\cite{ChenStroock} for the absorbing boundary case. In this setting,
Epstein and Mazzeo have shown that $u(t,x)$ is smooth for $f\in
\mathcal{C}^\infty([0,1])$. However, since we have an explicit formula for the
moments~\eqref{induc_moments}, we obtain easily that for any polynomial
function~$f$, the second point of Theorem~\ref{Thm_weak} is satisfied. By
the Stone-Weierstrass theorem, we can approximate for the supremum norm any continuous
function by a polynomial function and get the following interesting
corollary.
\begin{Corollary}\label{Cor_weak}Let $\hat{X}^x_t$ be 
  potential weak $\nu$th-order scheme for~$\corrp$. Let $f$ be a continuous
  function on~$\corr$. Then,
  $$\forall \varepsilon>0, \exists K>0, \ |\E[f(\xcn_{t^N_N})]-\E[f(X^{x}_T)]|
\le \varepsilon +  K /N^\nu.$$
\end{Corollary}

Let us now focus on the first assumption of Theorem~\ref{Thm_weak}. The
property of being a potential weak order scheme is easy to handle by using
scheme composition. This technique is well known in the literature and dates
back to Strang~\cite{Strang} the field of ODEs. In our framework, we recall
results that are stated in~\cite{Alfonsi}.
\begin{proposition}\label{prop_compo_schemas}
Let  $L_1, L_2$ be the generators of SDEs defined on~$\D$
 that satisfies the required assumption on~$\D$. Let 
$\hat{X}^{1,x}_t$ and $\hat{X}^{2,x}_t$ denote respectively two potential weak $\nu$th-order
schemes on~$\D$ for $L_1$ and  $L_2$.
\begin{enumerate}
  \item The scheme $\hat{X}^{2,\hat{X}^{1,x}_t}_t$ is a potential
    weak first order discretization scheme  for~$L_1+L_2$. Besides, if
    $L_1L_2=L_2L_1$, this is a potential  weak $\nu$th-order scheme  for~$L_1+L_2$.
  \item Let $B$ be an independent Bernoulli variable of parameter~$1/2$. If $\nu \ge 2$,
    $$ (a)\ B \hat{X}^{2,\hat{X}^{1,x}_t}_t+(1-B)\hat{X}^{1,\hat{X}^{2,x}_t}_t
    \ \ 
      \text{ and } \ \ \ \ (b) \ \hat{X}^{2,\hat{X}^{1,\hat{X}^{2,x}_{t/2}}_t}_{t/2}$$
    
    are potential weak second order
    schemes for~$L_1+L_2$.
  \end{enumerate}
\end{proposition}
Here, the composition $\hat{X}^{2,\hat{X}^{1,x}_{t_1}}_{t_2}$ means that we first use
the scheme~1 with time step~$t_1$ and then, conditionally to
$\hat{X}^{1,x}_{t_1}$, we sample the scheme~2 with initial value
$\hat{X}^{1,x}_{t_1}$ and time step~$t_2$. 

\subsection{A second-order scheme for MRC processes}

First, we
split the infinitesimal generator of~$\corrp$ as the sum $$L=L^{\xi}+\tilde{L},$$ where  $\tilde{L}$ is the infinitesimal generator of
$\corrps{d}{x}{\frac{d-2}{2}a^2}{I_d}{a}$ and $L^{\xi}$ is the operator
associated to~$\xi(t,x)$ given by~\eqref{def_xi}. Obviously, the ODE~\eqref{def_xi}
can be solved explicitly and we have to focus on the sampling of
$\corrps{d}{x}{\frac{d-2}{2}a^2}{I_d}{a}$. We use now
Theorem~\ref{theorem_spliopper} and consider the splitting
$$\tilde{L}=\sum_{i=1}^d a_i^2 \tilde{L}_i,$$
where $\tilde{L}_i$ is the infinitesimal generator of
$\corrps{d}{x}{\frac{d-2}{2}e^i_d}{I_d}{e^i_d}$. We claim now that it is sufficient to have a
potential second order scheme for~$\corrps{d}{x}{\frac{d-2}{2}e^1_d}{I_d}{e^1_d}$ in order to get a 
potential second order scheme for~$\corrp$. Indeed, if we have such a scheme,
we also get by a permutation of the coordinates  a
potential second order scheme $\hat{X}^{i,x}_t$
for~$\corrps{d}{x}{\frac{d-2}{2}e^i_d}{I_d}{e^i_d}$. Then, by time-scaling,
$\hat{X}^{i,x}_{a_i^2 t}$ is a potential second order scheme for
$\corrps{d}{x}{\frac{d-2}{2}a_i^2e^i_d}{I_d}{a_i e^i_d}$. Thanks to the commutativity, we get by
Proposition~\ref{prop_compo_schemas} that 
$\hat{X}^{d,\dots^{\hat{X}^{1,x}_{a_1^2t}}}_{a_d^2 t}$ is a potential second order
scheme for~$\tilde{L}$. Last, still by using
Proposition~\ref{prop_compo_schemas} we obtain that
\begin{equation}\label{second_order_sch}
\xi(t/2,\hat{X}^{d,\dots^{\hat{X}^{1,\xi(t/2,x)}_{a_1^2t}}}_{a_d^2 t}) \text{ is a potential second order scheme for } \corrp.
\end{equation}

Now, we focus on getting a second order scheme
for~$\corrps{d}{x}{\frac{d-2}{2}e^1_d}{I_d}{e^1_d}$. It is possible to
construct such a scheme by using an ad-hoc splitting of the infinitesimal
generator. This is made in Appendix~\ref{schema_original}. Here, we achieve this
task by using the connection between Wishart and MRC process and the existing
scheme for Wishart processes.  In Ahdida
and Alfonsi~\cite{AA}, we have obtained a potential second order scheme~$\hat{Y}^{1,x}_t$
for $WIS_d(x,d-1,0,e^1_d)$. 
Besides, this scheme is constructed with discrete
random variables, and we can check that there is a constant $K>0$ such that for
any $1\le i\le d$,
$|(\hat{Y}^{1,x}_t)_{i,i}-1|\le K \sqrt{t}$ holds almost surely (we even have
$(\hat{Y}^{1,x}_t)_{i,i}=1$ for $2\le i \le d$). Therefore, we
have $1/2\le (\hat{Y}^{1,x}_t)_{i,i}\le 3/2$ for $t\le 1/(4K^2)$. Let
$f\in\mathcal{C}^\infty(\corr)$. Then $f({\bf p}(y))$ is $\mathcal{C}^\infty$
with bounded derivatives on $\{y \in \posm, 1/2 \le y_{i,i}\le 3/2 \}$. Since
$\hat{Y}^{1,x}_t$ is a potential second order scheme, it
comes that there are constants $C,\eta>0$ that only depend on a good
sequence of~$f$ such that
\begin{equation}
  \forall t \in (0,\eta), \left|\E[f({\bf p}(\hat{Y}^{1,x}_t))]-f(x)-t \tilde{L}_1^W(f\circ{\bf
  p})(x)-\frac{t^2}{2}(\tilde{L}_1^W)^2(f\circ{\bf
  p})(x) \right|\le C t^3,
\end{equation}
where $\tilde{L}_1^W$ is the generator of~$WIS_d(x,d-1,0,e^1_d)$. Thanks to
Remark~\ref{rem_oper}, we get that there are constants $C,\eta$ depending
only on a good sequence of~$f$ such that
\begin{equation}
  \forall t \in (0,\eta), \left|\E[f({\bf p}(\hat{Y}^{1,x}_t))]-f(x)-\left(t+(5-d)\frac{t^2}{2}\right) \tilde{L}_1f(x)-\frac{t^2}{2}(\tilde{L}_1)^2f(x) \right|\le C t^3.
\end{equation}
In particular, ${\bf p}(\hat{Y}^{1,x}_t)$ is a potential first order scheme
for~$L_1$ and even a second order scheme when $d=5$. We can improve this by
taking a simple time-change. We set:
$$\phi(t)=\begin{cases} t-(5-d)\frac{t^2}{2} \text{ if } d\ge 5\\
\frac{-1+\sqrt{1+2(5-d)t}}{5-d} \text{ otherwise,}\end{cases}
$$
so that in both cases, $\phi(t)=t-(5-d)\frac{t^2}{2}+O(t^3)$. Then, we have
that  there are constants $C,\eta$ still depending
only on a good sequence of~$f$ such that $
  \forall t \in (0,\eta), \left|\E[f({\bf p}(\hat{Y}^{1,x}_{\phi(t)}))]-f(x)-t
    \tilde{L}_1f(x)-\frac{t^2}{2}(\tilde{L}_1)^2f(x) \right|\le C t^3$,
  and therefore
  \begin{equation}\label{second_order_sch_L1}
{\bf p}(\hat{Y}^{1,x}_{\phi(t)})    \text{ is a potential second order scheme for } \corrps{d}{x}{\frac{d-2}{2}e^1_d}{I_d}{e^1_d}.
\end{equation}

\subsection{A faster second-order scheme for MRC processes under Assumption~\eqref{assump_fast}}
We would like to discuss on the time complexity of the
scheme given by~\eqref{second_order_sch} and~\eqref{second_order_sch_L1} with
respect to the dimension~$d$. The second order scheme given in Ahdida and
Alfonsi~\cite{AA} for~$WIS_d(x,d-1,0,e^1_d)$ requires~$O(d^3)$ operations. Since it is
used $d$ times in~\eqref{second_order_sch} to generate a sample, the overall
complexity is in~$O(d^4)$. In the same manner, the second order given in
Appendix~\ref{schema_original} requires~$O(d^4)$ operations.  However, it is possible  to get a faster
second order scheme with complexity~$O(d^3)$ if we make the following assumption:
\begin{equation}\label{assump_fast}
a_1=\dots=a_d\ (i.e.\ a=a_1 I_d) \text{ and }  \kp c+c\kp -(d-1)a^2 \in \posm.
\end{equation}
This latter assumption is stronger than~\eqref{cond_weak_existence} but weaker
than~\eqref{cond_strong_existence}, which respectively ensures weak and strong
solutions to the SDE. Under~\eqref{assump_fast}, we can check by
Lemma~\ref{lemm_somme_ODE} that  
\begin{equation}\label{def_zeta}
  \zeta'(t,x)=\kp(c-x)+(c-x)\kp -\frac{d-1}{2}[a^2(I_d-x)+(I_d-x)a^2], \
  \zeta(0,x)=x \in \corr 
\end{equation}
takes values in~$\corr$. Then, we split the infinitesimal generator
of~$\corrp$ as the sum
$$L=L^\zeta+a_1^2 \bar{L},$$
where $L^\zeta$ is the operator associated to the ODE~$\zeta$, and $\bar{L}$
is the infinitesimal generator of
$\corrps{d}{x}{\frac{d-1}{2}I_d}{I_d}{I_d}$. In~\cite{AA}, it is given a second
order scheme $\hat{Y}^x_t$ for $WIS_d(x,d,0,I_d)$ that has a 
time-complexity in $O(d^3)$. We then consider $f\in\mathcal{C}^\infty(\corr)$
and get by using the same arguments as before that there are constants $C,\eta>0$
depending only on a good sequence of~$f$ such that
$$\forall t \in (0, \eta), \left|\E[f({\bf
    p}(\hat{Y}^x_t))]-f(x)-t\bar{L}^W(f\circ {\bf
    p})(x)-\frac{t^2}{2}(\bar{L}^W)^2(f\circ {\bf p})(x) \right|\le Ct^3,$$
where $\bar{L}^W$ is the infinitesimal generator of~$WIS_d(x,d,0,I_d)$. Thanks
to Remark~\ref{rem_oper}, we get that
$$\forall t \in (0, \eta), \left|\E[f({\bf
    p}(\hat{Y}^x_t))]-f(x)-\left(t+(4-d)\frac{t^2}{2} \right)\bar{L}f(x) - \frac{t^2}{2} \bar{L}^2f(x) \right|\le Ct^3. $$
In particular, ${\bf
    p}(\hat{Y}^x_t)$ is a first order scheme
  for~$\corrps{d}{x}{\frac{d-1}{2}I_d}{I_d}{I_d}$ and by
  Proposition~\ref{prop_compo_schemas},
  \begin{equation}\label{fast_first_order_sch}
\zeta(t,{\bf  p}(\hat{Y}^x_{a_1^2t})) \text{ is a potential first order scheme for }
\corrp.
\end{equation}
As before, we can improve this by using the following time-change:
$\psi(t)=t-(4-d)\frac{t^2}{2}$ if $d\ge 4$ and $\psi(t)=
\frac{-1+\sqrt{1+2(4-d)t}}{4-d} $ otherwise, so that
$\psi(t)=t-(4-d)\frac{t^2}{2}+O(t^3)$ in both cases. We get that ${\bf
    p}(\hat{Y}^x_{\psi(t)})$ is a potential second order scheme for
  $\corrps{d}{x}{\frac{d-1}{2}I_d}{I_d}{I_d}$. Then, we obtain that
\begin{equation}
\zeta(t/2,{\bf  p}(\hat{Y}^{\zeta(x,t/2)}_{a_1^2 \psi(t)})) \text{ is a potential second order scheme for } \corrp\label{fast_second_order_sch}
\end{equation}
by using Proposition~\ref{prop_compo_schemas}. Its time complexity is in $O(d^3)$.

\subsection{Numerical experiments on the discretization schemes}

In this part, we discuss briefly the time needed by the different schemes
presented in the paper. We also illustrate the weak convergence of the schemes
to check that it is in accordance with Corollary~\ref{Cor_weak}. In
Table~\ref{Result_Table}, we have indicated the time required to sample~$10^6$
scenarios for different time-grids in dimension $d=3$ and $d=10$. These times
have been obtained with a 2.50 GHz CPU computer. As expected,
the modified Euler scheme given by~\eqref{Euler_mod} is the most time
consuming. This is mainly due to the computation of the matrix square-roots that
require several diagonalizations. Between the second order schemes that are
defined for any parameters satisfying~\eqref{cond_weak_existence}, the second order scheme given
by~\eqref{second_order_sch} and~\eqref{second_order_sch_L1} is rather faster
than the ``direct'' one presented in
Appendix~\ref{schema_original}. However, it has a larger bias on our example in
Figure~\eqref{MRC01}, and their overall efficiency is similar.
Nonetheless, both are as expected overtaken
by the fast second order scheme~\eqref{fast_second_order_sch}. Let us
recall that it is only defined under Assumption~\eqref{assump_fast} which
is satisfied by our set of parameters. Also, the fast first order scheme given
by~\eqref{fast_first_order_sch} requires roughly the same computation time.
\begin{table}
\centering
\begin{tabular}{|l|c|c|}
  \hline
& $d=3$ & $d=10$ \\
\hline
   $2^{nd}$ order ``fast''     & 19 &  224  \\  
  $2^{nd}$ order  &65 &1677 \\
  $2^{nd}$ order ``direct''	 &  90 & 3105   \\
  $1^{st}$ order    ``fast''       &19 & 224  \\
   Corrected Euler          & 400 &  14322  \\
   \hline
\end{tabular}

\caption{{ Computation time in seconds to generate $10^6$ paths up to $T=1$
    with $N=10$ time-steps of the following $MRC$ process:  $\kappa = 1.25 I_d$, $c=I_d$, $a=I_d,$ and
    $x_{i,j} =0.7$ for $i \not = j$.
 }}\label{Result_Table}
\end{table}

Let us switch now to Figure~\ref{MRC01} that illustrates the weak
convergence of the different schemes. To be more precise, we have plotted the
following combinations the moments of order~$3$ and~$1$ (i.e. respectively
\begin{equation}\label{Eqt_MC}
\E\left[\sum_{\substack{1\leq i \neq j \leq 3\\ 1 \leq k \neq l \leq 3 }} \left[(\hat{X}^N_T)_{i,j}(\hat{X}^N_T)_{k,l}^2\right] + (\hat{X}_T^N)_{1,2}(\hat{X}_T^N)_{2,3}(\hat{X}_T^N)_{1,3}\right],
\end{equation}
and $\E\left[        \sum_{1 \leq i \neq j \leq d }(\hat{X}^N_T)_{i,j}\right]$)
in function of the time-step~$T/N$. These expectations can be calculated exactly for the MRC
process thanks to Proposition~\ref{prop_equ_ODE}, and the exact value is
reported in both graphics. As expected, we observe a quadratic convergence for
the second order schemes, and a linear convergence for the first order
scheme. In particular, this demonstrates numerically the gain that we get by considering the
simple change of time~$\psi$ between the schemes~\eqref{fast_first_order_sch}
and~\eqref{fast_second_order_sch}. Last, the modified Euler scheme shows a roughly linear
convergence. It has however a much larger bias and is
clearly not competitive.

\begin{figure}[h]
\centering
  \begin{tabular}{rl}
    \psfrag{Ex}{ {\scriptsize Exact value}}
    \psfrag{O2}{ {\scriptsize 2nd ``direct''}} 
    \psfrag{EL}{ {\scriptsize Euler}}
    \psfrag{O2W}{ {\scriptsize 2nd}}
    \psfrag{O2WF}{ {\scriptsize 2nd ``fast''}}	
    \psfrag{O1}{{\scriptsize 1st ``fast''   }}	
    \hspace{-0.8cm}
    \psfig{file=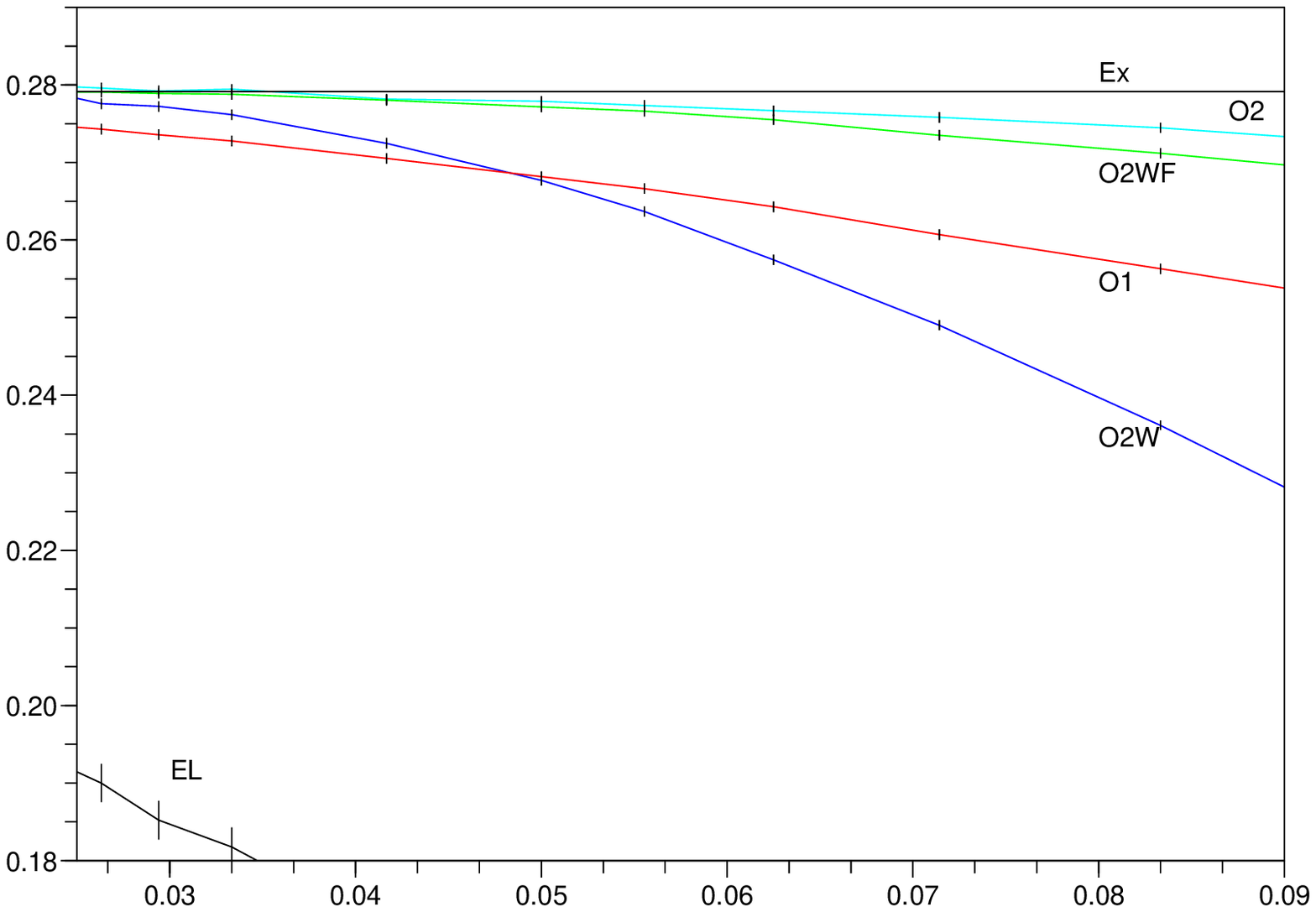,width=8.5cm,height=8cm}
    &\hspace{-0.2cm}
    \psfrag{Ex}{ {\scriptsize Exact value}}
    \psfrag{O2}{ {\scriptsize  2nd ``direct''}} 
    \psfrag{EL}{ {\scriptsize Euler}}
    \psfrag{O2W}{ {\scriptsize 2nd}}
    \psfrag{O2WF}{ {\scriptsize 2nd ``fast''}}	
    \psfrag{O1}{{\scriptsize 1st ``fast''}}
    \psfig{file=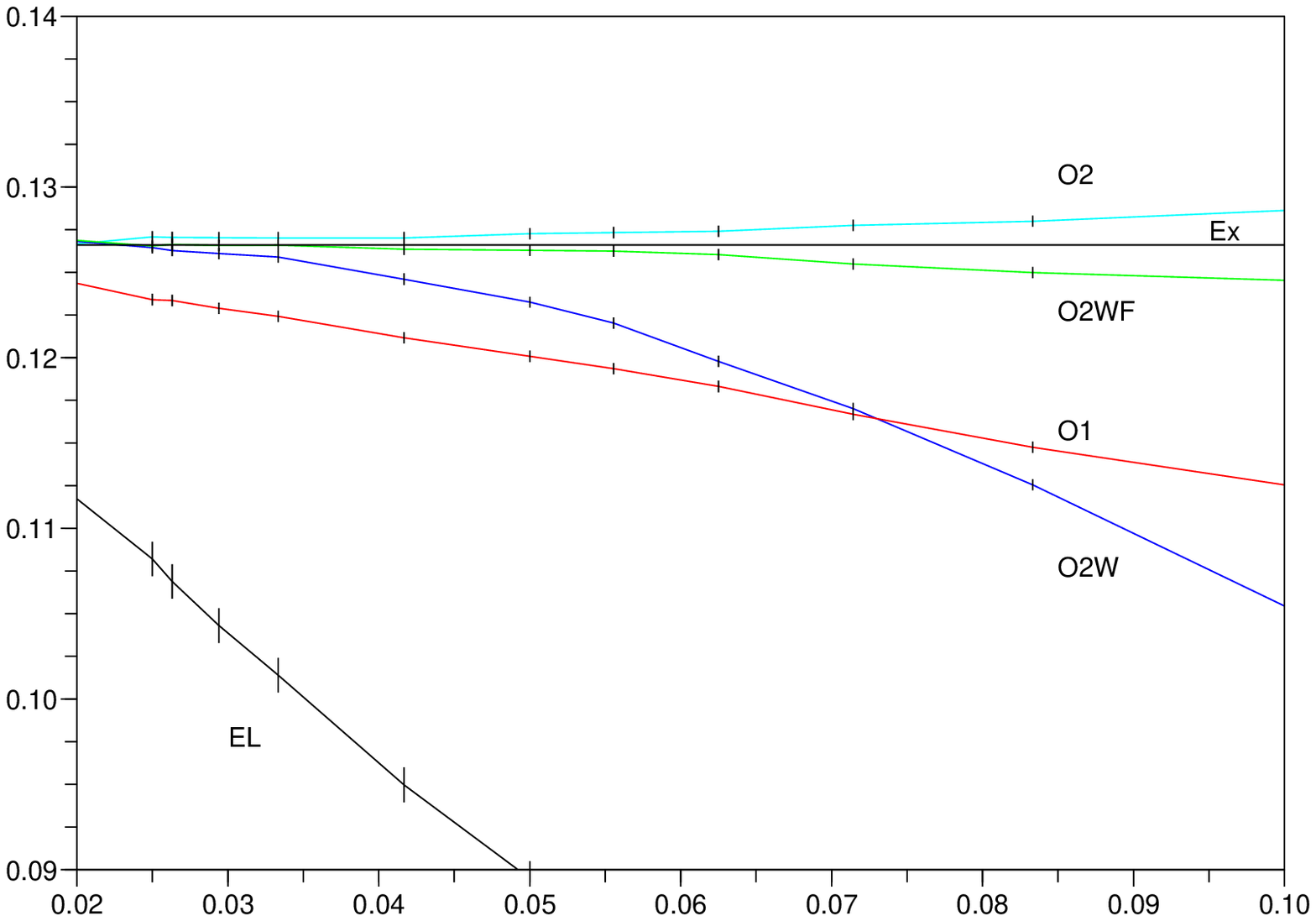,width=8.5cm,height=8cm}
  \end{tabular}
  \caption{{  $d=3$, same parameters as for Table~\ref{Result_Table}. In
      the left (resp. right) side is plotted~\eqref{Eqt_MC} (resp.  $\E\left[
        \sum_{1 \leq i \neq j \leq d }(\hat{X}^N_T)_{i,j}\right]$) in function
      of the time step~$1/N$. The width of each point
    represents the $95\%$ confidence interval ($10^7$ scenarios for the
    modified Euler scheme and $10^8$ for the others).}}
\label{MRC01}
\end{figure}


\section{Financial application of correlation processes}\label{sec_fin}

In this section, we focus on the modeling of the dependence between $d$~risky
assets. We will denote by $S^1_t,\dots,S^d_t$ their value at
time~$t$, and we set $\log(S_t)=(\log(S^1_t),\dots,\log(S^d_t))^T$. Basically
if we exclude jumps, we can assume that the assets follow under
a risk neutral probability space $(\Omega,\mathcal{F},\Px)$ the following dynamics:
\begin{equation}\label{general_stock_dyn}
  d \log(S^i_t) = \left(r-\frac{1}{2}
 (\Sigma_t)_{i,i} \right)dt + (\sqrt{\Sigma_t} dB_t)_i, \ 1 \le i \le d
\end{equation}
where $r$ is the interest rate, $(B_t,t\ge 0)$ is a $d$-dimensional standard Brownian
motion and $(\Sigma_t,t\ge 0)$ is an adapted $\posm$-valued process. This
process describes the instantaneous covariance of the stocks
\begin{equation}\label{IDaF_stocks} \langle d \log(S^i_t) , d \log (S^j_t)
  \rangle = (\Sigma_t)_{i,j} dt,
\end{equation}
and $\sqrt{(\Sigma_t)_{i,i}}$ is usually called the volatility of the asset $S^i_t$.

Modelling directly the whole covariance process $(\Sigma_t,t\ge 0)$ is not an
easy task. This path has recently been explored by Gourieroux and
Sufana~\cite{Gourieroux}. Their model has been enhanced by Da Fonseca et
al.~\cite{Dafonseca}. They assume that $(\Sigma_t,t\ge 0)$ is a Wishart
process and consider a dynamics that is a natural extension of the famous Heston
model~\cite{Heston} to $d$~stocks. Since Wishart processes are closely
connected to MRC processes, we will discuss in detail their model in
Section~\ref{sec_WSC}. In particular, we explain why, in our opinion, it can be
hardly used in practice when the number of assets is large (say $d\ge 5$).

Instead, the common practice of the market is to start with modeling the volatility
$\sigma^i_t$ of each stock~$S^i_t$. Then, the dependence is modeled by a correlation process $(C_t,t\ge
0)$, so that the covariance process is defined by
$$\Sigma_t=diag(\sigma^1_t,\dots,\sigma^d_t)C_tdiag(\sigma^1_t,\dots,\sigma^d_t).$$
Thus, if we consider the following dynamics for the assets
\begin{equation}\label{general_stock_dyn2}
  d \log(S^i_t) = \left(r-\frac{1}{2}
 (\sigma^i_t)^2 \right)dt + \sigma^i_t(\sqrt{C_t} dB_t)_i, \ 1 \le i \le d,
\end{equation}
we get back the same instantaneous covariance given
as~\eqref{IDaF_stocks}. This bottom-up approach has many advantages. Indeed,
the modelling of individual stocks is well documented and
handled every day by the financial desks. Besides, the choice of the
volatility model is free and can be chosen at one's convenience. For example,
we can take a local volatility model ($\sigma^i_t=\sigma(t,S^i_t)$) or a
stochastic volatility model such as Heston model, or a local
stochastic volatility model~\cite{Alexander}. The calibration of the single
stocks can be thus performed separately, before the calibration of the correlation
process. There is still few literature that brings on fitting the correlation
to market data. Today, the only liquid and quoted derivatives that bring on the
dependence between assets are index options. Recently, Langnau~\cite{Langnau}
and Reghai~\cite{Reghai} focused on the calibration of some particular local
correlation models (i.e. $C_t=C(t,S^1_t,\dots,S^d_t)$) to index option
prices. Under a slightly different setting, Jourdain and Sbai~\cite{Jourdain} have
also considered this issue. However, there is still up to our knowledge very few
studies on stochastic correlation modelling. We will discuss in
Section~\ref{appli_MRC} how MRC processes and possible extensions could be
used for that purpose.

\subsection{The Wishart stochastic covariance model}\label{sec_WSC}
Da Fonseca et al.~\cite{Dafonseca} assume that $(\Sigma_t,t\ge 0)$ follows a Wishart process
\begin{equation}\label{Wish_Cov}d\Sigma_t =\left( {\alpha}aa^T  + b\Sigma_t+\Sigma_tb^T \right )dt +
 \sqrt{\Sigma_t}dW_ta + a^TdW_t^T\sqrt{\Sigma_t}
\end{equation}
with $\alpha\ge d-1$, $a,b \in \genm$. The vector Brownian motion~$B$ and the
 matrix Brownian motion~$W$ may be correlated as follows:
  $$dB_t=dW_t \rho +\sqrt{1-\|\rho\|_2^2} d\tilde{B}_t, \rho \in \R^d \ s.t. \
  \|\rho\|_2\le 1,$$
 where $( \tilde{B}_t,t\ge 0)$ is a $d$-dimensional Brownian motion independent from~$W$.
With this choice, the couple $(\log(S_t),\Sigma_t)$
is an affine process, and its characteristic function can be obtained by
solving Riccati equations (see Da Fonseca et al.~\cite{Dafonseca}).

From~\eqref{Wish_Cov}, we get that there are Brownian motions
$\beta^i_t$, $i=1,\dots,d$ such that the volatility square of~$S^i_t$ follows:
$$ d (\Sigma_t)_{i,i}= (\alpha (aa^T)_{i,i} +2 \sum_{j=1}^{d}
b_{i,j}(\Sigma_t)_{i,j}  dt + 2 \sqrt{(aa^T)_{i,i}} \sqrt{(\Sigma_t)_{i,i}} d
\beta^i_t. $$
The non diagonal elements $(\Sigma_t)_{i,j}$ for $i\not = j$ can be seen as
factors that drive the volatility of each stock. In its full form, the
dynamics of one stock and its volatility is not autonomous. This means in
practice that it is not possible to calibrate this model separately  to single
stock market (mainly, European options on single stocks). This calibration has
to be done at the same time for all the stocks, which is a priori a very
challenging task: many parameters and a lot of data are involved when the
number of assets~$d$ gets large.

Then, one may want to recover autonomous dynamics for each stock in order to
calibrate them separately. Within this model, the
only possible choice is to assume that $b$ is a diagonal matrix. In this case,
$(\Sigma_t)_{i,i}$ follows a CIR diffusion and the couple
$((\Sigma_t)_{i,i},S^i_t)$ follows an Heston model~\cite{Heston}: there are independent Brownian motions
$\beta^i,\ \gamma^i$ such that 
\begin{eqnarray}
 d \log(S_t^i) &=& (r-\frac{(\Sigma_t)_{i,i}}{2})dt + \sqrt{(\Sigma_t)_{i,i}}\left(
   \tilde{\rho}^i d \beta^i_t + \sqrt{1-(\tilde{\rho}^i)^2 }d\gamma^i_t\right) \nonumber\\
d(\Sigma_t)_{i,i}&=& \left( \alpha (a^Ta)_{i,i} + 2b_{i,i}(\Sigma_t)_{i,i}\right)dt+ 2 \sqrt{(a^Ta)_{i,i}}\sqrt{(\Sigma_t)_{i,i}}d\beta_t^i, \nonumber\\ 
\tilde{\rho}^i&=& \frac{(a^T\rho)^i}{\sqrt{(a^Ta)_{i,i}}}= \frac{\sum_{k=1}^da_{k,i}\rho^i}{\sqrt{\sum_{k=1}^da_{k,i}^2}}.  \nonumber
\end{eqnarray}
In practice, each individual stock could be then calibrated like in the Heston
model. Unfortunately, there are further restrictions implied by this
model and especially that $\alpha \ge d-1$, which is the condition that ensures
the existence of the Wishart process. This is unlikely because when
calibrating Heston to market data, it is typical to get values of $\alpha$ around or
below~$1$. When modelling many stocks together (say $d\ge 5$), it is then not
possible to fit conveniently market data because of this restriction on~$\alpha$.

Last, one of the main feature of the model~\eqref{Wish_Cov} is the Affine
property. It allows to obtain the characteristic function of the stocks
by solving Riccati differential equations. Then, the pricing of European style
options can be made by using Fourier inversion. This approach is known to be
very efficient in a one-dimensional framework (see Carr and Madan~\cite{CM}) and
has been used successfully by Da Fonseca et al.\cite{Dafonseca} to
price Best-of options with $d=2$ assets. However, when the number of assets $d$
is much larger,  the Fourier inversion requires an integration in dimension~$d$ and
can no longer be computed quickly. Unless the payoff has a very particular
structure to reduce the dimension, the pricing by Fourier inversion is no longer
competitive with respect to Monte-Carlo methods and the
Affine property does not really give a crucial advantage in terms of
computational methods.

For all these reasons, we believe that this model can be used successfully in
practice for a small number of stocks ($d=2,3$) but is instead inadequate to
model a large basket.

\subsection{Towards a stochastic correlation model}\label{appli_MRC}

Now, we would like to discuss the application of the MRC processes under the
framework~\eqref{general_stock_dyn2}. The first natural idea would be simply
to take~$(C_t)_{t\ge 0} \sim MRC_{d}(C_0,\kp,c,a)$. With this choice, we would
get analytical formulas for correlation swaps. Indeed, a correlation swap between
assets $S^i_t$ and $S^j_t$ ($i\not=j$) with maturity~$T$ has, up to a discount factor, the
following price:
\begin{equation}\label{price_CS} \E\left[\frac{1}{T} \int_0^T (C_t)_{i,j}dt \right]=\frac{1}{(\kp_i+\kp_j)T}\left[(C_0)_{i,j}(1-e^{-(\kp_i+\kp_j)T})-c_{i,j}(1-e^{-(\kp_i+\kp_j)T}-(\kp_i+\kp_j)T) \right].
\end{equation}
To be precise, the payoff of a correlation swap is defined as the average
over the period $[0,T]$ of the daily correlation of log-returns
$\frac{
  \log\left(\frac{S^i_{t+\Delta}}{S^i_{t}}\right)\log\left(\frac{S^j_{t+\Delta}}{S^j_{t}}\right)}{
  {\sqrt{\log\left(\frac{S^i_{t+\Delta}}{S^i_{t}}\right)^2}\sqrt{\log\left(\frac{S^j_{t+\Delta}}{S^j_{t}}\right)^2}}}$
and is here approximated by $\frac{1}{T} \int_0^T
(C_t)_{i,j}dt$. Unfortunately, up to now, correlation swaps are not quoted on
the markets and are only dealt over the counter. It is then not possible to
get data on their prices in order to calibrate~$\kp$, $c$ and $C_0$, which
would have been very tractable thanks to formula~\eqref{price_CS}.

The only quoted options that bring on the dependence between assets are index
options.  We have been kindly given by Julien Guyon at Soci\'et\'e G\'en\'erale market data on the
DAX index at the 4th October 2010. He provides us with data on European index
option prices as well as parametrized local volatility functions
$\sigma^i(t,x)$ that are already calibrated to options price on each
asset. We thus assume in the sequel the dynamics~\eqref{general_stock_dyn2}
for the stocks with $\sigma^i_t=\sigma^i(t,S^i_t)$. We will also denote $$I_t=\sum_{i=1}^d \alpha_i S^i_t$$  the
index value and will assume constant weights $\alpha_i$ such that
$\sum_{i=1}^d \alpha_i= 1$. These weights are given in
Table~\ref{ITable_Euler} for the DAX.

\begin{table}[h]
  \centering
  \begin{tabular}{cccccccccc}
    \hline
{\tiny {SIEMENS }} &   {\tiny {BASF}} &  {\tiny {BAYER}}   &   {\tiny {E-ON}}
&   {\tiny {DAIMLER}} &   {\tiny {ALLIANZ}}  & {\tiny {SAP}} & {\tiny
  {D. TELEKOM}}  &   {\tiny {D. BANK}} &  {\tiny {RWE }}\\
9.91 &   8.03  &  7.97   &   7.67  &   7.35 &   6.97  &  6.01  & 5.6 &   5.06 &  3.86 \\ 
{\tiny {M-RUECK}} & {\tiny {LINDE}}   &   {\tiny {BMW }} &  {\tiny {VW
  }} & {\tiny {D. POST}}  & {\tiny {ADIDAS }}  &  {\tiny {D. BOERSE }} &
{\tiny {FRESENIUS-MC}} &   {\tiny {THYSS. KRUPP}} &  {\tiny {MAN}}\\
  3.23 &  3.06   &   2.92  &  2.26 & 2.05  &  1.76   &  1.65  &  1.63 &   1.52 &  1.47  \\
{\tiny {HENKEL }} & {\tiny {SDF}} & {\tiny {LUFTHANSA }}  &  {\tiny
  {METRO}}  &  {\tiny {INFINEON }}  & {\tiny {HEIDEL.}}  &  {\tiny
  {FRESENIUS}}   &{\tiny {BEIERSDORF}} & {\tiny {COMMERZBANK }}  & {\tiny {MERCK}}\\
 1.29 & 1.17  & 1.16   &  1.12  &  1.02   & 0.94  &  0.9   & 0.86  & 0.84   & 0.74 \\

    \hline
  \end{tabular}
  \caption{DAX index composition,  in percentage (4 October 2010)}\label{ITable_Euler}
\end{table}

To calibrate MRC processes to index options, it is desirable to reduce the
number of parameters. We will assume in the sequel with a slight abuse of
notation that $\kp=\kp I_d$, $a=a I_d$ with $\kp, a \in \R_+^*$,
$c_{i,j}=\indi{i=j}+\rho\indi{i \not =j}$ for $\rho \in [-1/(d-1),1]$ and
$C_0=c$. Condition~\eqref{cond_weak_existence} is then simply equivalent to
$2\kp(1-\rho)\ge (d-2)a^2$, but we will assume in addition that
$$ 2\kp(1-\rho)\ge (d-1)a^2$$
in order to take advantage of the $O(d^3)$ discretization scheme given
by~\eqref{fast_second_order_sch}.
Thus, $C_t$ follows the following dynamics
\begin{equation}\label{MRC_model} C_t=c +2 \kp \int_0^t (C_s-c)ds + a \sum_{n=1}^d \int_0^t \left( \sqrt{C_s-C_s e_d^{n} C_s} dW_se_d^{n}
  +e_d^{n} dW_s^T \sqrt{C_s-C_s e_d^{n} C_s} \right),
\end{equation}
and we assume that the Brownian motion~$W$ is independent from the Brownian
motion~$B$ that drives the stocks~\eqref{general_stock_dyn2}. We would like to
calibrate such a process to European index option prices. To do so, we first
calculate the value $\rho$ such that the constant correlation model
$\indi{i=j}+\rho\indi{i \not =j}$ fits the at the money implied volatility. Then, we use
this value and look at the impact of~$\kp$ and~$a$. Figure~\ref{too_flat_smile}
illustrates these results. As one could expected, this model gives a smile
which is not enough sloping and is unable to fit the volatility skew. The
parameters~$\kp$ and~$a$ have no impact on the slope of the smile. 
\begin{figure}[h]
  \centering
  \begin{tabular}{rl}
    \hspace{-0.8cm}
    \psfig{file=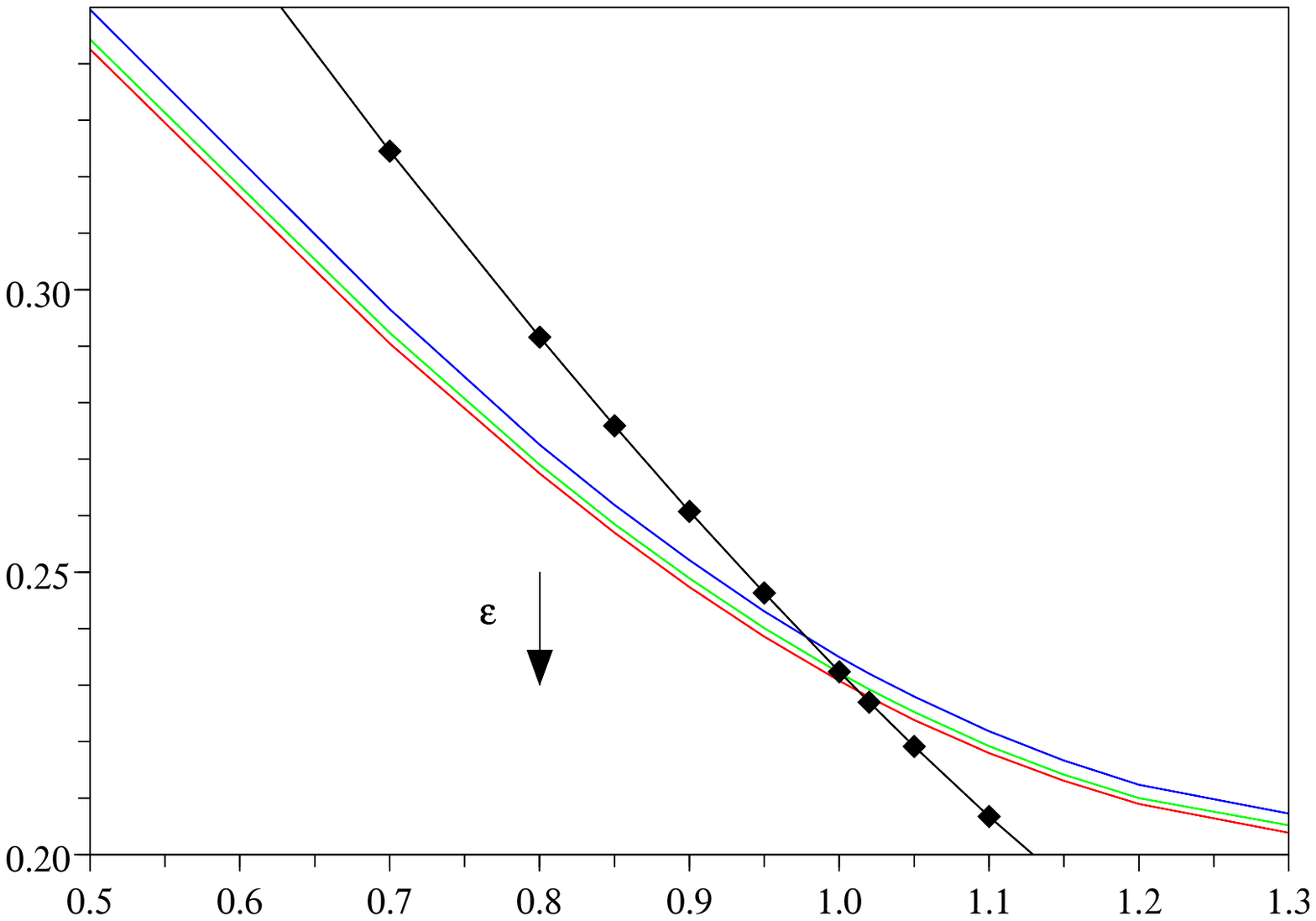,width=8cm,height=7cm}
    &\hspace{-0.2cm}
    \psfig{file=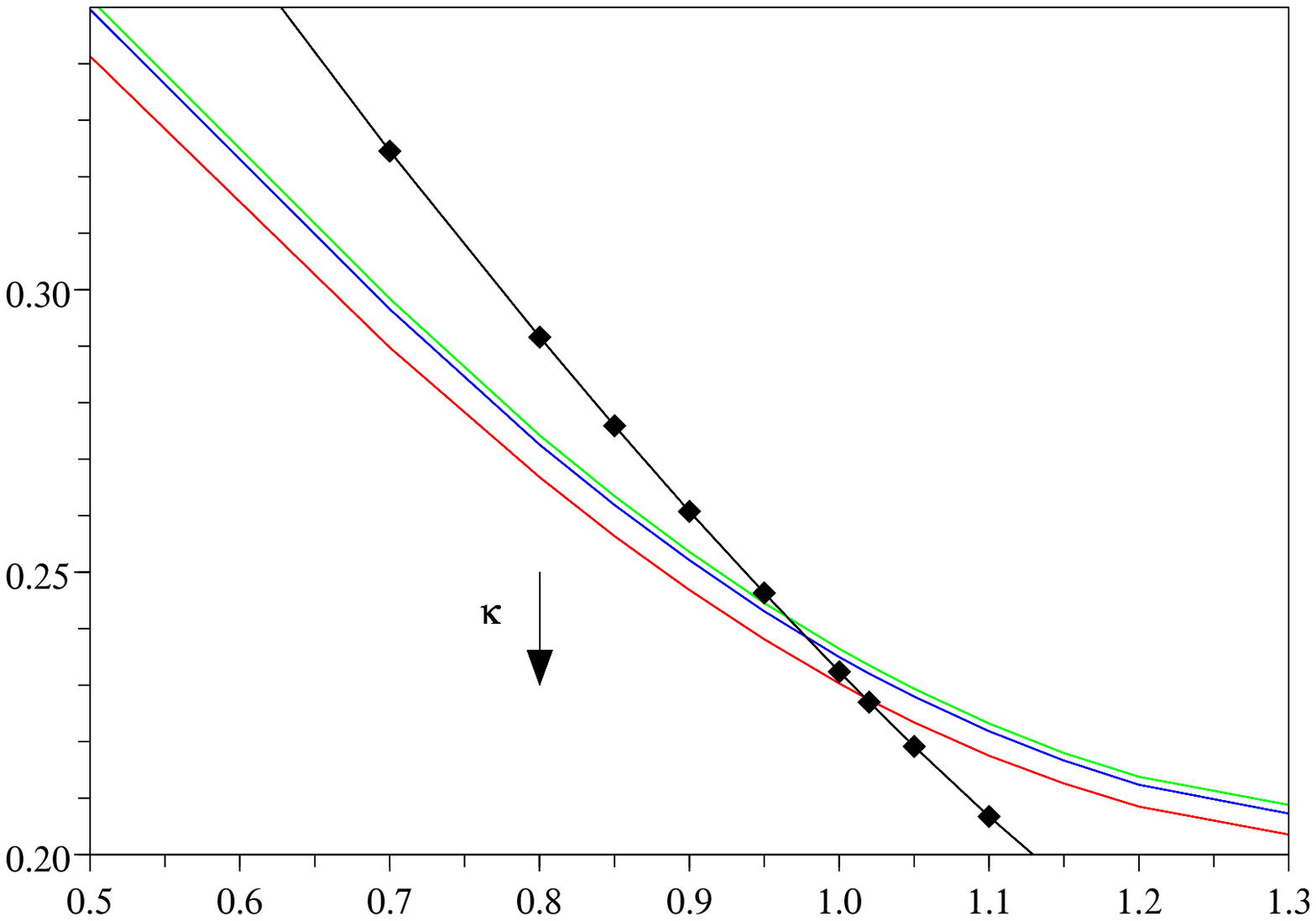,width=8cm,height=7cm}\\
  \end{tabular}
  \caption{\small{ In both graphics, the line with diamonds indicate the market
      $1$~year implied volatility. We plot the implied given by
      model~\eqref{MRC_model} with $\rho=0.67968$ and use the parametrization
      $a=\sqrt{\frac{2\kp \varepsilon}{d-1}(1-\rho)},$ with $\varepsilon \in
      [0,1]$. {\bf Left:}~$\kappa=100$ and $\varepsilon \in \left \lbrace 0,  0.6, 1 \right \rbrace$. {\bf  Right:} $\varepsilon=0.5$ and  $\kappa \in \left \lbrace 1,  20, 100 \right \rbrace.$ These curves are obtained with  $10^6$ Monte-Carlo samples and discretization time step $0.025.$}}\label{too_flat_smile}  
\end{figure}

The reason of this too flat smile is that the correlation is not
related to the index price, while the market would expect that the correlation
is high (resp. low) when the index is low (resp. high).
There are at least two ways to correct this. First, one may assume some dependence between the
Brownian motions $W$ and $B$, which would be very analogous to what is made
for stochastic volatility models. This requires to find an adequate way to
correlate these Brownian motion from a financial point of view that is also
tractable for simulation purposes. We have left this for further
research.

The second way to make depend the correlation in function of the
index value is simply to assume a local correlation model $C(t,I_t)$ with
$C:\R_+\times \R_+^* \rightarrow \corr$ and keep $W$ and $B$ independent. This
approach has been considered by Reghai~\cite{Reghai}. In Ahdida~\cite{these}, it
is shown that the following parametric form
\begin{equation}\label{local_corr} C(t,I_t)_{i,j}=\indi{i=j}+\rho(t,I_t)\indi{i \not =j}, \text{ with }
\rho(t,I_t)=\max\left(\frac{1}{1+\eta\left(I_t/I_0 \right)^\gamma},
  \rho_{\min} \right),
\end{equation}
with $\eta, \gamma,\rho_{min} \ge 0$ is very tractable to calibrate the index
smile at a given maturity. Indeed, it is shown that $\eta$ tunes the implied
volatility at the money, $\gamma$ tunes the skew (i.e. the slope at the money
of the implied volatility), and $\rho_{min}$ tunes the right tail of the
smile. The left hand side of Figure~\ref{lc_extension} shows the index smile
data and the implied volatility given by this model. Also, an extension of this model with time-dependent
parameters can be used to fit the index smile for different maturities.
\begin{figure}[t]
\centering
\begin{tabular}{rl}
    \hspace{-0.8cm}
    \psfig{file=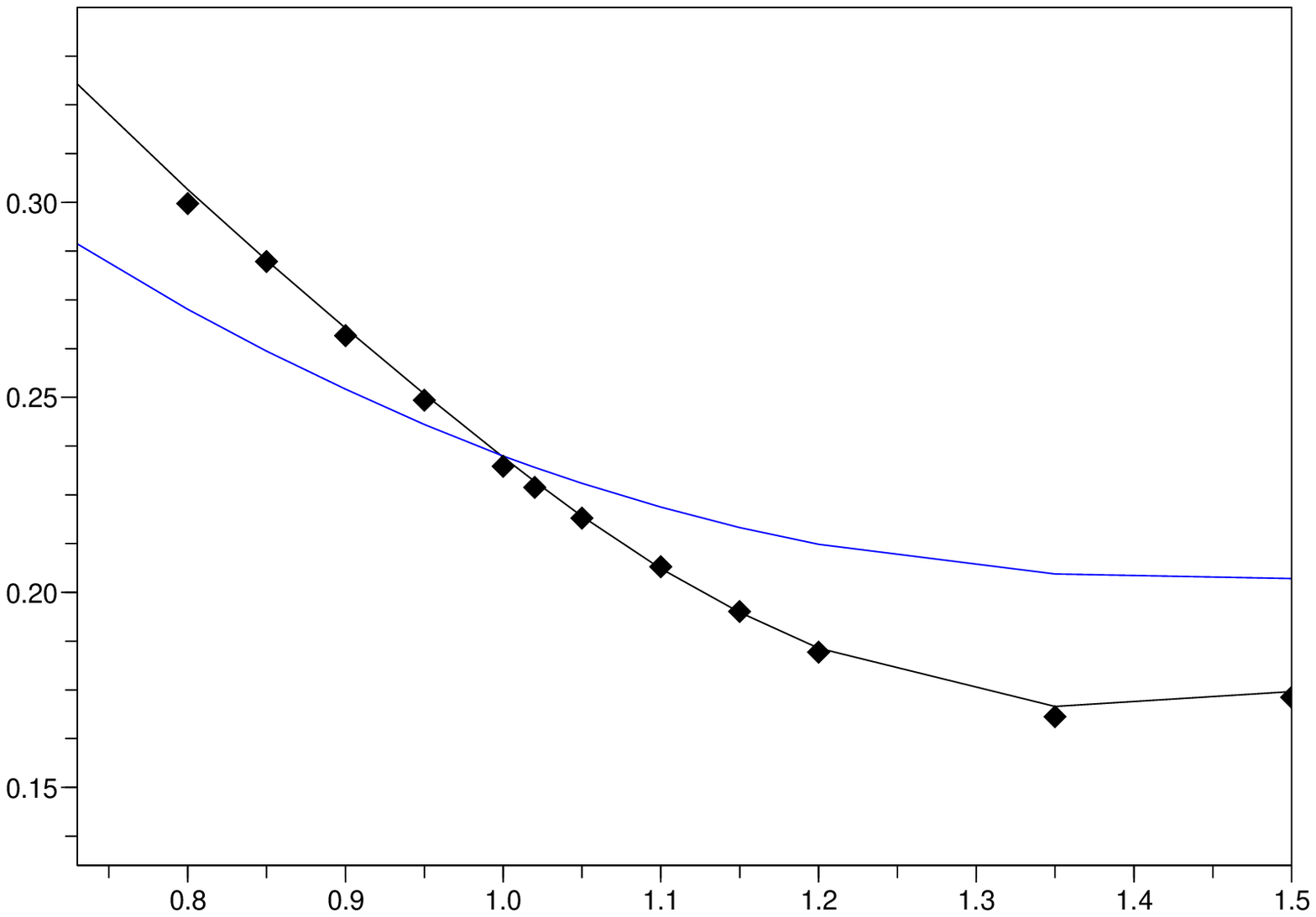,width=8cm,height=7cm}
    &\hspace{-0.2cm}
    \psfig{file=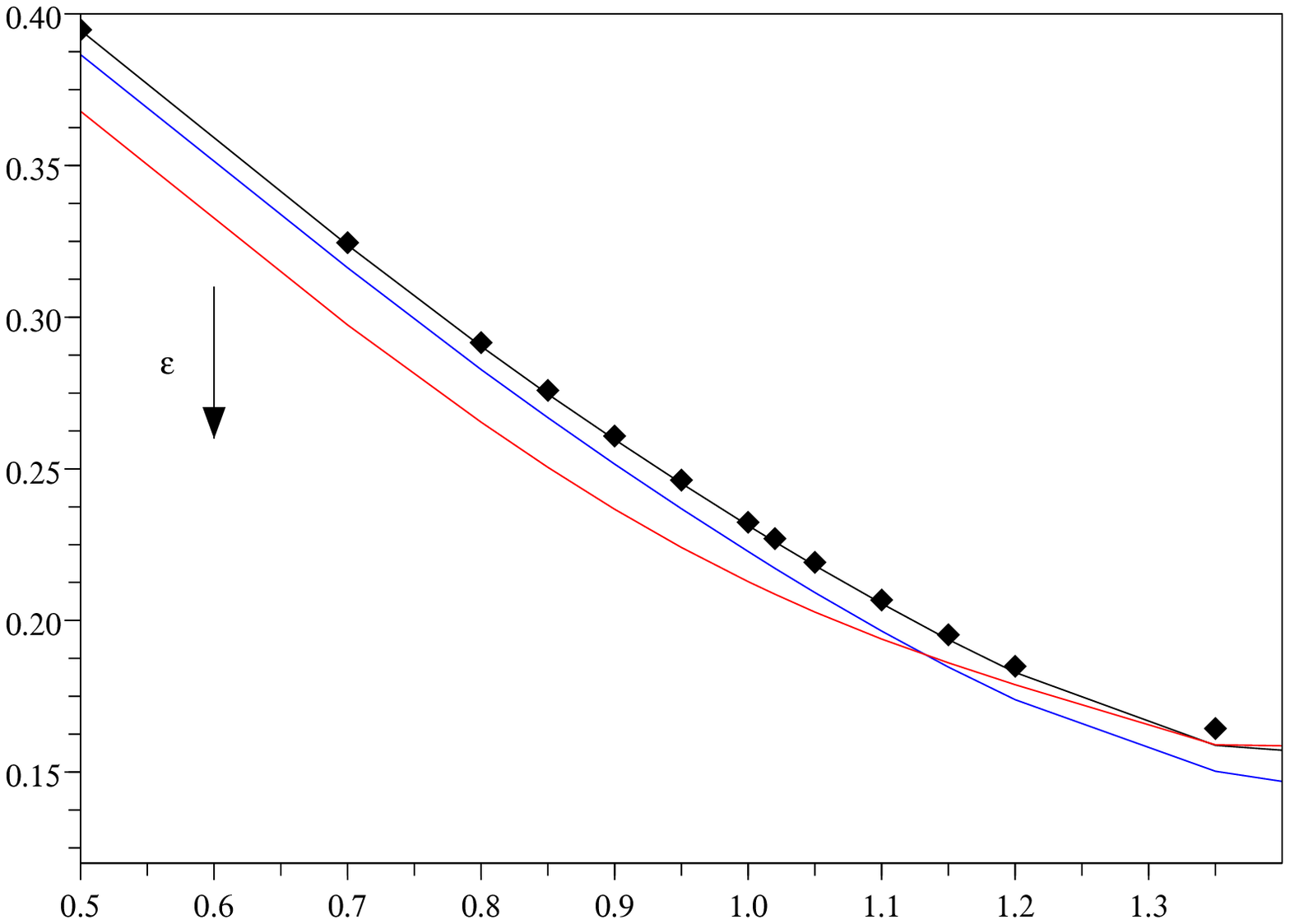,width=8cm,height=7cm}\\
  \end{tabular}  
  \caption{\small{{\bf Left:} diamonds indicate market $1$~year implied volatility for the
      DAX. The curve fitting this data is the implied volatility given by
      the local correlation model~\eqref{local_corr} with $\gamma=8.672568,$
      $\eta=0.500714$ and $\rho_{\min}= 0.1.$ The other curve is the implied
      volatility given by a constant correlation model $(C_t)_{i,j}=\indi{i=j}+\rho\indi{i \not =j}$ that fits the at
      the money implied volatility: $\rho=0.67968$. {\bf Right:} With the same
      values for $\gamma$, $\eta$ and $\rho_{\min}$, we have plotted for $\kp=100$ and
      $\varepsilon \in \left \lbrace 0, 0.3, 0.9 \right \rbrace$ the implied
      volatility given by model~\eqref{SLC}. These curves are obtained with  $10^6$ Monte-Carlo samples and discretization time step $0.025.$}}\label{lc_extension} 
\end{figure}

We want to illustrate now how dynamics such as MRC processes could be used to
extend local correlation models and add a new source of randomness. Namely, we
consider the following dynamics
\begin{equation}\label{SLC} C_t=C_0 +2 \kp \int_0^t (C_s-C(s,I_s))ds +  \int_0^t a(s,I_s) \sum_{n=1}^d\left( \sqrt{C_s-C_s e_d^{n} C_s} dW_se_d^{n}
  +e_d^{n} dW_s^T \sqrt{C_s-C_s e_d^{n} C_s} \right),
\end{equation}
where $\kp>0$, $a:\R_+\times \R_+^* \rightarrow \R_+$ and~$C(t,I)$ is defined by~\eqref{local_corr}.
The limit case $a(t,x)\equiv 0$, $\kp=+\infty$ corresponds to the local correlation model.
Following the same lines as in the proof of Theorem~\ref{thm_we}, we could show
under some rather mild assumptions on~$\rho(t,x)$ and $\sigma^i(t,x)$ that the whole
SDE on $(C_t,S^1_t,\dots,S^d_t)$ has a weak solution if $2\kp(1-\rho(t,x))\ge
(d-2)a^2(t,x)$ for all $t\ge 0, x>0$, and a unique strong solution if
$2\kp(1-\rho(t,x))\ge d a^2(t,x)$. To simulate such an SDE, we will simply use
the Euler-Maruyama scheme for the stocks and use our scheme for the MRC
process with coefficients $\kp$, $C(t_i^N,I_{t_i^N})$ and
$a(t_i^N,I_{t_i^N})$ on the time-step $[t^N_i,t^N_{i+1}]$. More precisely, we
will assume moreover that $2\kp(1-\rho(t,x))\ge (d-1) a^2(t,x)$ and even set
$$a(t,x)=\sqrt{\frac{2\kp \varepsilon}{d-1}(1-\rho(t,x))},$$
where $\varepsilon \in [0,1]$ is a free parameter. This choice allows to use the
$O(d^3)$ discretization scheme~\eqref{fast_second_order_sch} for the MRC process.
 Starting from the calibrated local correlation model, we have plotted in
the right-hand side of Figure~\ref{lc_extension} the effect of the volatility
on the index smile. We have chosen a large value for $\kp$ so that the
model~\eqref{SLC} fits the data for $\varepsilon=0$. We see that the volatility
of the correlation tends to reduce prices of call option on the index.  The
same monotonicity was already observed in~Figure~\ref{too_flat_smile}. This
indicates some concavity of the index option prices with respect to the
correlation.

Last, a natural question is to wonder if this is really necessary to sample
 a whole correlation process. For example, we could consider the following one-dimensional model
 \begin{equation}\label{SLC2} C_t=\indi{i=j}+\rho_t \indi{i \not =j}, \text{
     with } \rho_t=\rho_0+ \kp \int_0^t (\rho_s-\rho(s,I_s))ds+ a
   \int_0^t \sqrt{\rho_s(1-\rho_s)}dW_s,
\end{equation}
with $\rho_0\in[0,1]$, $\kp >0$, $a \ge 0$. This dynamics would have rather close
qualitative features to~\eqref{SLC} and is much less demanding in terms of
computational effort. To be fair, as far as index modeling is concerned it may
be sufficient to parametrize the correlation matrix by a single
parameter~$\rho_t$.  The heuristic reason is that index options do not really
depend on the individual pairwise correlations $(C_t)_{i,j}$ but rather depend
on an average correlation in the basket. Instead, if the aim is then to price
and hedge exotic products on the dependence, it may be relevant to model all
the pairwise correlations. To give a caricatural example, an option on the
difference of two correlation swaps that pays $(\frac{1}{T}\int_0^T
(C_t)_{i,j}dt-\frac{1}{T}\int_0^T (C_t)_{k,l}dt)^+$ is almost surely equal to
zero in model~\eqref{SLC2} or~\eqref{local_corr}, which may basically give an
arbitrage. It has instead a non trivial price if $(C_t)_{t\ge 0} \sim
MRC_{d}(C_0,\kp,c,a)$ or in model~\eqref{SLC}.

We have tested MRC type dynamics on index options because they are the only
liquid quoted options that bring on dependence. The tractability offered by
these processes is not really exploited for such options. Unfortunately today,
there is no quoted options that could give the market view on pairwise
correlations. However, if an investor has some personal  views on correlations
between some companies or some industry sectors, processes such as MRC can be
a relevant tool to take into account these views and price exotic
products. Generally speaking, modelling precisely the dependence between the
stocks in order to get a model that prices consistently single-name and basket
products is an important challenge in finance, and we hope that processes such
as MRC may be tool to achieve it.


\bibliographystyle{plain}
\bibliography{Biblio2}

\begin{thebibliography}{10}

\bibitem{these}
A.~Ahdida.
\newblock {\em Processus matriciels : simulation et mod\'elisation de la
  d\'ependance en finance}.
\newblock PhD thesis, Universit\'e Paris-Est, 2011.

\bibitem{AA}
A.~Ahdida and A.~Alfonsi.
\newblock Exact and high order discretization schemes for {W}ishart processes
  and their affine extensions.
\newblock {\em Arxiv Preprint}, 2010.

\bibitem{Alexander}
C.~Alexander and L.~Nogueira.
\newblock Stochastic local volatility.
\newblock Technical report, 2004.

\bibitem{Alfonsi}
A.~Alfonsi.
\newblock High order discretization schemes for the {CIR} process: application
  to affine term structure and {H}eston models.
\newblock {\em Math. Comp.}, 79(269):209--237, 2010.

\bibitem{Bru}
M.F. Bru.
\newblock Wishart processes.
\newblock {\em J. Theoret. Probab.}, 4(4):725--751, 1991.

\bibitem{CM}
P.~Carr and A.~Madan.
\newblock Option pricing and the fast fourier transform.
\newblock {\em Journal of Computational Finance}, 2(4):61--73, 1999.

\bibitem{ChenStroock}
L.~Chen and D.~W. Stroock.
\newblock The fundamental solution to the {W}right-{F}isher equation.
\newblock {\em SIAM J. Math. Anal.}, 42(2):539--567, 2010.

\bibitem{Teichmann}
C.~Cuchiero, D.~Filipovi\'c, E.~Mayerhofer, and J.~Teichmann.
\newblock {Affine processes on positive semidefinite matrices.}
\newblock {\em Ann. Appl. Probab.}, 21(2):397--463, 2011.

\bibitem{Dafonseca}
J.~Da~Fonseca, M.~Grasselli, and C.~Tebaldi.
\newblock Option pricing when correlations are stochastic: an analytical
  framework.
\newblock {\em Review of Derivatives Research}, 10:151--180, 2008.

\bibitem{EpsteinMazzeo}
C.~L. Epstein and R.~Mazzeo.
\newblock Wright-{F}isher diffusion in one dimension.
\newblock {\em SIAM J. Math. Anal.}, 42(2):568--608, 2010.

\bibitem{Etheridge}
A.~Etheridge.
\newblock {\em Some mathematical models from population genetics}, volume 2012
  of {\em Lecture Notes in Mathematics}.
\newblock Springer, Heidelberg, 2011.
\newblock Lectures from the 39th Probability Summer School held in Saint-Flour,
  2009.

\bibitem{Fernholz_Karatzas}
R.~Fernholz and I.~Karatzas.
\newblock Relative arbitrage in volatility-stabilized markets.
\newblock {\em Annals of Finance}, 1:149--177, 2005.
\newblock 10.1007/s10436-004-0011-6.

\bibitem{Golub}
G.H. Golub and C.F. Van~Loan.
\newblock {\em Matrix computations}.
\newblock Johns Hopkins Studies in the Mathematical Sciences. Johns Hopkins
  University Press, Baltimore, MD, third edition, 1996.

\bibitem{Gourieroux2}
C.~Gourieroux and J.~Jasiak.
\newblock Multivariate {J}acobi process with application to smooth transitions.
\newblock {\em J. Econometrics}, 131(1-2):475--505, 2006.

\bibitem{Gourieroux}
C.~Gourieroux and R.~Sufana.
\newblock Wishart quadratic term structure models.
\newblock {\em Working paper.}, 2003.

\bibitem{Heston}
Steven~L. Heston.
\newblock A closed-form solution for options with stochastic volatility with
  applications to bond and currency options.
\newblock {\em The Review of Financial Studies}, 6(2):pp. 327--343, 1993.

\bibitem{Yor4}
M.~Jeanblanc, M.~Yor, and M.~Chesney.
\newblock {\em Mathematical methods for financial markets}.
\newblock Springer Finance. Springer-Verlag London Ltd., London, 2009.

\bibitem{Jourdain}
B.~Jourdain and M.~Sbai.
\newblock {Coupling Index and Stocks}.
\newblock {\em Quantitative Finance}, 2010.

\bibitem{Karatzas}
I.~Karatzas and S.E. Shreve.
\newblock {\em Brownian motion and stochastic calculus}, volume 113 of {\em
  Graduate Texts in Mathematics}.
\newblock Springer-Verlag, New York, second edition, 1991.

\bibitem{KT}
S.~Karlin and H.~M. Taylor.
\newblock {\em A second course in stochastic processes}.
\newblock Academic Press Inc. [Harcourt Brace Jovanovich Publishers], New York,
  1981.

\bibitem{Kaya}
C.~Kaya~Boortz.
\newblock Modelling correlation risk.
\newblock {\em Diplomarbeit, preprint}, 2008.

\bibitem{Langnau}
A.~Langnau.
\newblock A dynamic model for correlation.
\newblock {\em Risk}, (April):74--78, 2010.

\bibitem{Mazet}
O.~Mazet.
\newblock Classification des semi-groupes de diffusion sur {${\bf R}$}
  associ\'es \`a une famille de polyn\^omes orthogonaux.
\newblock In {\em S\'eminaire de {P}robabilit\'es, {XXXI}}, volume 1655 of {\em
  Lecture Notes in Math.}, pages 40--53. Springer, Berlin, 1997.

\bibitem{NV}
S.~Ninomiya and N.~Victoir.
\newblock Weak approximation of stochastic differential equations and
  application to derivative pricing.
\newblock {\em Appl. Math. Finance}, 15(1-2):107--121, 2008.

\bibitem{Reghai}
A.~Reghai.
\newblock Breaking correlation breaks.
\newblock {\em Risk}, (October):90--95, 2010.

\bibitem{Rydberg}
T.~H. Rydberg.
\newblock A note on the existence of unique equivalent martingale measures in a
  markovian setting.
\newblock {\em Finance and Stochastics}, 1:251--257, 1997.
\newblock 10.1007/s007800050024.

\bibitem{Strang}
G.~Strang.
\newblock On the construction and comparison of difference schemes.
\newblock {\em SIAM J. Numer. Anal.}, 5:506--517, 1968.

\bibitem{Talay}
D.~Talay and L.~Tubaro.
\newblock Expansion of the global error for numerical schemes solving
  stochastic differential equations.
\newblock {\em Stochastic Anal. Appl.}, 8(4):483--509 (1991), 1990.

\bibitem{Yor5}
M.~Yor.
\newblock {\em Exponential functionals of {B}rownian motion and related
  processes}.
\newblock Springer Finance. Springer-Verlag, Berlin, 2001.

\end{thebibliography}

\pagebreak

\appendix

\section{Some results on correlation matrices}\label{App_Notations_matrices}

\subsection{Linear ODEs on correlation matrices}

Let $b \in \symm$ and $\kp \in \genm$. In this section, we consider the
following linear ODE
\begin{equation}\label{ODE_general}
x'(t)=b-(\kp x(t)+x(t) \kp ^T), \ x(0)=x \in \corr,
\end{equation}
and we are interested in necessary and
sufficient conditions on~$\kp$ and $b$ such
that \begin{equation}\label{ODE_corr_stable}
  \forall x \in \corr, \forall t \ge 0, x(t) \in \corr.
\end{equation}
Let us first look at necessary conditions. We have for
$1 \le i,j\le d$:
$$ x'_{i,j}(t)=b_{i,j} -  \sum_{k=1}^d
\kp_{i,k}x_{k,j}(t)+x_{i,k}(t)\kp_{j,k}.$$
In particular, we necessarily have $x'_{i,i}(t)=0$. This gives for $t=0$,
$l\not= i$ and $x(0)=I_d+\rho(e^{i,l}_d+e^{l,i}_d)$ that $ b_{i,i}-2\kp_{i,i}- 2\rho \kp_{i,l} =0$ for any $ \rho \in
[-1,1]$. It comes out that:
$$ \kp_{i,l}=0 \text{ if } l \not = i, \ b_{i,i}=2\kp_{i,i}.$$
Thus, the matrix $\kp$ is diagonal and we denote $\kp_i=\kp_{i,i}$. We get
$x'_{i,j}(t)=b_{i,j} -  (\kp_i+\kp_j)x_{i,j}(t)$ for $i\not = j$. If
$\kp_i+\kp_j=0$, we have $x_{i,j}(t)=x_{i,j}+b_{i,j}t$, which implies that
$b_{i,j}=0$. Otherwise, $\kp_i+\kp_j \not =0$ and we get:
$$ x_{i,j}(t)=x_{i,j} \exp \left(-(\kp_i+\kp_j) t \right)+ \frac{
  b_{i,j}}{\kp_i+\kp_j} \left[1-\exp \left(-(\kp_i+\kp_j) t \right)
\right] .$$ Once again, this implies that $\kp_i+\kp_j>0$ since the initial
value $x\in \corr$ is arbitrary. We set for $1\le i,j\le d$,
\begin{equation}\label{ODE_intermed}c_{i,i}=1, \text{ and for } i \not = j, \ c_{i,j}=
\begin{cases}
  \frac{b_{i,j}}{\kp_i+\kp_j} \text{ if } \kp_i+\kp_j>0 \\ 0 \text{ if } \kp_i+\kp_j=0.
\end{cases}
\end{equation}
We have $b=\kp c + c \kp$ and for $x=I_d$, $c=\lim_{t\rightarrow +\infty}
x(t) \in \corr$, and deduce the following result.
\begin{proposition}\label{CN_ODE_correl} Let $b \in\symm$ and $\kp \in
  \genm$. If the linear ODE~\eqref{ODE_general} satisfies~\eqref{ODE_corr_stable}, then we have necessarily:
\begin{equation}\label{CN_ODE}\exists c \in \corr, \exists \kp_1,\dots,\kp_d \in \R, \forall i\not = j,
\kp_i+\kp_j \ge 0, \kp=diag(\kp_1,\dots,\kp_d) \text{ and } b=\kp c+c\kp.
\end{equation}
\end{proposition}
Conversely, let us assume that~\eqref{CN_ODE} holds and $b\in \posm$. We get that $\kp_i
=b_{i,i}/2\ge 0$ and for $t \ge 0$,  $\exp(\kp t) x(t)\exp(\kp t) = x + \int_0^t \exp(\kp
s) b\exp(\kp s) ds$ is clearly positive
semidefinite. Therefore,~\eqref{ODE_corr_stable} holds. We get the following
result.

\begin{proposition}\label{CS_ODE_correl} Let $ \kp_1,\dots,\kp_d  \ge 0$, $\kp=diag(\kp_1,\dots,\kp_d)$ and
  $c \in \corr$. If $\kp c+c\kp \in \posm$ or $d=2$,  the ODE
\begin{equation}\label{ODE_corr}
  x'(t)=\kp(c-x)+(c-x)\kp, \ x(0)=x \in \corr
\end{equation}
satisfies~\eqref{ODE_corr_stable}.
\end{proposition}
Let us note here that the parametrization of the ODE~\eqref{ODE_corr} is redundant when $d=2$, and we can
assume without loss of generality that $\kp_1=\kp_2$ for which $\kp c+c\kp \in
\posm$ is clearly satisfied.

\begin{remark}\label{rem_CNS} The condition given by Proposition~\ref{CN_ODE_correl} is
  necessary but not sufficient, and the condition given by Proposition~\ref{CS_ODE_correl} is
  sufficient but not necessary.  Let $d=3$ and $c=I_3$. We can check that for $\kp=(1,\frac{1}{2},-\frac{1}{2})$,
  \eqref{CN_ODE} holds but \eqref{ODE_corr_stable} is not true. Also,  we can
  check that for $\kp=(1,1,-\frac{1}{2})$, \eqref{ODE_corr_stable} holds. 
\end{remark}

\begin{lemma}\label{lemm_somme_ODE} Let $\kp^1,\kp^2 $ be diagonal matrices
  and $c^1,c^2 \in \corr$ such that $\kp^1 c^1+c^1 \kp^1+\kp^2 c^2+c^2 \kp^2
  \in \posm$. Then, the ODE
  $$x'=\kp^1(c^1-x)+(c^1-x)\kp^1+\kp^2(c^2-x)+(c^2-x)\kp^2$$
  satisfies~\eqref{ODE_corr_stable}. Besides, $x'=\kp(c-x)+(c-x)\kp$ with
  $\kp=\kp^1+\kp^2 \in \posm $ and $c\in \corr$ defined by:
  $$c_{i,i}=1, \text{ and for } i \not = j, \ c_{i,j}=
\begin{cases}
  \frac{(\kp^1_i+\kp^1_j)c^1_{i,j}+ (\kp^2_i+\kp^2_j)c^2_{i,j}}{\kp_i+\kp_j} \text{ if } \kp_i+\kp_j>0 \\ 0 \text{ if } \kp_i+\kp_j=0.
\end{cases}
$$
\end{lemma}
\begin{proof} Since $b=\kp^1 c^1+c^1 \kp^1+\kp^2 c^2+c^2 \kp^2 \in
  \posm$,~\eqref{ODE_corr_stable} holds for $x'=b-\kp x +x \kp$. Then, we know by~\eqref{ODE_intermed} that
  $c$ is a correlation matrix.  
\end{proof}

\subsection{Some algebraic results on correlation matrices}

\begin{lemma}\label{lemma_correlmatrix}
Let $c\in \corr$ and $1 \le i \le d$. Then we have: $c-ce_d^ic \in \posm$,
$(c-ce_d^ic)_{i,j}=0$ for $1\le j\le d$, $\Subm{\left({c-ce_d^ic}\right)}{i}={\Subm{c}{i} - c^i(c^i)^T} $ and:
$$\Subm{\left(\sqrt{c-ce_d^ic}\right)}{i}=\sqrt{\Subm{c}{i} - c^i(c^i)^T}
\text{ and }\left(\sqrt{c-ce_d^ic}\right)_{i,j}=0.$$
Besides,
if $c \in \corri$, $\Subm{c}{i} - c^i(c^i)^T \in \dposs{d-1}$. 
\end{lemma}
\begin{proof}
Up to a permutation, it is sufficient to prove the result for $i=1$. We have
$$c-ce_d^1c = \left( \begin{array}{cc} 0&0_{d-1}^T\\0_{d-1}&\Subm{c}{1} -
    c^1(c^1)^T \end{array} \right)=a c a^T, \text{ with } a=\left(\begin{array}{cc}
    0 &0_{d-1}\\ -c^1&
    I_{d-1} \end{array} \right)\in \posm.  $$
Besides, we have $\Rg(aca^T)=\Rg(a \sqrt{c})=d-1$ when $c \in \corri$, which
gives $\Subm{c}{i} - c^i(c^i)^T \in \dposs{d-1}$.
\end{proof}

\begin{lemma}\label{lemma_correlmatrix2}
Let $c\in \corr$ and $1 \le n \le d$. Then $I_d-\sqrt{c}e_d^n\sqrt{c}\in
\posm$ and is such that $$\sqrt{I_d-\sqrt{c}e_d^n\sqrt{c}}=I_d-\sqrt{c}e_d^n\sqrt{c}.$$
\end{lemma}
\begin{proof}
The matrix $(\sqrt{c}e_d^n\sqrt{c})_{i,j}=(\sqrt{c})_{i,n}(\sqrt{c})_{j,n}$ is
of rank~$1$ and
$\sum_{j=1}^d(\sqrt{c}e_d^n\sqrt{c})_{i,j}(\sqrt{c})_{j,n}=(\sqrt{c})_{i,n}$
since $\sum_{j=1}^d(\sqrt{c})_{j,n}^2= c_{j,j}=1$. Therefore
$((\sqrt{c})_{i,n})_{1 \le i\le d}$ is an eigenvector, and the eigenvalues of
$I_d-\sqrt{c}e_d^n\sqrt{c}$ are $0$ and $1$ (with multiplicity $d-1$).
\end{proof}

\begin{lemma}\label{OuterProdDec}
Let  $q \in \mathcal{S}^+_{d}(\mathbb R)$ be a matrix with rank $r$. Then
there is a permutation matrix $p$, an invertible lower triangular matrix~$m_r
\in \mathcal{G}_r(\mathbb R)$ and $k_r\in
\mathcal{M}_{d-r\times r}(\mathbb R)$ such that:
$$p q p^T = m m^T,\,\,
m =  \left(\begin{array}{cc}
m_r & 0\\
k_r & 0 \\
\end{array}\right)
.$$
The triplet $(m_r,k_r,p)$ is called an extended Cholesky decomposition of
$q$.
\end{lemma}
The proof of this result and a numerical procedure to get such a decomposition
can be found in Golub and Van Loan~(\cite{Golub}, Algorithm 4.2.4). When $r=d$, we can take $p=I_{d}$, and $m_r$ is the usual
Cholesky decomposition.

\begin{lemma}\label{Lemma_Decompo}
Let $c \in \corr$, $r=\Rg((c_{i,j})_{2\leq i,j \leq d})$ and 
$(m_r,k_r,\tilde{p})$ an extended Cholesky decomposition of~$(c_{i,j})_{2\leq
  i,j \leq d}$. We set $p= \left( \begin{array}{c c}
1 & 0\\
0 & \tilde{p}^T \\
\end{array}\right)$, $m=\left(\begin{array}{c|c c}
1& 0 & 0 \\
\hline
0 & m_r & 0\\
0 & k_r & 0 \\
\end{array}\right)$ and $\check{c}=\left(\begin{array}{c|c c}
1 & (m^{-1}_r c_1^r)^T & 0 \\
\hline
m^{-1}_r c_1^r &I_r & 0\\
0 &0 & I_{d-r-1} \\
\end{array}\right)$, where $c^r_1 \in \R^r$, with $(c^r_1)_i=(p^Tcp)_{1,i+1}$ for $1\le i\le
r$.
 We have:
 $$ c=p m \check{c} m^T p^T \text{ and } \check{c} \in \corr.$$
\end{lemma}
\begin{proof}
By straightforward block-matrix calculations, on has to check that the vector
$c_1^{r,d}\in \R^{d-(r+1)}$ defined by $(c_1^{r,d})_i=(p^Tcp)_{1,i}$ for
$r+1\le i \le d$ is equal to $k_rm_r^{-1}c_1^r$. To get this,  
we introduce the matrix  $q=\left( \begin{array}{c|cc}
1&0&0\\
\hline
0&m_r & 0\\
0&k_r & I_{d-r-1} \\
\end{array}\right)$ and have  $q^{-1}=\left( \begin{array}{c|cc}
1&0&0\\
\hline
0&m^{-1}_r & 0\\
0&-k_rm^{-1}_r & I_{d-r-1} \\
\end{array}\right)$. Since the matrix
$$ q^{-1}p^Tcp(q^{-1})^T=  \left(\begin{array}{c|c c}
1 & (m^{-1}_r c_1^r)^T & (c_1^{r,d}-k_r m_r^{-1} c_1^r)^T \\
\hline
m^{-1}_r c_1^r&I_r & 0\\
c_1^{r,d}-k_r m_r^{-1} c_1^r&0 & 0 \\
\end{array}\right)
$$
is positive semidefinite, we  have $c_1^{r,d}=k_r m_r^{-1}
c_1^r$, $ \left(\begin{array}{c c}
1 & (m^{-1}_r c_1^r)^T  \\
m^{-1}_r c_1^r&I_r
\end{array}
\right) \in \mathcal{S}_{r+1}^+(\R)$ and thus $\check{c} \in \corr$.
\end{proof}

\section{Some auxiliary results}
\subsection{Calculation of quadratic variations}

\begin{lemma}\label{calcul_crochet}
Let $(\mathcal{F}_t)_{t \ge 0}$ denote the filtration generated by
$(W_t,t\ge0)$. We consider a process $(Y_t)_{t \geq 0}$ valued in $\symm$ such
that
\begin{equation*}
 dY_t = B_t dt + \sum_{n=1}^d (A_t^n dW_t e_d^{n}  +  e_d^{n}dW_t^T (A_t^n)^T), 
\end{equation*}
where  $(A^n_t)_{t \geq 0}$, $(B_t)_{t\ge 0}$ are continuous
$(\mathcal{F}_t)$-adapted processes respectively valued in $\genm$,  and
$\symm$. Then, we have for $1\le i, j,k ,l \le d$:
\begin{equation}
  d\langle Y_{i,j},Y_{k,l} \rangle_t=\left[\indi{i=k}(A^i_t(A^i_t)^T)_{j,l} +
    \indi{i=l}(A^i_t(A^i_t)^T)_{j,k}+ \indi{j=k}(A^j_t(A^j_t)^T)_{i,l} +
    \indi{j=l}(A^j_t(A^j_t)^T)_{i,k}\right]dt \label{crochet_semimg}
\end{equation}
\end{lemma}

\begin{proof}
  Since $(A^n_tdW_te^n_d)_{i,j}=\indi{j=n}(A^j_t
  dW_t)_{i,j}$ and $(e^n_d dW_t^T(A^n_t)^T)_{i,j}=\indi{i=n}(A^i_t
  dW_t)_{j,i}$, we get:
$$d(Y_t)_{i,j}=(B_t)_{i,j}dt +  \sum_{n=1}^d (A^j_t)_{i,n} (dW_t)_{n,j} +
(A^i_t)_{j,n} (dW_t)_{n,i}.$$
Then, $d\langle Y_{i,j},Y_{k,l} \rangle_t=\left[ \indi{j=l} \sum_{n=1}^d
  (A^j_t)_{i,n}(A^j_t)_{k,n} + \indi{j=k} \sum_{n=1}^d
  (A^j_t)_{i,n}(A^j_t)_{l,n} +  \indi{i=l} \sum_{n=1}^d
  (A^i_t)_{j,n}(A^i_t)_{k,n}\right.$ \newline
$\left.+\indi{i=k} \sum_{n=1}^d
  (A^i_t)_{j,n}(A^i_t)_{l,n}\right]dt $, which precisely gives~\eqref{crochet_semimg}.
\end{proof}

\begin{lemma}\label{lemma_det_sde}
Let us consider  $x \in \corri$, and $(X_t)_{t \geq 0}$ a solution of the
SDE $\eqref{SDE_CORR}$. Let $\tau$ denote the stopping time defined as $\tau=
\lbrace t \geq 0, X_t \not \in \corri \rbrace$. Then, there exists a real Brownian motion
$(\beta_t)_{t \geq 0}$ such that for $0\le t < \tau$,
\begin{eqnarray}\label{SDE_DET}
\frac{d(\det(X_t))}{\det(X_t)}&=& \Tr[X_t^{-1}  ( \kp c+c \kp-(d-2)a^2) ]
  dt-\Tr(2\kp+a^2)dt +2 \sqrt{\Tr\left[a^2(X_t^{-1}-I_d)\right]} d\beta_t, \\
d \log( \det(X_t)) &=&\Tr[X_t^{-1}  ( \kp c+c \kp-d a^2) ]
  dt-\Tr(2\kp-a^2)dt +2 \sqrt{\Tr\left[a^2(X_t^{-1}-I_d)\right]} d\beta_t. \label{SDE_LOGDET}
  \end{eqnarray}
\end{lemma}
\begin{proof}
 First, let us recall that $\forall i,j,k,l
\in \interv{1}{d},\,\,\forall x\in \dpos\,\,$
$\partial_{i,j}\det(x)=(\adj(x))_{i,j}=\det(x)x^{-1}_{i,j},$ $ \partial_{k,l}\partial_{i,j}(\det(x))=\det(x)(x^{-1}_{l,k}x^{-1}_{i,j}-x^{-1}_{l,j}x^{-1}_{i,k})$.
Since $x$ is symmetric, we have in particular that
$\partial_{k,l}\partial_{i,j}(\det(x))=0$ if $i=l$ or $j=k$. It\^o's
Formula gives for $t<\tau$:
\begin{eqnarray*}
\frac{d(\det(X_t))}{\det(X_t)}&=& \sum_{1\leq i,j\leq d} (X_t^{-1})_{i,j}
d(X_t)_{i,j}
+\frac{1}{2}\sum_{\substack{1\leq i,j \leq d\\ 1 \leq k,l\leq
      d}} \left((X_t^{-1})_{i,j}(X_t^{-1})_{k,l}-(X_t^{-1})_{i,k}(X_t^{-1})_{j,l} 
\right) \langle
  d(X_t)_{i,j},d(X_t)_{k,l} \rangle. 
\end{eqnarray*}
On the one hand we have
\begin{eqnarray*}
\sum_{1\leq i,j\leq d} (X_t^{-1})_{i,j}
d(X_t)_{i,j} &=& \Tr[X_t^{-1}  (\kp c+c \kp) ]
  dt-\Tr(2 \kp)dt+2\sum_{i=1}^d a_i \Tr\left[X_t^{-1}e_d^{i}
      dW_s^T\sqrt{X_t-X_te_d^i X_t}\right].\\
\end{eqnarray*}
On the other hand we get by~\eqref{crochet_MRC}:
\begin{eqnarray*}
&& \sum_{\substack{1\leq i,j \leq d\\ 1 \leq k,l\leq
      d}} \left((X_t^{-1})_{i,j}(X_t^{-1})_{k,l}-(X_t^{-1})_{i,k}(X_t^{-1})_{j,l} 
\right) \langle
  d(X_t)_{i,j},d(X_t)_{k,l} \rangle
 \\ &=&  \sum_{\substack{1\leq i,j\leq d\\ 1\leq k,l\leq
    d}}\left((X_t^{-1})_{i,j}(X_t^{-1})_{k,l}-(X_t^{-1})_{i,k}(X_t^{-1})_{j,l}
\right) \times\left \lbrace a_{j}^2\indi{j=k}(X_t-X_te_d^{j}X_t)_{i,l}
\right. \\ 
&&\,\,\,\,\,\,\,+a_{j}^2\indi{j=l}(X_t-X_te_d^{j}X_t)_{i,k}
+a_{i}^2\indi{i=l}(X_t-X_te_d^{i}X_t)_{j,k} \left.+a_{i}^2\indi{i=k}(X_t-X_te_d^{i}X_t)_{j,l}\right \rbrace \\
&=&\sum_{j=1}^d\left(\sum_{1 \leq i,k \leq d}  a_{j}^2(X_t-X_te_d^{j}X_t)_{i,k}\left( (X_t^{-1})_{i,j}(X_t^{-1})_{k,j}-(X_t^{-1})_{i,k}(X_t^{-1})_{j,j}\right)\right)\\
&&+\sum_{i=1}^d\left(\sum_{1 \leq j,l \leq d}  a_{i}^2(X_t-X_te_d^{i}X_t)_{j,l}\left( (X_t^{-1})_{i,j}(X_t^{-1})_{i,l}-(X_t^{-1})_{i,i}(X_t^{-1})_{j,l}\right)\right)\\
&=& 2\sum_{i=1}^d a_{i}^2\left( \Tr\left[ (X_t-X_te_d^{i}X_t)X_t^{-1}e_d^{i}X_t^{-1}\right] -(X_t^{-1})_{i,i}\Tr\left[(X_t-X_te_d^{i}X_t)X_t^{-1}\right] \right).
\end{eqnarray*}
Since $X_t \in \corri$, we obtain that
 $\Tr\left[(X_t-X_te_d^{i}X_t)X_t^{-1}e_d^{i}X_t^{-1}\right]=(X_t^{-1})_{i,i}-1$
 and $\Tr\left[X_t^{-1}(X_t-X_te_d^{i}X_t)\right]=d-(X_t)_{i,i}=d-1$. We
 finally get:
\begin{equation}\label{dyn_det}
\frac{d(\det(X_t))}{\det(X_t)}= \Tr[X_t^{-1}  ( \kp c+c \kp-(d-2)a^2) ]
  dt-\Tr(2\kp+a^2)dt +2 \sum_{i=1}^d a_i \Tr\left[X_t^{-1}e_d^{i}
      dW_s^T\sqrt{X_t-X_te_d^i X_t}\right].
  \end{equation}
Now, we compute the quadratic variation of~$\det(X_t)$ by using~\eqref{crochet_MRC}:
\begin{eqnarray*}
\frac{d\langle \det(X) \rangle_t}{\det(X_t)^2}  &=&\sum_{\substack{1\leq i,j\leq d\\ 1\leq k,l\leq
    d}}(X_t^{-1})_{i,j}(X_t^{-1})_{k,l}\left \lbrace
  a_{j}^2\indi{j=k}(X_t-X_te_d^{j}X_t)_{i,l}
  +a_{j}^2\indi{j=l}(X_t-X_te_d^{j}X_t)_{i,k} \right. \\ && \ \left.
+a_{i}^2\indi{i=l}(X_t-X_te_d^{i}X_t)_{j,k}+a_{i}^2 \indi{i=k}(X_t-X_te_d^{i}X_t)_{j,l}\right \rbrace dt\\
&=& 4\sum_{i=1}^da_{i}^2\Tr\left[X_t^{-1}e_d^{i}X_t^{-1}(X_t-X_te_d^{i}X_t)
\right]dt \\
&=& 4\sum_{i=1}^da_{i}^2 ((X_t^{-1})_{i,i}-1)dt= 4[\Tr(a^2X_t^{-1})-\Tr(a^2)]dt.
\end{eqnarray*}
It is indeed nonnegative: we can show by diagonalizing and using the convexity of $x\mapsto 1/x$
that $x^{-1}_{i,i}\ge 1/x_{i,i}=1$. Then, there is a Brownian
motion~$(\beta_t,t\ge 0)$ such that~\eqref{SDE_DET} holds (see Theorem
3.4.2 in~\cite{Karatzas}).
\end{proof}

\begin{proposition}\label{moments_jacobi}
Let $k,\theta,\eta \ge 0$. For a given $x \in [-1.1],$ let us consider a
process $(X_t^x)_{ t \geq 0},$ starting from $x,$  and defined as the solution of the following SDE
\begin{equation}
 dX_t^x = k(\theta-X_t^x)dt + \eta \sqrt{1-(X_t^x)^2}dB_t,
\end{equation}
where $(B_t)_{t \geq 0}$ is a real Brownian motion. Then there exists a positive constant $K>0,$ such that 
$$\forall t \geq 0, \forall x \in [-1,1], \,\, \E\left[ (X_t^x-x)^4 \right]\leq K t^2 $$
\end{proposition}
\begin{proof}
 For a given $x \in [-1,1],$ we set $f^x(y)=(y-x)^4$. If we denote $L$
 the infinitesimal operator of the process $X_t^x,$ then we notice that
 $f^x(x) =Lf^x(x)=0$. Besides, $(x,y) \in [-1,1]^2\mapsto L^2f^x(y)$
is continuous and therefore bounded:  \vspace{-3mm}\begin{equation}\label{ineq_L_2} \exists K>0, \forall x, y \in[-1,1], \, |L^2f^x(y)| \leq 2K.\end{equation}

Since the process $(X_t^x)_{t \geq 0}$ is defined on $[-1,1],$ we get by
applying twice It\^o's formula: 
\begin{equation*}
 \E\left[f^x(X_t^x) \right]  =  \int_0^t \int_0^s \E\left[ L^2f^x(X_u^x) \right]duds.
\end{equation*}
From $\eqref{ineq_L_2},$ one can deduce that ${\Big{|}}\int_0^t \int_0^s
\E\left[ L^2f^x(X_u^x) \right]duds{\Big{|}} \leq K t^2,$ and obtain the final result.

\end{proof}

\subsection{Some basic results on squared Bessel processes}

\begin{lemma}\label{lemma_change_time}
 Let $\beta \geq  2$ and  $Z_t = z+ \beta t + 2 \int_0^t \sqrt{Z_s} dB_s$ be a squared Bessel process of
 dimension $\beta$ starting from $z> 0$. Then we have 
\begin{eqnarray*} 
 \mathbb P(\forall t \geq 0, \int_0^t \frac{ds}{Z_s} < \infty ) =1 & \text{and}&   \int_0^{+\infty} \frac{ds}{Z_s } = + \infty \,\,a.s.
\end{eqnarray*}
\end{lemma}
\begin{proof}
 The first claim is obvious, since the square Bessel process does never touch
 zero under the condition of $\beta \geq 2.$ (see for instance~\cite{Yor4},
 part 6.1.3). By using a comparison theorem ($\forall t \ge 0, Z_t \le Z'_t$ a.s.
 if $\beta \le \beta'$), it is sufficient to prove the second claim for $\beta
 \in \N$. In this case, it is well known that  $(W^1_t+\sqrt{z})^2
 +\sum_{k=2}^n (W^k_t)^2 $ follows a square Bessel process of dimension~$n$,
 where $(W^k_t, t\ge 0)$ are independent Brownian motion. By the law of the
 iterated logarithm, $\limsup_{t\rightarrow + \infty} \frac{(W^k_t)^2}{2t
   \log(\log(t))}=1$, which gives the desired result since $\int_1^\infty \frac{dt}{t
   \log(\log(t))}=+\infty$.
\end{proof}

\begin{lemma}\label{lemma_change_time2}
Let $\beta \geq  6$.  Let $Z_t = 1+ \beta t + 2 \int_0^t \sqrt{Z_s} dB_s$ be a squared Bessel process of
 dimension $\beta$ starting from $1$ and $\phi(t)=\int_0^t \frac{1}{Z_s}ds$. Then we have 
$$\E[\phi(t)]=t+\frac{4-\beta}{2}t^2+O(t^3),\ \E[\phi(t)^2]=t^2+O(t^3), \ \E[\phi(t)^3]=O(t^3).$$
\end{lemma}
\begin{proof} For a fixed time $t > 0$, the density of $Z_t$ is given by:

\vspace{-3mm}\begin{equation*}z>0, p(t,z)=\sum_{k=0}^{+\infty}
  \frac{e^{-\frac{1}{2t}}(\frac{1}{2t})^{k}}{k!}
  \frac{1}{2t\Gamma(k+\frac{\beta}{2})}
  (\frac{z}{2t})^{k-1+\frac{\beta}{2}}e^{-\frac{z}{2t}}.\end{equation*}

Let us consider that $\gamma \in \left \lbrace  1,2,3 \right \rbrace, $ then all negative moments can be written as 
\begin{eqnarray*}
 \E\left[ \frac{1}{Z_t^{\gamma}}\right]=   \sum_{k=0}^{+\infty} \
 \frac{e^{-\frac{1}{2t}}(\frac{1}{2t})^{k+\gamma}}{k!}
 \frac{\Gamma(k+\frac{\beta}{2}-\gamma)}{\Gamma(k+\frac{\beta}{2})}=
 \sum_{k=0}^{+\infty} \ \frac{e^{-\frac{1}{2t}}(\frac{1}{2t})^{k+\gamma}}{k!}
 \frac{1}{(k+\frac{\beta}{2}-1)\times \dots \times(k+\frac{\beta}{2}-\gamma)}.
\end{eqnarray*}

We have $\frac{1}{(k+\frac{\beta}{2}-1)} = \frac{1}{k+1} -
\frac{\beta-4}{2(k+2)(k+1)} + O(\frac{1}{k^3})$, which yields to the following expansion: 
\begin{eqnarray}\label{equ1}
 \E\left[ \frac{1}{Z_t}\right] &=& \sum_{k=0}^{+\infty}
 \frac{e^{-\frac{1}{2t}}(\frac{1}{2t})^{k+1}}{(k +1)!}  -(\beta-4)  t
 \sum_{k=0}^{+\infty}  \frac{e^{-\frac{1}{2t}}(\frac{1}{2t})^{k+2}}{(k +2)!}
 +O \left(\frac{t^2}{2} \sum_{k=0}^{+\infty}  \frac{e^{-\frac{1}{2t}}(\frac{1}{2t})^{k+3}}{(k +3)!} \right)  \nonumber \\
&=& 1 - (\beta-4)t + O(t^2)
\end{eqnarray}
The first equality is thus obtained. We use the same argument to get: 
\vspace{-2mm}\begin{eqnarray}\label{equ2}
\E\left[ \frac{1}{Z_t^2}\right]&=& \sum_{k=0}^{+\infty}  \frac{e^{-\frac{1}{2t}}(\frac{1}{2t})^{k+2}}{(k +2)!} + O\left(t \sum_{k=0}^{+\infty}  \frac{e^{-\frac{1}{2t}}(\frac{1}{2t})^{k+3}}{(k +3)!}\right)= 1  + O(t) \nonumber \\
\E\left[ \frac{1}{Z_t^3}\right]&=&  O\left(\sum_{k=0}^{+\infty}  \frac{e^{-\frac{1}{2t}}(\frac{1}{2t})^{k+3}}{(k +3)!} \right)  = O(1). 
 \end{eqnarray}
 By Jensen's inequality, one can deduce that $\E\left[ \left(\int_0^t  \frac{ds}{Z_s} \right)^3\right] \leq t^2 \E\left[\int_0^t \frac{ds}{(Z_s)^3}\right].$ Thanks to the moment expansion  in $\eqref{equ2},$ we find the third equality. Finally, by Jensen's equality, we obtain that 
\vspace{-3mm}\begin{eqnarray*}\E\left[ \left(\int_0^t  \left[\frac{1}{Z_s}-1\right]ds \right)^2\right] \leq t \E\left[\int_0^t \left( \frac{1}{Z_s}-1\right)^2 ds\right]&=& t\E\left[\int_0^t \frac{ds}{(Z_s)^2} \right]-2t\E\left[\int_0^t \frac{ds}{(Z_s)} \right]  + t^2 \\
 &= & t^2  -2t^2 + t^2 + O(t^3)=O(t^3).
\end{eqnarray*}
It yields that 
\vspace{-3mm}\begin{eqnarray*}\E\left[ \left(\int_0^t  \left[\frac{1}{Z_s}\right]ds \right)^2\right] &=& \E\left[ \left(\int_0^t  \left[\frac{1}{Z_s}-1\right]ds \right)^2\right] -t^2 + 2t \int_0^t \E\left[ \frac{1}{Z_s}\right]ds =t^2 + O(t^3).
\end{eqnarray*}

\end{proof}
\section{A direct proof of Theorem \ref{theorem_spliopper}}\label{proof_theorem_spliopper}
\begin{proof}
From~\eqref{EQUATION_OPERATOR_C} we have $2L_i=-\alpha L_i^D+L_i^M$, with:
$$ L^D_i= \sum_{\substack{1 \le j \le d \\ j \not = i}}
x_{\set{i,j}} \partial_{\{i,j\}}, \ \ L_i^M=\sum_{\substack{1 \le
  j,k \le d \\ j \not = i, k \not = i}} 
(x_{\{j,k\}}-x_{\{i,j\}}x_{\{i,k\}}) \partial_{\{i,j\}}\partial_{\{i,k\}}. $$
We want to show that $L_iL_j=L_jL_i$ for $i \not = j$. Up to a permutation of
the coordinates, $L_i$ and $L_j$ are the same operators as $L_1$ and $L_2$. It
is therefore sufficient to check that $L_1L_2=L_2L_1$. Since $L_1L_2=
L_1^ML_2^M -\alpha(L_1^DL_2^M+L_1^ML_2^D)+\alpha^2 L_1^DL_2^D$, it is
sufficient to check that the three terms remain unchanged when we exchange
indices~$1$ and~$2$. To do so we write:
\begin{eqnarray*}
\begin{split}
L_1^M &= {\sum_{\substack{3\leq i,j \leq d}}
  (\ivec{x}{i}{j}-\ivec{x}{1}{i}\ivec{x}{1}{j})\ipartial{1}{i}\ipartial{1}{j}+2\sum_{\substack{3
      \leq i \leq d
    }}(\ivec{x}{2}{i}-\ivec{x}{1}{2}\ivec{x}{1}{i})\ipartial{1}{2}\ipartial{1}{i}+(1-\ivec{x}{1}{2}^2)\ipartial{1}{2}^2}\\
L_2^M &= {\sum_{\substack{3\leq k,l \leq d }} (\ivec{x}{k}{l}-\ivec{x}{2}{k}\ivec{x}{2}{l})\ipartial{2}{k}\ipartial{2}{l}+2\sum_{\substack{3 \leq l \leq d }}(\ivec{x}{1}{l}-\ivec{x}{1}{2}\ivec{x}{2}{l})\ipartial{1}{2}\ipartial{2}{l}+(1-\ivec{x}{1}{2}^2)\ipartial{1}{2}^2}\\
 L_1^D&=\ivec{x}{1}{2} \ipartial{1}{2} + \sum_{\substack{3\leq i\leq
     d}}\ivec{x}{1}{i}\ipartial{1}{i}, \ \ \ \ L_2^D=\ivec{x}{1}{2} \ipartial{1}{2} + \sum_{\substack{3\leq l\leq
     d}}\ivec{x}{2}{l}\ipartial{2}{l}.
\end{split}
\end{eqnarray*}
By a straightforward but tedious calculation, we get :\\
$\displaystyle
L_1^ML_2^M=\iunderset{\overline{1}}{\sum_{\substack{3 \leq i,j,k,l \leq d}}(\ivec{x}{i}{j}-\ivec{x}{1}{i}\ivec{x}{1}{j})(\ivec{x}{k}{l}-\ivec{x}{2}{k}\ivec{x}{2}{l})\ipartial{1}{i}\ipartial{1}{j}\ipartial{2}{k}\ipartial{2}{l}}$\\
$\displaystyle+ \iunderset{\tilde{2}}{\sum_{\substack{3\leq i,j \leq
      d}}(\ivec{x}{i}{j}-\ivec{x}{1}{i}\ivec{x}{1}{j})\left(
    2\ipartial{1}{2}\ipartial{2}{i}\ipartial{1}{j}+2\ipartial{1}{2}\ipartial{2}{j}\ipartial{1}{i}\right)}
$ \\
$\displaystyle + 2\iunderset{3}{\sum_{\substack{3\leq i,j,l \leq d }} (\ivec{x}{i}{j}-\ivec{x}{1}{i}\ivec{x}{1}{j})(\ivec{x}{1}{l}-\ivec{x}{1}{2}\ivec{x}{1}{l})\ipartial{1}{2}\ipartial{2}{l}\ipartial{1}{i}\ipartial{1}{j}}$\\
$\displaystyle +\iunderset{4}{\sum_{\substack{3\leq i,j \leq d
    }}(\ivec{x}{i}{j}-\ivec{x}{1}{i}\ivec{x}{1}{j})(1-\ivec{x}{1}{2}^2)\ipartial{1}{i}\ipartial{1}{j}\ipartial{1}{2}^2}$
\\
$\displaystyle+2\iunderset{3}{\sum_{\substack{3\leq i,k,l \leq
      d}}(\ivec{x}{2}{i}-\ivec{x}{1}{2}\ivec{x}{1}{i})(\ivec{x}{k}{l}-\ivec{x}{2}{k}\ivec{x}{2}{l})\ipartial{2}{k}\ipartial{2}{l}\ipartial{1}{2}\ipartial{1}{i}}$\\
$\displaystyle+4\sum_{\substack{3\leq i \leq d
  }}(\ivec{x}{2}{i}-\ivec{x}{1}{2}\ivec{x}{1}{i})\left(\iunderset{\tilde{5}}{\ipartial{1}{2}^2\ipartial{2}{i}}-\iunderset{\tilde{6}}{\sum_{\substack{3
        \leq l\leq d
      }}\ivec{x}{2}{l}\ipartial{1}{2}\ipartial{2}{l}\ipartial{1}{i}}
  +\iunderset{\overline{7}}{\sum_{\substack{3\leq l \leq
        d}}(\ivec{x}{1}{l}-\ivec{x}{1}{2}\ivec{x}{2}{l})\ipartial{1}{2}^2\ipartial{2}{l}\ipartial{1}{i}}\right)$\\
$\displaystyle +2\sum_{\substack{3 \leq i \leq d
  }}(\ivec{x}{2}{i}-\ivec{x}{1}{2}\ivec{x}{1}{i})\left(
  \iunderset{\tilde{8}}{-2\ivec{x}{1}{2}\ipartial{1}{2}^2\ipartial{1}{i}}+\iunderset{9}{(1-\ivec{x}{1}{2}^2)\ipartial{1}{2}^3\ipartial{1}{i}}\right)$\\
$\displaystyle +\iunderset{4}{\sum_{\substack{3\leq k,l \leq d }}(\ivec{x}{k}{l}-\ivec{x}{2}{k}\ivec{x}{2}{m})(1-\ivec{x}{1}{2}^2)\ipartial{2}{k}\ipartial{2}{m}\ipartial{1}{2}^2}$\\
$\displaystyle + (1-\ivec{x}{1}{2}^2) \left(\sum_{\substack{3 \leq l \leq d}} \iunderset{9}{2(\ivec{x}{1}{l}-\ivec{x}{1}{2}\ivec{x}{2}{l})\ipartial{1}{2}^3\ipartial{2}{l}} -\iunderset{\tilde{10}}{4\sum_{\substack{3 \leq l \leq d}} \ivec{x}{2}{l}\ipartial{1}{2}^2\ipartial{2}{l}}\right)$\\
$\displaystyle + \iunderset{\overline{11}}{(1-\ivec{x}{1}{2}^2)\ipartial{1}{2}^2((1-\ivec{x}{1}{2}^2)\ipartial{1}{2}^2)}$\\
In this formula, the terms~$\bar{n}$ are already symmetric by exchanging~$1$
and~$2$. The terms~$n$ are paired with the corresponding symmetric term. To
analyse the terms~$\tilde{n}$, we have to do further calculations. On the one hand,
\begin{eqnarray}\label{tilde_equ_1}\tilde{10}+\tilde{5}
&=& \sum_{\substack{3\leq l \leq d}}  4\ivec{x}{1}{2}(\ivec{x}{1}{2}\ivec{x}{2}{l}-\ivec{x}{1}{l}) \ipartial{1}{2}^2\ipartial{2}{l}\nonumber \\
\tilde{8}
&=& \sum_{\substack{3 \leq l\leq d
  }}4\ivec{x}{1}{2}(\ivec{x}{1}{2}\ivec{x}{1}{l}-\ivec{x}{2}{l})\ipartial{1}{2}^2\ipartial{1}{l},\nonumber\end{eqnarray}
are symmetric together. On the other hand we have
\begin{eqnarray}\label{tilde_equ_2}
\tilde{2}+\tilde{6}= \sum_{\substack{1\leq i,j \leq d\\i\neq 1,2, j  \neq 1,2}}\left \lbrace  4\ivec{x}{i}{j}-4\ivec{x}{1}{i}\ivec{x}{1}{j}-4\ivec{x}{2}{i}\ivec{x}{2}{j}+4\ivec{x}{1}{i}\ivec{x}{2}{j}\ivec{x}{1}{2} \right \rbrace \ipartial{1}{2}\ipartial{1}{i}\ipartial{2}{j},\nonumber
\end{eqnarray}
which is symmetric.

Now we focus on $L_1^DL_2^M+L_1^ML_2^D$. We number the terms with the same
rule as above, and get:
\begin{equation*}
\begin{split}
&L_1^DL_2^M+L_1^ML_2^D= \iunderset{1}{\sum_{\substack{3\leq k,l \leq d}} ( \ivec{x}{k}{l}-\ivec{x}{2}{l}\ivec{x}{2}{k})\ivec{x}{1}{2}\ipartial{2}{l}\ipartial{2}{k}\ipartial{1}{2}} \\
&+2\iunderset{2}{\sum_{\substack{3\leq l \leq d
    }}\ivec{x}{1}{2}(\ivec{x}{1}{l}-\ivec{x}{1}{2}\ivec{x}{2}{l})\ipartial{1}{2}^2\ipartial{2}{l}}-2\iunderset{{3}}{\sum_{\substack{3\leq
      l\leq d }}
  \ivec{x}{1}{2}\ivec{x}{2}{l}\ipartial{2}{l}\ipartial{1}{2}}+\iunderset{\overline{4}}{\ivec{x}{1}{2}\ipartial{1}{2}\left
    \lbrace  (1-\ivec{x}{1}{2}^2)\ipartial{1}{2}^2\right \rbrace}\\
& +\iunderset{5}{\sum_{\substack{3\leq i,k,l \leq d}}\ivec{x}{1}{i}(\ivec{x}{l}{k}-\ivec{x}{2}{k}\ivec{x}{2}{l})\ipartial{2}{k}\ipartial{2}{l}\ipartial{1}{i}}+\iunderset{6}{2\sum_{\substack{3\leq i,l \leq d }} \ivec{x}{1}{i}(\ivec{x}{1}{l}-\ivec{x}{1}{2}\ivec{x}{2}{l})\ipartial{1}{2}\ipartial{2}{l}\ipartial{1}{i}}\\
&+2\iunderset{{7}}{\sum_{\substack{3 \leq i \leq d }} \ivec{x}{1}{i}\ipartial{1}{2}\ipartial{2}{i}}+\iunderset{8}{\sum_{\substack{3 \leq i \leq d}}\ivec{x}{1}{i}(1-\ivec{x}{1}{2}^2)\ipartial{1}{2}^2\ipartial{1}{i}}+\iunderset{1}{\sum_{\substack{3\leq i,j  \leq d}} (\ivec{x}{i}{j}-\ivec{x}{1}{i}\ivec{x}{1}{j})\ivec{x}{1}{2}\ipartial{1}{2}\ipartial{1}{i}\ipartial{1}{j}}\\
&+\iunderset{5}{\sum_{\substack{3 \leq i,j,l \leq d }}\ivec{x}{2}{l}(\ivec{x}{i}{j}-\ivec{x}{1}{i}\ivec{x}{1}{j})\ipartial{1}{i}\ipartial{1}{j}\ipartial{2}{l}}+2\iunderset{2}{\sum_{\substack{3 \leq i \leq d}}\ivec{x}{1}{2}(\ivec{x}{2}{i}-\ivec{x}{1}{2}\ivec{x}{1}{i})\ipartial{1}{2}^2\ipartial{1}{i}}\\
&+2 \sum_{\substack{3\leq i \leq d }}(\iunderset{7}{\ivec{x}{2}{i}}-\iunderset{3}{\ivec{x}{1}{i}\ivec{x}{1}{2}})\ipartial{1}{i}\ipartial{1}{2}+2\iunderset{6}{\sum_{\substack{ 3 \leq i,l \leq d }}\ivec{x}{2}{l}(\ivec{x}{2}{i}-\ivec{x}{1}{2}\ivec{x}{1}{i})\ipartial{1}{i}\ipartial{1}{2}\ipartial{2}{l}}\\
&+\iunderset{\overline{9}}{(1-\ivec{x}{1}{2}^2)\ipartial{1}{2}^2\lbrace \ivec{x}{1}{2} \ipartial{1}{2}\rbrace} + \iunderset{8}{\sum_{\substack{3\leq l \leq d }} (1-\ivec{x}{1}{2}^2)\ivec{x}{2}{l}\ipartial{1}{2}^2\ipartial{2}{l}}.
\end{split}
\end{equation*}
Therefore, $L_1^DL_2^M+L_1^ML_2^D$ is symmetric when we exchange $1$ and
$2$. Last, it is easy to check that $L_1^DL_2^D= L_2^DL_1^D$, which concludes
the proof.
\end{proof}

\section{A direct construction of a second order scheme for MRC processes}\label{schema_original}
In Section~\ref{Sec_simu}, we have presented a second order scheme for Mean-Reverting
Correlation processes that is obtained from a second order scheme for Wishart
processes. In this section, we propose a second order scheme that is
constructed directly by a splitting of the generator of Mean-Reverting
Correlation processes. As pointed in~\eqref{second_order_sch}, it is
sufficient to construct a potential second
order scheme for $ \corrpst{d}{{x}}{\frac{d-2}{2}e_d^1}{I_d}{e_d^1}$. Thanks
to the transformation given by Proposition~\ref{prop_reduction}, it is even sufficient to construct such a
scheme when $(x)_{2 \leq i,j \leq d}=I_{d-1}$.

Consequently, in the rest of this section, we focus on getting a potential second
order scheme for $ \corrpst{d}{{x}}{\frac{d-2}{2}e_d^1}{I_d}{e_d^1},$  where
$(x)_{2 \leq i,j \leq d}=I_{d-1}$. By~\eqref{cns_corr}, the matrix $x$ is a
correlation matrix if  $\sum_{i=2}^dx_{1,i}^2 \leq 1$. Besides, the only non
constant elements are on the first row (or the first column) and the vector
$((X_t)_{1,i})_{2 \ldots d}$ is thus defined on the unit ball $\D$: \begin{equation}
\D = \left \lbrace x \in \R^{d-1},\,\, \sum_{i=1}^{d-1}x_{i}^2 \leq 1 \right \rbrace.                                                              \end{equation}

With a slight abuse of notation, the process $((X_t)_{i})_{ 1 \ldots   d-1  }$
will denote the vector $((X_t)_{1,i+1})_{1\ldots d-1}.$ Its quadratic
covariance is given by $d\langle (X_t)_i, d(X_t)_j\rangle =\left(
  \indi{i=j}-(X_t)_i (X_t)_j\right)dt$, and the  infinitesimal generator $L^1$
of~$\corrps{d}{{x}}{\frac{d-2}{2}e_d^1}{I_d}{e_d^1}$ can be rewritten on $\D,$ as 
\begin{eqnarray}\label{L_1_Basic}
L^1 = -\frac{d-2}{2} \sum_{i=1}^{d-1}x_i\partial_{i} + \frac{1}{2} \sum_{1 \leq i,j \leq d-1}(\indi{i=j}-x_ix_j)\partial_{i}\partial_{j}.
             \end{eqnarray}
One can prove that the following stochastic differential equation  \begin{equation*}\label{SDE_BASIC}
 \forall 1 \leq i \leq d-1,\,\,dM_t^i = -\frac{d-2}{2} M_t^i + M_t^i\sqrt{1 - \sum_{j=1}^{d-1} (M_t^j)^2}dB_t^1 + (1-(M_t^i)^2)dB_t^{i+1} - M_t^i\sum_{\substack{1 \leq j \leq {d-1}\\ j \neq i }}M_t^jdB_t^{j+1}
\end{equation*}
 is associated to the martingale problem of $L^1$, where $(B_t)_{t \geq 0}$
 denotes a standard Brownian motion in dimension~$d$. By
 Theorem~\eqref{thm_we}, there is a unique weak solution $(M_t)_{t \geq 0}$ that is  defined on $\D.$ 

 The scope of this section is to
derive a potential second order discretization for the operator $L^1,$ by
using an ad-hoc splitting and the results of
Proposition~\ref{prop_compo_schemas}.  We consider the following splitting
\begin{equation}\label{splitting_equation}L^1= {\opr}^1+\sum_{m=1}^{d-1}\opr^{m+1},\end{equation} 
where we have, for $1 \le m \le d-1$:
\begin{eqnarray*}
{\opr}^1\,\, &=&  \frac{1}{2}\left(1-\sum_{i=1}^{d-1} x_i^2\right)\sum_{\substack{1 \leq
    l,k \leq d-1}} x_k x_l \partial_k\partial_l, \\
\opr^{m+1}&=& \frac{1}{2}\left(-\sum_{1 \leq k \neq m \leq {d-1}
  }x_k\partial_k+ (1- x_m^2)^2\partial_m^2- 2x_m(1-x_m^2) \sum_{1 \leq k \neq
    m \leq {d-1} } x_k \partial_k\partial_m + \sum_{{\small \substack{1 \leq
        k \neq m \leq {d-1} \\ 1 \leq l \neq m \leq
        {d-1}}}}x_kx_l x_m^2\partial_k\partial_l  \right).
\end{eqnarray*}
Thanks to Proposition~\ref{prop_compo_schemas}, it is sufficient to focus on getting
potential second-order schemes for the operators~$\opr^1,\dots,\opr^d$.



\subsection{Potential second order schemes for  $\opr^2,\ldots, \opr^d$} \label{SubsL_2}
All the generators $\opr^{l+1}$, $l=1, \dots, d-1$ have the same solution as
$\opr^2$ up to the permutation of the first coordinate and the $l$-th one. It
is then sufficient to focus on the first operator $\opr^2.$ By straightforward
calculus, we find that the following SDE
\begin{eqnarray}\label{SDE_opL2}
              d(X_t)_1= (1-(X_t)_{1}^2)dB_t&,&\,\, \forall 2 \leq i \leq
              {d-1}\,\, d(X_t)_{i}=-(X_t)_{i}\left(\frac{dt}{2} +
                (X_t)_{1}dB_t\right), \ X_0=x\in \D, 
             \end{eqnarray}
 is well a solution of the martingale problem for the generator $\opr^2$. The
 SDE that defines~$(X_t)_1$ is autonomous. Since $x_1 \in [-1,1]$, it has
 clearly a  unique strong valued in $[-1,1]$. It yields that the SDE  $\eqref{SDE_opL2}$ has
 a unique strong solution on $\R^{d}.$ To prove that $(X_t)_{t \geq 0}$ 
 takes values in $\D$ we consider $V_t = \sum_{i=1}^d (X_t)_i^2$. By It\^o calculus, it follows that 
\begin{equation*}dV_t = (1-V_t)(1-(X_t)_1^2)dt
  +2(X_t)_1(1-V_t)dB_t.\end{equation*}
Thus, $1-V_t$ can be written as a stochastic exponential starting from
$1-V_0\ge 0$ and is therefore nonnegative. We now introduce the
Ninomiya-Victoir scheme for the SDE~\eqref{SDE_opL2}.

\begin{proposition}\label{prop_L_2}
Let us consider $x\in \mathbb D$. Let $Y$ be sampled according to $\mathbb
P(Y=\sqrt{3})=\mathbb  P(Y=-\sqrt{3})=\frac{1}{6}$, so that it fits the first five moments of a standard Gaussian variable.  Then $\hat{X}_t^x=X^0(\frac{t}{2},X^1(\sqrt{t}Y,X^0(\frac{t}{2},x)))$ is well defined on $\mathbb D$ and is a potential second order scheme for the infinitesimal operator $\opr^2$, where:
\begin{eqnarray*} 
\forall t \geq 0, \forall x \in \mathbb D,&& X^0_1(t,x)=\frac{x_1{
    e^t}}{\sqrt{e^{2t}x_1^2+(1-x_1^2)}},\,\, \forall 2\leq l \leq
d-1,\,\,X^0_l(t,x) =  \frac{x_l}{\sqrt{e^{2t}x_1^2+(1-x_1^2)}},\\
\forall y \in \R , \forall x \in \mathbb D,&&  X^1_1(y,x)=\frac{e^{2y}(1+x_1)-(1-x_1)}{e^{2y}(1+x_1)+(1-x_1)},\,\, \forall 2\leq l \leq d-1,\,\,X^1_l(y,x) =  \frac{2e^yx_l}{e^{2y}(1+x_1)+(1-x_1)}.
\end{eqnarray*}
\end{proposition}
\begin{proof}
 The proof is a direct application of the Ninomiya-Victoir's scheme~\cite{NV}
 and we introduce the following ODEs:
\begin{equation*}
\begin{array}{ll}
  \partial_t X^0_1(t,x)= X^0_1(t,x)(1-(X^0_1(t,x))^2),&\,\, \forall 2 \leq l \leq d-1,\,\partial_t X^0_l(t,x)=-X^0_l(t,x)(X^0_1(t,x))^2\\
\partial_y   X^1_1(y,x) = (1-(X^1_1(y,x))^2),&\,\, \forall 2 \leq l \leq d-1,\, \partial_yX^1_l(y,x) = -X^1_1(y,x)X^1_l(y,x).
\end{array}
\end{equation*}
These ODEs can be solved explicitly as stated above. We have to check that
they are well defined on $\mathbb D$. This can be checked with the explicit
formulas or by observing that  $\partial_t (\sum_{l=1}^{d-1}(X^0_l(t,x))^2)=2(X^0_1(t,x))^2(1-\sum_{l=1}^{d-1}(X^0_l(t,x))^2),\,\,\partial_t (\sum_{l=1}^{d-1}(X^1_l(t,x))^2)=2X^1_1(t,x)(1-\sum_{l=1}^{d-1}(X^1_l(t,x))^2).$
Last, Theorem 1.18 in Alfonsi~\cite{Alfonsi} ensures that $\hat{X}_t^x$ is a
potential second order scheme for~$\opr^2$.\end{proof}

\subsection{Potential second order scheme for  $\opr^1$}\label{SubsL_1}

Let   $(B_t)_{t \geq 0}$ be a real a Brownian motion. We consider the following SDE:
\begin{equation}\label{SDE_opL1}
\forall 1 \leq i \leq {d-1},\,\,
d(X_t)_i=(X_t)_i\sqrt{1-\sum_{m=1}^{d-1}(X_t)_m^2}dB_t, \ X_0=x \in \D 
\end{equation}  
Its infinitesimal generator is~$\opr^1$, and we claim that it has a unique
strong solution. To check this, we set $Z_t = \sqrt{\sum_{i=1}^{d-1}
  (X_t)_i^2}.$ By It\^o calculus, we get that the process $(Z_t)_{t \geq 0}$ is solution of the following SDE
\begin{equation}\label{SDE_LV}
 dZ_t =  Z_t\sqrt{1-Z_t^2}dB_t,\,\, Z_0= \sqrt{\sum_{i=1}^{d-1}x_i^2}.  \end{equation} 
Since the SDE $\eqref{SDE_LV}$ satisfies the Yamada-Watanabe conditions
(Proposition 2.13, Chapter~5 of~\cite{Karatzas}),   it has a unique strong
solution defined on $[0,1]$. If $Z_0=0$, we necessarily have $Z_t=0$ and thus
$(X_t)_i=0$ for any $t\ge 0$. Otherwise, we have by Itô calculus
$d\ln((X_t)_i)=d \ln(Z_t)$, and then
\begin{equation}\forall 1 \leq i \leq d-1,\,\,
  (X_t)_{i}=\label{TransfX}\begin{cases} 0, \text{ if } Z_0=0 \\
     \frac{x_i}{Z_0}Z_t \text{ otherwise.}\end{cases}\end{equation}
Conversely, we check easily that~\eqref{TransfX} is a strong solution
of~\eqref{SDE_LV}, which proves our claim. The explicit
solution~\eqref{TransfX} indicates that the SDE~\eqref{SDE_LV} is
one-dimensional up to a basic transformation. Thanks to the next proposition,
it is sufficient to construct a potential second order scheme for~$Z_t$ in
order to get a potential second order scheme for~\eqref{SDE_LV}.

\begin{proposition}
Let us consider $x \in \D$, and  $\hat{Z}_t^z$ denote the  second potential order scheme  for $(Z_t)_{t \geq 0},$ starting from a given value $z \in [0,1].$ Then the following scheme  $\hat{X}_t^x$ 
\begin{equation*}
\forall 1 \leq i \leq d-1,\,\, (\hat{X}_t^x)_{i}= \begin{cases} 
 0 &\text{ if } \sum_{j=1}^{d-1}x_j^2=0,  \\
\frac{x_i}{\sqrt{\sum_{j=1}^{d-1}x_j^2}}\hat{Z}_t^{\sqrt{\sum_{j=1}^{d-1}x_j^2}}
&\text{ otherwise,} \end{cases}
\end{equation*}
is a second potential order scheme for $\opr^1$ which is well defined on $\D.$
\end{proposition}
\begin{proof}
 For a given $x \in \D$ and $f \in \mathcal{C}^{\infty}(\D),$ let $(X_t^x)_{t \geq 0}$ denote a process defined by $\eqref{TransfX}$ and starting from $x\in \D.$ It is  sufficient to prove that \begin{equation*}{\Big{|}}\E\left[f(X_t^x)\right]-\E\left[f(\hat{X}_t^x)\right]{\Big{|}}\leq Kt^3.\end{equation*}
The case where $x=0$ is trivial, and we assume thus that
$\sum_{i=1}^{d-1}x_i^2 > 0.$ Let $f \in \mathcal{C}^{\infty}(\D)$. We define
$g^x: [0,1] \rightarrow \R$ by $ \forall y \in [0,1],\,\,
g^x(y)=f(\frac{x_1}{\sqrt{\sum_{j=1}^{d-1}x_j^2}}
y,\ldots,\frac{x_i}{\sqrt{\sum_{j=1}^{d-1}x_j^2}}  y).$ Since for every $1
\leq i \leq
d-1,\,\,{\big{|}}\frac{x_i}{\sqrt{\sum_{j=1}^{d-1}x_j^2}}{\big{|}}\leq 1 ,$
it  follows we can construct from a good sequence of~$f$ a good sequence
for~$g^x$ that does not depend on~$x$.  By the defintion of the second
potential scheme, there exist  positive constants $K>0$ and $\eta>0,$
depending  only on a good sequence of~$f$ such that $\forall t \in [0,\eta]$
\begin{equation*}
{\Big{|}}\E\left[g^x(Z_t^{\sqrt{\sum_{j=1}^{d-1}x_j^2}})\right]-\E\left[g^x(\hat{Z}_t^{\sqrt{\sum_{j=1}^{d-1}x_j^2}})\right]{\Big{|}}\leq Kt^3 ,
             \end{equation*}
which gives the desired result.
\end{proof}

We now focus on finding a potential second order scheme for $(Z_t)_{t \geq
  0}$. To do so, we try the Ninomiya-Victoir's scheme~\cite{NV} and consider the following ODEs
for $z\in[0,1]$,
\begin{equation*}
\forall t \geq 0,\,\partial_t
 Z_0(t,z)=Z_0(t,z)(Z_0(t,z)-\frac{1}{2}), \ \ \forall x\in \R,\, \partial_xZ_1(x,z)= Z_1(x,z)\sqrt{1-Z_1(x,z)^2}.
\end{equation*}
These ODEs can be solved explicitly. On the one hand, it follows that for every $t \geq 0$ and $z \in [0,1]$
\begin{equation*} Z_0(t,z)=\frac{z\exp(-t/2)}{\sqrt{1-2z^2(1-\exp(-t))}}.
\end{equation*}
On the other hand, we get by considering the change of variable
$\sqrt{1-Z_1^2}$ that for every  $x \in \R$ and $z \in [0,1]$,
\begin{equation*}
 Z_1(x,z)= \begin{cases}
\frac{2z\exp(-x)}{1-\sqrt{1-z^2}+\exp(-2x)(1+\sqrt{1-z^2})} & \text{ if } x \le  \frac{1}{2} \ln(\frac{1+\sqrt{1-z^2}}{1-\sqrt{1-z^2}}),
 \\
1   & \text{ otherwise.}
 \end{cases}
\end{equation*}
Then, the Ninomiya-Victoir scheme is given
by~$Z_0(t/2,Z_1(\sqrt{t}Y,Z_0(t/2,z)))$, where $Y$ is a random variable that
matches the five first moments of the standard Gaussian
variable. Unfortunately, the composition $Z_0(t/2,Z_1(\sqrt{t}Y,Z_0(t/2,z)))$
may not be defined if $z$ is close to~$1$. To correct this, we proceed like
Alfonsi~\cite{Alfonsi} for the CIR diffusion. First, we consider $Y$ that has
a bounded support so that $Z_0(t/2,Z_1(\sqrt{t}Y,Z_0(t/2,z)))$ is well defined
when $z$ is far enough from $1$ (namely when $0\le z\le K(t)\le 1$ with
$K(t)=1+O(t)$). When the initial value~$z$ is close to~$1$, we instead use a
moment-matching scheme, and then we prove that the whole scheme is potentially
of order~$2$ (Propositions~\ref{prop_NV_Z} and~\ref{prop_NV_Z_App}).

\subsubsection{Ninomiya-Victoir's scheme for $(Z_t)_{t \geq 0}$ away from~$1$}
\begin{proposition}\label{prop_NV_Z}
 Let us consider a discrete random variable $Y$ that follows $\mathbb
 P(Y=\sqrt{3})=\mathbb  P(Y=-\sqrt{3})=\frac{1}{6}$, and $\mathbb
 P(Y=0)=\frac{2}{3}$, so that it matches  the five first moments of a standard Gaussian.
\begin{itemize}
 \item For a given $z \in [0,1],$ the map  $z \mapsto Z_0(t/2,Z_1(\sqrt{t}Y,Z_0(t/2,z))))$ is well defined on $[0,1],$ if and only if $z\in [0,K(t)],$ where  the threshold function $K(t)$ is given in $\eqref{K_tb}.$  
\item For a given function $f \in \mathcal C^{\infty}([0,1])$, there are
  constants $\eta,C>0$ depending only on a good sequence of~$f$ such that $\forall t\in [0,\eta],\,\, \forall z \in [0,K(t)]$,
\begin{equation}\label{O2_away_from_1}
\left| \E\left[Z_0(t/2,Z_1(\sqrt{t}Y,Z_0(t/2,z)))) \right] - \left(f(z)-tL_Zf
    (z)+ \frac{t^2}{t}L_Z^2f(z) \right) \right|\leq Ct^3,
\end{equation}
where $L_Z$ is the infinitesimal operator associated to the SDE $\eqref{SDE_LV}.$
\end{itemize}
For every $t \geq 0$  the function $K(t)$ is valued on $[0,1]$ such that \begin{equation}\label{K_tb}
 K(t)=\sqrt{\frac{1}{2-e^{-t/2}}}\wedge  \frac{\sqrt{1-D(t,\sqrt{3})^2}}{\sqrt{e^{-t/2}+2(1-D(t,\sqrt{3})^2)(1-e^{-t/2})}}, \,\,\, \lim_{t \rightarrow 0 }\frac{1-K(t)}{t} = \frac{{\sqrt{3}}}{2}(1+{\sqrt{3}}),\\
\end{equation}
with $ \forall y \in \R^+\, D(t,y)= \frac{1-e^{-2\sqrt{t}y}+\sqrt{\frac{1-e^{-t/2}}{2-e^{-t/2}}}(1+e^{-2\sqrt{t}y})}{e^{-2\sqrt{t}y}+1+  \sqrt{\frac{1-e^{-t/2}}{2-e^{-t/2}}}(1-e^{-2\sqrt{t}y})}.$

\end{proposition}
\begin{proof}
  The main technical thing here is to check the first point. Then,
  \eqref{O2_away_from_1} is a direct consequence of Theorem 1.18 in Alfonsi~\cite{Alfonsi}.
By construction, we have $Z_0(t/2,z) \in [0,1] \Leftrightarrow z \leq
\frac{1}{\sqrt{2 - \exp{(t/2)}}}.$ We conclude that the whole scheme
$Z_0(t/2,Z_1(\sqrt{t}Y,Z_0(t/2,z)))) $ is well defined on $[0,1],$ if and only
if $Z_1(\sqrt{t}Y,Z_0(t/2,z)))\leq \frac{1}{\sqrt{2 - \exp{(t/2)}}}.$ By
slight abuse of notation, we denote in the following $Z_0(t/2,z)$ by the shorthand~$Z_0$. 
Let us
assume for a while that we have:
\begin{equation}\label{Const1}
\sqrt{1-Z_0^2}(1+e^{-2\sqrt{t}Y}) + e^{-2\sqrt{t}Y}-1 \geq 0 , a.s.
\end{equation}
It yields then to 
\begin{eqnarray} \label{Const2}
 Z_0(t/2,Z_1(\sqrt{t}Y,Z_0(t/2,z))))  \in [0,1]& \Longleftrightarrow &  \sqrt{ 1-\left[Z_1(\sqrt{t}Y,Z_0(t/2,z)))\right]^2 }\geq \sqrt{\frac{1-e^{-t/2}}{2-e^{-t/2}}}   \nonumber \\
& \Longleftrightarrow & \sqrt{\left(\frac{e^{-2\sqrt{t}Y}(1+\sqrt{1-Z_0^2} )-(1-\sqrt{1-Z_0^2}) }{e^{-2\sqrt{t}Y}(1+\sqrt{1-Z_0^2} )+(1-\sqrt{1-Z_0^2}) }\right)^2}\geq \sqrt{\frac{1-e^{-t/2}}{2-e^{-t/2}}} \nonumber\\
& \underset{\text{By}\,\,\eqref{Const1}}{ \implies }& \frac{\sqrt{1-Z_0^2}(1+e^{-2\sqrt{t}Y}) + e^{-2\sqrt{t}Y}-1 }{\sqrt{1-Z_0^2}(e^{-2\sqrt{t}Y}-1) + 1+ e^{-2\sqrt{t}Y}} \geq \sqrt{\frac{1-e^{-t/2}}{2-e^{-t/2}}} \\
&\Longleftrightarrow & \sqrt{1-Z_0^2} \geq \frac{1-e^{-2\sqrt{t}Y}+\sqrt{\frac{1-e^{-t/2}}{2-e^{-t/2}}}(1+e^{-2\sqrt{t}Y})}{e^{-2\sqrt{t}Y}+1+  \sqrt{\frac{1-e^{-t/2}}{2-e^{-t/2}}}(1-e^{-2\sqrt{t}Y})}:=D(t,Y). \nonumber
\end{eqnarray}
We can check that the mapping $D : (t,x)\in \R_+\times \R \mapsto D(t,x) =-1+
\frac{2(1+\sqrt{\frac{1-e^{-t/2}}{2-e^{-t/2}}} )}{e^{-2\sqrt{t}x}+1+
  \sqrt{\frac{1-e^{-t/2}}{2-e^{-t/2}}}(1-e^{-2\sqrt{t}x})}$ is non decreasing
on $x$, and $D(t,x)\leq 1.$ Since $Y \in \left \lbrace -\sqrt{3}, 0, \sqrt{3}
\right \rbrace,$ it yields thus that the last condition is equivalent to:
\begin{equation}\label{Const3}Z_0(t/2,z) \leq
  \sqrt{1-D(t,\sqrt{3})^2}\Leftrightarrow z \leq
  \frac{\sqrt{1-D(t,\sqrt{3})^2}}{\sqrt{e^{t/2}+2(1-e^{-t/2})(1-D(t,\sqrt{3})^2)}}.\end{equation}
Conversely, if~\eqref{Const3} is satisfied, we can check that
$D(t,Y)(1+e^{-2\sqrt{t}Y}) + e^{-2\sqrt{t}Y}-1 \geq
0$. Therefore~\eqref{Const1} holds. To sum up, when $z\in [0,K(t)]$, we both
have  $Z_0(t/2,z), Z_0(t/2,Z_1(\sqrt{t}Y,Z_0(t/2,z))))  \in [0,1]$.

Last, it remains to compute the limit of $(1-K(t))/t$. First, it is obvious
that $\lim_{t \rightarrow 0 }K(t)=1.$  We can check that
$\sqrt{1-D^2(t,\sqrt{3})} =
\frac{2e^{-\sqrt{3}\sqrt{t}}}{\sqrt{2-e^{-t/2}}+\sqrt{1-e^{-t/2}}(1-e^{-\sqrt{3}\sqrt{t}})}=
1+t(\frac{1}{4}-\frac{\sqrt{3}}{2}(1+\sqrt{3})) + o(t)$, and therefore $ 1- \frac{\sqrt{1-D(t,\sqrt{3})^2}}{\sqrt{1+2(1-D(t,\sqrt{3})^2)(1-e^{-t/2})}} = t(\frac{\sqrt{3}}{2}(1+\sqrt{3})) + o(t).$ It yields that $\lim_{t \rightarrow 0 }\frac{1-K(t)}{t}=\frac{{\sqrt{3}}}{2}(1+{\sqrt{3}}) \vee \frac{1}{2}.$

\end{proof}

\subsubsection{Potential second order scheme for $(Z_t)_{t \geq 0}$ in a
  neighbourhood of 1}
Let $(Z_t)_{t \geq 0}$ be solution of the SDE $\eqref{SDE_LV}.$ By Itô
calculus, its moments satisfy the following induction:
$$\forall k \geq 2, \,\E\left[ Z_t^k\right]=\left( z^k - \int_0^t
  \frac{k(k-1)}{2}e^{-\frac{k(k-1)}{2}s}\E\left[
    Z_s^{k+2}\right]ds\right)\exp(\frac{k(k-1)}{2}t).$$ We obtain first that
$\E\left[ Z_t^6\right] = z^6+ O(t)$, then  $\E\left[
  Z_t^4\right]=z^4+6z^4t(1-z^2)+ O(t^2)$ and last   \begin{equation}\label{moment_Z}
 \E\left[ Z_t^2\right] = z^2 + tz^2(1-z^2)+ \frac{t^2}{2}z^2(1-z^2)(1-6z^2) + \mathcal O(t^3).
\end{equation}
Moreover, by straightforward calculus, one can check that if $t \leq \frac{2}{5}$ and for every $z \in [0,1]$\vspace{-3mm}\begin{eqnarray}\label{cond_moment_Z}
 z^2 + tz^2(1-z^2)+ \frac{t^2}{2}z^2(1-z^2)(1-6z^2) \leq 1 &, & tz^2(1-z^2)+ \frac{t^2}{2}z^2(1-z^2)(1-6z^2)\geq 0.
\end{eqnarray}
Since~$\E(Z_t)=z$,  the right hand side  of~\eqref{cond_moment_Z} corresponds
to the asymptotic variance of $Z_t$. To approximate the process $(Z_t)_{t \geq
  0}$ near to one, we use a discrete random variable, denoted by
$\hat{Z}_t^z,$ that fits both the exact first moment and the asymptotic second
given by $\eqref{moment_Z}$. We assume that $\hat{Z}_t^z$
takes two possible values $0 \leq z^+< z^-\leq 1$, with probability $p(t,z)$
and $1-p(t,z)$ respectively. We introduce two positive variables $(m^+, m^-)$,
defined as $z^+ = z+ m^+$ and $z^-=z-m^-$. Since we are looking to match the
moment, we get the following equations: 
\begin{equation}\label{moment_condition}
 \begin{array}{cll}
 \E\left[ \hat{Z}_t^z\right] &=z&\equi m^+p(t,z)=m^-(1-p(t,z))\\
 \E\left[ (\hat{Z}_t^z)^2\right] &=z^2 + tz^2(1-z^2)+ \frac{t^2}{2}z^2(1-z^2)(1-6z^2)&\equi (m^+)^2\frac{p}{1-p}=	tz^2(1-z^2)+ \frac{t^2}{2}z^2(1-z^2)(1-6z^2)\\
 \end{array}
\end{equation}
We choose 
 \begin{equation*}m^+=z(1-z) \text{ and then have }  p(t,z) =1-\frac{1}{1+\frac{t(1+z)(1+\frac{t}{2}(1-6z^2))}{1-z}}, \,\, m^-=tz(1+z)(1+\frac{t}{2}(1-6z^2)).\end{equation*}
 The random variable $\hat{Z}_t^z$ is well defined on $[0,1]$ if and only if
 $z^+ \leq 1$ and $z^- \geq 0,$ which is respectively equivalent to $z(1-z)\leq (1-z)$ and $  t(1+z)(1+\frac{t}{2}(1-6z^2))\leq 1.$
By straightforward calculus, we can check that these conditions are
satisfied. Since $1-K(t)\underset{t\rightarrow 0}{=}O(t)$ by Proposition~\ref{prop_NV_Z}, we deduce that there is $C>0$ such that \begin{equation}\label{moment_asymptotic_value}
 \forall t \in [0,\frac{2}{5}], \forall z \in [K(t),1],\,\, \forall q \in \mathbb N^*,\,\, \E\left[ (1-\hat{Z}_t^z)^q\right] \leq C^q t^q
\end{equation}
\begin{proposition}\label{prop_NV_Z_App}
 Let $U\sim \mathcal U([0,1]).$ The scheme $\hat{Z}_t^z=z^+\indi{ \left \lbrace U\leq p(t,z)\right \rbrace } + z^-\indi{\left \lbrace U > p(t,z) \right \rbrace } $ is a potential second order scheme on $z \in [K(t),1]$: for any function $f \in \mathcal C^{\infty}([0,1])$,  there are positive constants $C$ and $\mu $ that depend  on a good sequence of the function $f,$ such that
\vspace{-3mm}\begin{equation}
 \forall t \in [0,\eta \wedge \frac{2}{5}],\,\,\forall z \in [K(t),1],\,\,{\Big{|}}\E\left[f(\hat{Z}_t^z) \right]-f(z)-tL_Zf(z)-\frac{t^2}{2}(L_Z)^2f(z){\Big{|}}\leq Ct^3,
\end{equation}
where $L_Z$ is the infinitesimal operator associated to the SDE $\eqref{SDE_LV}.$
\end{proposition}
\begin{proof}
 Let us consider a function $f \in \mathcal C^{\infty}([0,1])$. Since the
 exact scheme is a potential second order scheme (see Alfonsi
 ~\cite{Alfonsi}), there exist then two positive constants $\eta$ and $C$,
 such that $\forall t \in [0,\mu ],\,\,\forall z \in
 [0,1],\,\,|\E\left[f(Z_t^z)
 \right]-f(z)-tL_Zf(z)-\frac{t^2}{2}(L_Z)^2f(z)|\leq Ct^3$. We conclude that
 it is sufficient to prove that $\forall z \in [K(t),1],\,\,|\E\left[f(Z_t^z)
 \right]-\E[f(\hat{Z}_t^z) ]|\leq Ct^3$, for a constant positive variable
 $C$. By a third order Taylor expansion of $f$ near to one, we obtain that \begin{equation*}
 \forall z \in [0,1],\,\,{\Big{|}} f(z)- \left(f(1) - f'(1)(1-z)+ \frac{(1-z)^2}{2}f''(1)\right){\Big{|}} \leq \|f^{(3)}\|_{\infty}(1-z)^3.\\
\end{equation*}
Thus, there is a constant $C>0$ depending on a good sequence of~$f$ such that 
\begin {eqnarray*}
 |\E\left[f(Z_t^z) \right]-\E[f(\hat{Z}_t^z) ]| &\leq& C\left(
   {\E[(1-\hat{Z}_t^z)^3 ]} +  {\E[(1-Z_t^z)^3 ]} +{\Big{|}}
   \E\left[(1-Z_t^z)^2 \right]-\E\left[(1-\hat{Z}_t^z)^2 \right]{\Big{|}}
 \right)
\end {eqnarray*}

By $\eqref{moment_asymptotic_value}$, the first term is of order~$O(t^3)$. The
last term is equal to $|\E[({Z}_t^z)^2 ]-
       z^2-tz^2(1-z^2)-\frac{t^2}{2}z^2(1-z^2)(1-6z^2)|$ and is also of
       order~$O(t^3)$ by~\eqref{moment_Z}. Last, we  have by It\^o calculus
       that  $\forall q \geq 2,\,\,  \E[(1-Z_t^z)^q]\leq
       (1-z)^q+q(q-1)\int_0^t \E[(1-Z_s^z)^{q-1}]ds$. By induction, we get
       that there is a  constant $R_q>0,$ such that $\forall z\in
       [K(t),1],\,\, \E[(1-Z_t^z)^q]\leq R_qt^q$, which finally gives the
       claimed result.
\end{proof}

\end{document}